\documentclass[11pt]{amsart}

\usepackage{amsmath, amsthm, amssymb, amsfonts, enumerate}
\usepackage{enumitem}
\usepackage[colorlinks=true,linkcolor=blue,urlcolor=blue]{hyperref}
\usepackage{dsfont}  
\usepackage{color}
\usepackage{geometry}
\usepackage{todonotes}
\usepackage{enumerate}

\newtheorem{theorem}{Theorem}[section]
\newtheorem{remark}[theorem]{Remark}
\newtheorem{assumption}[theorem]{Assumption}
\newtheorem{lemma}[theorem]{Lemma}
\newtheorem{proposition}[theorem]{Proposition}

\newtheorem{definition}{Definition}[]
\newtheorem{example}{Example}[]
\theoremstyle{plain}

\def \R{\mathbb{R}}

\def \C{C}
\usepackage[utf8]{inputenc}
\usepackage[english]{babel}
\usepackage{indentfirst}
\usepackage{graphicx}
\usepackage{pdfpages}
\usepackage[title]{appendix}
\usepackage{fancyhdr}

\usepackage{enumitem}
\usepackage{amsmath}
\usepackage{amssymb}
\usepackage{bm}
\usepackage{mathrsfs}
\usepackage{amsthm}
\usepackage{comment}
\usepackage{hyperref}
\usepackage{dsfont} 
 
\usepackage{mathtools}

\DeclareMathOperator*{\tr}{tr}

\DeclareMathOperator*{\esssup}{ess\,sup}
\DeclareMathOperator*{\lip}{Lip}

\title[ Singular control problem and related Skorokhod problem]{Multidimensional singular control and related Skorokhod problem: sufficient conditions for the characterization of optimal controls}  
\author[Dianetti]{Jodi Dianetti}
\author[Ferrari]{Giorgio Ferrari}
\address{J.~Dianetti: Center for Mathematical Economics (IMW), Bielefeld University, Universit\"atsstrasse 25, 33615, Bielefeld, Germany}
\email{\href{mailto:giorgio.ferrari@uni-bielefeld.de}{jodi.dianetti@uni-bielefeld.de}}
\address{G.~Ferrari: Center for Mathematical Economics (IMW), Bielefeld University, Universit\"atsstrasse 25, 33615, Bielefeld, Germany}
\email{\href{mailto:giorgio.ferrari@uni-bielefeld.de}{giorgio.ferrari@uni-bielefeld.de}}
\date{\today}

\numberwithin{equation}{section}

\begin{document}
\begin{abstract} We characterize the optimal control for a class of singular stochastic control problems
as the unique solution to a related Skorokhod reflection problem. The optimization
problems concern the minimization of a discounted cost over an infinite time-horizon
through a process of bounded variation affecting an It\^o-diffusion. The setting is
multidimensional, the drift of the state equation and the costs are convex, the volatility
matrix can be constant or linear in the state. Our result applies to a relevant class of
linear-quadratic models and it allows to construct the optimal control in degenerate and
non degenerate settings considered in the literature.
\end{abstract} 
\maketitle
  

\section{Introduction} 
This paper considers the problem of characterizing optimal policies for singular stochastic control problems in multidimensional settings. 
More precisely,  we consider the problem of controlling, through a one-dimensional c\`adl\`ag (i.e., right-continuous with left limits) process $v$ with locally bounded variation, the first component of a multidimensional diffusion with initial condition $x$. 
Namely, the controller can affect a state process $X^{x;v}$ which evolves according to the equation   
\begin{equation}\label{eq introduction SDE}
dX_t^{x;v} = b(X_t^{x;v})dt + \sigma(X_t^{x;v}) dW_t + e_1dv_t,\  t\geq 0, \quad X_{0-}^{x;v}=x, 
\end{equation}
for a multidimensional Brownian motion $W$, a suitable convex Lipschitz function $b$, and a volatility matrix $\sigma$, which is either constant or linear in the state. 
The vector $e_1$ denotes the first element of the canonical basis of $\R^d$, for $d\geq1$, and the processes $X^{x;v}$ take values in $\R^d$. 
The aim of the controller is to minimize the expected discounted cost
\begin{equation}\label{eq introduction cost functional}
J(x;v):= \mathbb{E} \bigg[ \int_0^\infty e^{-\rho t} h(X_t^{x;v}) dt + \int_{[0,\infty)} e^{-\rho t} d|v|_{t}  \bigg], 
\end{equation}
for a given convex function $h$ and a suitable discount factor $\rho >0$.
Here, $|v|$ denotes the total variation of the process $v$. 
The value function $V$ of the problem is defined, at any given initial condition $x$, as the minimum of $J(x;v)$ over the choice of controls $v$. 
Also, a control $\bar{v}$ is said to be optimal for $x$ if $J(x;\bar{v})=V(x)$. Existence of optimal controls can be proved in very general frameworks using different probabilistic compactification methods (see, e.g., \cite{budhiraja.ross2006, cohen2021, Haussmann&Suo95, li&zitkovic2017, Menaldi&Taksar89}).
  
Natural questions that immediately arise are whether it is possible to characterize $V$, and how one should act on the system in order to obtain the minimal cost $V$. 
As a matter of fact, the Markovian nature of the problem, together with mild regularity and growth conditions on the problem's data, allows to employ the dynamic-programming approach. This leads to the  characterization of the value function as a solution (in a suitable sense) to the Hamilton-Jacobi-Bellman equation  \begin{equation}\label{eq introduction HJB}
\max \{ \rho V-b DV - \tr (\sigma \sigma^{\text{\tiny {$\top$}}} D^2 V)/2 -h , | V_{x_1}| - 1 \} = 0.  \end{equation}
This equation provides key insights on the way the controller should act on the system in order to minimize the cost functional \eqref{eq introduction cost functional}. 
Indeed, when $V$ is sufficiently regular, an application of It\^o's formula suggests that the controller should make the state process not leaving the set  $\mathcal{W}:=\{ |V_{x_1}| <1 \}$, usually referred to as the waiting region.
In fact, in many examples (see, e.g.,  \cite{angelis&ferrari&moriarty2019, guo&tomecek2009, koch&vargiolu2019, kruk2000, ma1992, soner&shreve1989, yang2014}, among others) it is possible to construct the optimal control as the solution to a related Skorokhod reflection problem; that is, the optimal control can be characterized as that process $\bar{v}$, with minimal total variation, which is able to keep the process $X^{x;\bar{v}}$ inside the closure of the waiting region $\mathcal{W}$, by reflecting it in a direction prescribed by the gradient of the value function.
However, in multidimensional settings, such a characterization often remains a conjecture (see the discussion in Chapter 6 in \cite{suo1994}, Remark 5.2 in \cite{Boetius&Kohlmann98}, and also \cite{chiarolla&haussmann1998, chiarolla&haussmann2000,    federico&ferrari&schuhmann2019, federico&ferrari&schuhumann2020b}), and many questions about the properties of optimal controls remain open, representing a strong limitation to the theory.

\subsection{The characterization problem} 
We now discuss more in detail the  problem of characterizing optimal controls. 
When the state process is one dimensional, optimal controls can be explicitly constructed as Skorokhod reflections in a general class of models in \cite{alvarez2001, davis&zervos1998, jack&johnson&zervos2008, jack&zervos2006,  ma1992, weerasinghe2005}, among others.   
Also, in the (non necessarily Markovian) one dimensional case, a similar characterization of optimal controls has been achieved in \cite{Bank05, bank&elkaroui2004, Bank&Riedel01}, without relying on the dynamic-programming approach.       
When the dimension of the problem becomes larger than one, the difficulty of characterizing optimal controls drastically increases. 
Indeed, classical results on the existence of solutions to the Skorokhod reflection problem in the multidimensional domain $\mathcal{W}$ require some regularity of the boundary of $\mathcal{W}$ and of the direction of reflection, which are, in most of the cases, unknown. 
When the value function $V$ is convex, this difficulty is overcome in some specific settings.
A celebrated example is presented in \cite{soner&shreve1989}, where the problem of controlling a two-dimensional Brownian motion with a two-dimensional process of bounded variation is considered. There, the authors show that the boundary of the waiting region (the so-called free boundary) is of class $C^2$, and they are therefore able to construct the optimal policy as a solution to the associated Skorokhod problem. 
{Similar results are obtained when controlling a multidimensional Brownian motion with a one dimensional control in \cite{shreve.soner.1990elliptic} (in the infinite time-horizon case) and in \cite{soner.shreve.elkaroui.1991parabolic} (for the finite time-horizon).} 
The problem of the characterization is also encountered in  \cite{chiarolla&haussmann1998, chiarolla&haussmann2000,   federico&ferrari&schuhmann2019, federico&ferrari&schuhumann2020b}, where the construction of the optimal control can be provided only by requiring additional properties on the boundary of the waiting region.
{Other examples are exhibited in \cite{davis&zervos1998}, in which the case of controlling a multidimensional Brownian motion with a multidimensional control is considered in the case of a radial running cost $h(x)=|x|^2$, and in \cite{williams.chow.menaldi.1994}, where $h$ is convex and the cost of pushing in different directions is additive in the direction of pushing.}  
We also refer to  \cite{koch&vargiolu2019}, where the construction of the optimal policy is provided in a two-dimensional context in which the drift is non-zero.
To the best of our knowledge, in the case of a convex $V$, the most general multidimensional setting in which this characterization is shown is presented in \cite{kruk2000}, and in its finite time-horizon counterpart   \cite{boryc&kruk2016}.
There, the problem of controlling a multidimensional Brownian motion with a multidimensional control is considered for a convex running cost.
Remarkably, in \cite{kruk2000} (and in \cite{boryc&kruk2016}) the author presents an approach which allows to construct the unique optimal policy as a solution to the related Skorokhod problem bypassing the problems related to the  regularity of the free boundary. 
In non-convex settings, the number of contributions are even rarer. 
The suitable regularity of the boundary of $\mathcal{W}$ is shown, in some particular two-dimensional settings, in \cite{guo&tomecek2009} and in \cite{angelis&ferrari&moriarty2019},  while a multidimensional case is considered in \cite{yang2014}, via a connection with Dynkin games. 
We also mention that the construction of  multidimensional reflected diffusions in polyhedral domains has been recently studied in \cite{Cont.Guo.Xu2020, Guo&Tang&Xu18, Guo&Xu18}, in the context of games with singular controls.
To the best of our knowledge, in the general multidimensional case there is no result on the characterization of the optimal control even in the classical \emph{linear-quadratic} setting (i.e., when $b$ and $\sigma$ are affine and $h$ is quadratic). 
This underlines the lack of theoretical understanding of fundamental properties of the optimally controlled process in the main benchmark models.
To conclude, despite many decades of research in the field, the nature of optimal controls is, in general, far from being completely understood, and this motivates our study. 

\subsection{Our result and methodology.} In this paper, we provide sufficient conditions for the characterization of the optimal policy of the singular control problem specified by \eqref{eq introduction SDE} and \eqref{eq introduction cost functional} as the solution to the related Skorokhod reflection problem. 
Despite in our setting the control is one dimensional, the multidimensional nature of the problem arises from the fact that the components of the state process are interconnected; in particular, the action of the controller on the first component of the state process can affect all the other components. 
We will show the claimed characterization under two main classes of assumptions in which the volatility matrix is constant or linearly dependent on the state.
In both cases, additional monotonicity assumptions are enforced on the running cost $h$ and on the drift $b$. 
These structural conditions are satisfied in a relevant class of linear-quadratic models (see Example \ref{example key} below), and in some specific settings considered in the literature (see \cite{chiarolla&haussmann1998, chiarolla&haussmann2000, federico&ferrari&schuhmann2019, federico&ferrari&schuhumann2020b}), for which the problem of constructing the optimal control remained partially open (see the examples in Subsections \ref{subsection examples} and \ref{subsection example Patrick}).   
Indeed, though in principle these hypotheses narrow the applicability of our result, it is important to underline that similar conditions (and setup) are often in place in the very few multidimensional settings in which an analysis of the free boundary is actually provided (see \cite{chiarolla&haussmann1998, chiarolla&haussmann2000, federico&ferrari&schuhmann2019, federico&ferrari&schuhumann2020b, federico.pham2014}). 

The main novelty of our approach is to investigate the characterization problem via a proper monotonicity of $V_{x_1}$. 
This is done first by identifying $V_{x_1}$ as the value of a related Dynkin game (through a variational formulation in the spirit of \cite{chiarolla&haussmann2000}), and then by exploiting our structural conditions as well as a comparison principle for SDEs. 
The monotonicity of $V_{x_1}$ allows, via a thorough analysis of the signs of the derivatives of $V$, to adapt to our setting some arguments in \cite{kruk2000} in order to 
construct solutions $\bar{v}^\varepsilon$ to a family of Skorokhod problems in domains $\mathcal{W}_\varepsilon$ \emph{approximating} $\mathcal{W}$. 
The controls  $\bar{v}^\varepsilon$ are $\varepsilon$-optimal for \eqref{eq introduction cost functional} (i.e.\ $ J(x;\bar{v}^\varepsilon) \leq V(x) + \varepsilon$) and converge to the optimal control $\bar{v}$ as $\varepsilon \to 0$.
Finally, the properties of $\bar{v}^\varepsilon$ allows to prove that $\bar{v}$ solves the Skorokhod problem on the original domain $\mathcal{W}$. 

As a consequence of our result, some works (in particular \cite{chiarolla&haussmann2000} and \cite{yang2014}) in the literature on singular control can be revisited, and the characterization of optimal controls can be provided under slightly different assumptions. 
Also, our approach allows to treat the singular control problems with degenerate diffusion matrix studied in \cite{ federico&ferrari&schuhmann2019, federico&ferrari&schuhumann2020b}. 
The results apply  to problems with monotone controls, and to the case in which increasing the underlying diffusion has a different cost than decreasing it. 
The approach presented in this paper seems to be suitable to treat also singular control problems in the finite time-horizon.  

\subsection{Related literature}
Closely related to our result is the analysis presented in \cite{kruk2000}.  
In comparison to this work, our approach seems to work only when the control is one-dimensional, but it allows to treat problems in which the components of the underlying state process are interconnected. 
Therefore, differently from \cite{kruk2000}, our result can be (almost) directly employed to construct the optimal control in  \cite{chiarolla&haussmann2000, federico&ferrari&schuhmann2019, federico&ferrari&schuhumann2020b, yang2014} (see Section \ref{section examples} below).
Also, our setting and some methodologies are similar to  those in \cite{chiarolla&haussmann2000}, where the main aim is to provide a study of the variational inequality \eqref{eq introduction HJB} through a connection with Dynkin games.
The problem of characterizing the optimal control is also discussed therein, though such a characterization is only provided under a (strong) additional assumption on the regularity of the free boundary. 
Despite similarities in the setup, the focus of our paper is instead on the characterization of optimal controls, and our approach allows to construct such a control under an easily verifiable condition (see Example \ref{example chiarolla haussmann} below).

{Clearly, our results relate to  stochastic differential equations (SDEs, in short) with reflecting boundary conditions, also known as Skorokhod reflection problems for SDEs.} 
In this field, existence and uniqueness of strong solutions to reflected SDEs in convex time-independent domains with normal reflection was first shown in the seminal 
\cite{tanaka1979}. 
These results were then generalized to non-convex smooth domains with smooth oblique reflection in \cite{lions&sznitman1984}, and subsequently refined in \cite{saisho1987}.  
Existence of strong solutions in a class of non-smooth domains has been proved in
\cite{dupuis&ishii1993}, and therefore generalized to the time-dependent case in \cite{lundstrom&onskog2019}. This list is, however, far from being exhaustive, and we therefore refer the interested reader to 
\cite{burdzy2004,burdzy&ramanan2009, costantini1992,costantini2006, nystrom&onskog2010,taksar1992}  
and to the references therein.  
A crucial point to be observed is that all of the previous results require some regularity of the boundary and of the direction of reflection. 
In the Skorokhod problem related to singular control, the boundary and the direction of reflection are implicitly given through the optimization problem, and their regularity is, in general, unknown. 
Therefore, to the best of our knowledge, none of the previous papers can be directly employed in order to construct an optimal policy for the general multidimensional singular control problem with cost functional \eqref{eq introduction cost functional}. 
From the reflected SDEs-perspective, our results provide existence and uniqueness of a (strong) solution to a Skorokhod problem in which the domain is given by the noncoincidence set $\mathcal{W}$ of the solution of the variational inequality with gradient constraint \eqref{eq introduction HJB}, and in which the reflection direction is prescribed by its gradient.   
 
An essential tool for our analysis is the connection between optimal stopping and singular stochastic control theory. 
This connection is known since the seminal  \cite{Bather&Chernoff67}, where the authors observed that the derivative of the value function of a singular control problem identifies with the value of an optimal stopping problem.  
Since then, this connections has been elaborated through different approaches (see \cite{benth&reikvam2004,Boetius&Kohlmann98,Karatzas&Shreve84}, among others), until the more recent interpretation given in \cite{li&zitkovic2017}.  
When the control is assumed to be of locally bounded variation, and the system has  dynamics with independent components, with one of them being controlled, the space derivative of the value function of the control problem coincides with the value of a zero-sum game of stopping; i.e., a Dynkin game (cf.\   \cite{boetius2005,chiarolla&haussmann1998, chiarolla&haussmann2000,   guo&tomecek2009,  karatzas&wang2001}).  
This connection was described in a multi-dimensional setting with interconnected dynamics
in \cite{chiarolla&haussmann2000} and \cite{chiarolla&haussmann1998} by employing a variational formulation of the problem. 
In this paper, we employ essentially the formulation and the techniques elaborated in \cite{chiarolla&haussmann2000}, however extending their arguments to fit our convex setting. 
 
\subsection{Outline of the paper}
The rest of this paper is organised as follows. 
In Section \ref{section Statement of the main result} we formulate the problem, we enforce some structural conditions, and we state the main result of this paper. 
The proof of the main result for a constant volatility is presented in Section \ref{section proof sigma constant}, while the proof for a linear volatility is discussed in Section  \ref{section proof main result  geometric}.
Extensions and examples are provided in Section \ref{section examples}, while Appendix \ref{appendix auxiliary results} and Appendix \ref{appendix proof propositions kruk} are devoted to some auxiliary technical results.

\subsection{Notation} 
For $d \in \mathbb{N}$ with $d\geq 1$, an open set $B \subset \mathbb{R}^d$, $\alpha=(\alpha_1,...,\alpha_{d})\in \mathbb{N}^{d}$ and a function $f:B \to \mathbb{R}$, we denote by $D^\alpha f:= D_1^{\alpha_1} ... D_{d}^{\alpha_{d}} f$ the weak derivative of $f$, where $D_i f:=f_{x_i} := \partial f / \partial x_i$, and we set $|\alpha|:=\alpha_1 + ... +\alpha_d$. For $\ell \in \mathbb{N}$, $q\in [1,\infty]$, and a measure space $(E, \mathcal{E}, m)$, we define the spaces: 
\begin{itemize}
\item $L^q(E):=\{ \text{measurable } f:E\to \mathbb{R} \text{ s.t. } \| f\|_{L^q(E)} < \infty \}$, where $\| f\|_{L^q(E)}^q:=\int_E |f|^q dm$ if $q <\infty$, and $\| f\|_{L^\infty(E)}:=\esssup_E f$ for $q=\infty$;
\item ${L}_{loc}^{q}(B):=\{ f \,|\, f \in  {L}^{q}(D) \text{ for each bounded open set } D\subset B \}$;
\item $C^{\ell }(B):=\{f:B\to \R \text{ with continuous $\ell$-order derivatives} \}$  and \\  $C_c^\infty(B):=\{f:B \to \mathbb{R} \text{ with} \text{ compact support, s.t.\ }f\in C^{\ell}(B) \text{ for each } \ell \in \mathbb{N} \}$;
    \item $C^{\ell ; 1}(B):=\{f:B\to \R \text{ s.t. $\|f\|_{C^{\ell;1}{(B)}}< \infty$} \}$, where $\|f \|_{C^0{(B)}}:=\sup_{x \in B}|f(x)|$,  $\|f\|_{\lip{(B)}}:= \sup_{x,y \in B}|f(y)-f(x)|/|y-x|$, and
    $\|f\|_{C^{\ell;1}{(B)}}:=\sum_{|\alpha| \leq \ell}\| D^\alpha f \|_{C^0(B)} + \sum_{|\alpha|= \ell } \|D^\alpha f \|_{\lip{(B)}};$
    \item $C^{\ell ; 1}_{loc}(B):=\{ f \,|\, f \in  C^{\ell ; 1}(D) \text{ for each bounded open set } D\subset B \}$; 
    \item ${W}^{\ell;q}(B):=\{f \in L^q(B) \text{ with $\| f\|_{{W}^{\ell;q}(B)}< \infty$}\}$, \\ ${W}_{loc}^{\ell;q}(B):=\{ f \,|\, f \in  {W}^{\ell;q}(D) \text{ for each bounded open set } D\subset B \}$, and ${W}_{0}^{\ell;q}(B)$ as the closure of $C_c^\infty(B)$ in the norm $\| \cdot \|_{{W}^{\ell;q}(B)}$, where $\| f\|_{{W}^{\ell;q}(B)}:= \sum_{|\alpha| \leq \ell} \| D^{\alpha} f\|_{L^q(B)}$.
\end{itemize} 
For $x \in \mathbb{R}^d$ we denote by $x^{\text{\tiny {$\top$}}}$ the transpose of $x$. 
For $x,y \in \R^d$, we denote by $xy$ the scalar product in $\R^d$, as well as by $|\cdot|$ the Euclidean norm in $\R^d$. 
{Moreover, we set $x\leq y$ if $x_i \leq y_i$ for any $i=1,...,d$.} 
The vector $e_i\in \R^d$ indicates the $i$-th element of the canonical basis of $\mathbb{R}^d$ and, for $x\in \mathbb{R}^{d}$ and $R>0$, set $B_R(x):=\{ y \in \mathbb{R}^{d} \, | \, |y-x|<R \}$.  
Finally, in this paper $C$ indicates a generic positive constant, which may change from line to line.   

\section{Problem formulation and main result}\label{section Statement of the main result} 
\subsection{Singular control and Skorokhod problem}
\label{subsection Problem formulation and assumptions} 
Fix $d \in \mathbb{N}$, $d\geq 2$, and 
a $d$-dimensional Brownian motion $W=(W^1,...,W^d)$
on a filtered probability space $(\Omega, \mathcal{F},  \mathbb{F},\mathbb{P})$ satisfying the usual conditions.  
For each $x=(x_1,...,x_d) \in \mathbb{R}^{d}$, let the process $X^{x}=(X^{1,x},...,X^{d,x})$ denote the solution to the stochastic differential equation (SDE, in short) 
\begin{equation}
\label{SDE dynamics uncontrolled}
    \begin{cases}
        dX_t^{1,x}=(a_1 +b_1^1 X_t^{1,x})dt + \bar{\sigma}(X_t^{1,x}) dW_{t}^1, &\quad t \geq 0,\quad X_{0-}^{1,x}=x_1, \\ 
        dX_t^{i,x}=b^i(X_t^{1,x},X_t^{i,x})dt + \bar{\sigma}(X_t^{i,x}) dW_{t}^i, & \quad t \geq 0,\quad X_{0-}^{i,x}=x_i, \quad i=2,...,d. 
    \end{cases}     
\end{equation} 
Here  $a_1, b_1^1 $ are constants, while the coefficients $b^i \in C(\R^2)$ and $\bar{\sigma} \in C(\R)$ are deterministic Lipschitz continuous functions. The drift $\bar{b}(x):=(a_1 +b_1^1x_1, b^2(x_1,x_2),..,b^d (x_1,x_d))^{\text{\tiny {$\top$}}}$ and the function $\bar{\sigma}$ satisfy Assumption \ref{assumption} below. 
Next, introduce the set of \emph{admissible controls} as
\begin{equation*}  
\mathcal{V} : = \left\{ \text{$\mathbb{R}$-valued $\mathbb{F}$-adapted and c\`adl\`ag processes with locally bounded variation}  \right\},
\end{equation*} 
and, for each $v\in \mathcal{V}$ and $x\in \mathbb{R}^{d}$, let the process $X^{x;v}=(X^{1,x;v},...,X^{d,x;v})$ denote the unique strong solution to the  controlled stochastic differential equation
\begin{equation}
\label{SDE dynamics}
    \begin{cases}
        dX_t^{1,x;v}=(a_1 +b_1^1 X_t^{1,x;v})dt + \bar{\sigma}(X_t^{1,x;v}) dW_{t}^1 +dv_t, & t \geq 0,\ X_{0-}^{1,x;v}=x_1, \\ 
        dX_t^{i,x;v}=b^i(X_t^{1,x;v},X_t^{i,x;v})dt + \bar{\sigma}(X_t^{i,x;v}) dW_{t}^i, &  t \geq 0,\ X_{0-}^{i,x;v}=x_i, \ i=2,...,d. 
    \end{cases}  
\end{equation}   

For any given initial condition $x\in \mathbb{R}^{d}$, consider the problem of minimizing the expected discounted cost
\begin{equation}\label{eq definition J}
J(x;v):= \mathbb{E} \bigg[ \int_0^\infty e^{-\rho t} h(X_t^{x;v}) dt + \int_{[0,\infty)} e^{-\rho t} d|v|_{t}  \bigg], \quad v \in \mathcal{V}, 
\end{equation}
where $|v|$ denotes the total variation of the process $v$, $h: \mathbb{R}^{d} \to \mathbb{R}$ is a continuous function, and $\rho>0$ is a constant discount factor.
We will say that the control $\bar{v} \in \mathcal{V}$ is optimal if 
\begin{equation}\label{eq definition V}
V(x):= \inf_{v \in \mathcal{V}} J(x;v) = J(x;\bar{v}),  
\end{equation}  
and, in the following, we will refer to the function $V$ as to the value function of the problem, and to $X^{x;\bar{v}}$ as to the optimal trajectory. 

The second integral appearing in \eqref{eq definition J} has to be understood in the Lebesgue-Stieltjes sense, and it is defined as
$$
 \int_{[0,\infty)} e^{-\rho t} d|v|_{t}:= |v|_0 + \int_0^\infty e^{-\rho t} d|v|_{t},
$$
in order to take into account possible jumps of the control at time zero.
Moreover, for $v\in \mathcal{V}$ we will often write $dv=\gamma d|v|$ to denote the disintegration
$$
v_t = \int_0^t \gamma_s d|v|_s, \quad \text{for each }t\geq 0, \ \mathbb{P}\text{-a.s.,}
$$
where $|v|$ denotes the total variation of the signed measure $v$, and the process $\gamma$ is the Radon-Nikodym derivative of the signed measure $v$ with respect to $|v|$.  Also, for a control $v$, the nondecreasing c\`adl\`ag processes $\xi^+, \, \xi^-$ will denote the minimal decomposition of the signed measure $v$; that is, $v=\xi^+ - \xi^-$, and $\xi^+ \leq \bar{\xi}^+$ and $\xi^- \leq \bar{\xi}^-$ for any other couple of nondecreasing c\`adl\`ag processes $\bar{\xi}^+, \, \bar{\xi}^-$ which satisfy  $v=\bar{\xi}^+ - \bar{\xi}^-$.

Finally, recall from \cite{kruk2000} the following notion of solution to the Skorokhod problem, which we adapt to our setting.  
\begin{definition} \label{def Skorokhod problem}
Let $\mathcal{O}$ be an open subset of $\mathbb{R}^{d}$ with closure $\overline{\mathcal{O}}$,  $ {x} \in \overline{\mathcal{O}}$, and set $S:=\partial \mathcal{O}$. Let $\bar{\nu}$ be a continuous vector field on $S$, with $\bar{\nu}=e_1 \nu$ and $|\nu (y)|=1$ for each $y\in S$. 

We say that the process ${v} \in \mathcal{V}$ is a solution to the modified Skorokhod problem for the SDE \eqref{SDE dynamics} in $\overline{\mathcal{O}}$ starting at ${x}$ with reflection direction $\bar{\nu}$ if
\begin{enumerate} 
    \item $\mathbb{P}[X_t^{{x};{v}} \in \overline{\mathcal{O}}, \, \forall t\geq 0 ]=1$;
    \item $\mathbb{P}$-a.s., for each $t\geq 0$ one has $d{v}={\gamma}d{|v|}$, with  
    $${|v|}_t=\int_0^t \mathds{1}_{ \{  X_{s-}^{{x};{v}} \in S, \, \nu (X_{s-}^{{x};{v}}) = {\gamma}_s \} } d {| v | }_s;  $$
    \item  $\mathbb{P}$-a.s., for each $t \geq 0$, a possible jump of the process $X^{{x};{v}}$ at time $t$ occurs on some interval $I\subset S$ parallel to the vector field $\bar{\nu}$; i.e., such that $\bar{\nu}(y)$ is parallel to $I$ for each $y\in I$. If $X^{{x};{v}}$ encounters such an interval $I$, it instantaneously jumps to its endpoint in the direction $\bar{\nu}$ on $I$. 
\end{enumerate}

Moreover, if ${v}$ is continuous
, then we say that ${v}$ is a solution to the (classical) Skorokhod problem for the SDE \eqref{SDE dynamics} in $\overline{\mathcal{O}}$ starting at ${x}$ with reflection direction $\bar{\nu}$.
\end{definition}

\subsection{Assumptions and main result}\label{subsection main result}
The main objective of this paper is to characterize optimal control policies for Problem \eqref{eq definition V} as solutions to related Skorokhod problems.  

We will prove our main result under the following structural conditions, which we enforce throughout the rest of this paper. We postpone the discussion of some generalizations to Section \ref{section examples}.
\begin{assumption}
\label{assumption} For $p\geq 2$ we have:\
\begin{enumerate}
\item \label{assumption on h} The running cost $h$ is $C_{loc}^{2;1}(\mathbb{R}^{d})$, convex, and, for suitable  constants $K, \kappa_1, \kappa_2>0$, it satisfies, for each $x,y\in \mathbb{R}^{d}$  and for all $\lambda \in [0,1]$, the conditions
    \begin{align*}
        \kappa_1|x_1|^{{p}} - \kappa_2 \leq h(x) & \leq K (1+|x|^p),
        \\       
        |h(y)-h(x)| &\leq K (1+|x|^{p-1}+|y|^{p-1})|y-x|,  \\
        \lambda h(x) +(1-\lambda) h(y) -h( \lambda x + (1-\lambda) y) & \leq K \lambda(1-\lambda)(1+|x|^{p-2}+|y|^{p-2})|x-y|^2, \\
        0 & < h_{x_1 x_1}(x).   
    \end{align*}    
\item\label{assumption on b} There exists a constant  $\bar{L} \geq 0$ such that, for each $x, y \in \R^d$, we have 
\begin{align*}
    |\bar{b}(x)| & \leq \bar{L}(1+|x|), \\
    |\bar{b}(y)-\bar{b}(x)| & \leq \bar{L}|y-x|.  
\end{align*}  
The functions  $b^{i}$ are convex of class  $C_{loc}^{2;1}(\R^d)$. Furthermore, we assume that $h_{x_i} \geq 0$ and $ {b}_{x_1 }^{i},\, {b}_{x_1 x_{i}}^{i},\, h_{x_1 x_{i}} \leq 0 $  for each $i=2,...,d$, and that $D\bar{b}$ is globally Lipschitz.
\item\label{ass discount factor} For $\rho^*:={{p}(2 p -1 )}$ and a constant $\sigma>0$, either of the two conditions below is satisfied: 
\begin{enumerate} 
    \item\label{ass sigma constant} $\bar{\sigma}(y)=\sigma, \ y \in \R$, and the discount factor satisfies the relation  $ {\rho} > {3 \rho^*} \bar{L} $.
     \item\label{ass sigma geometric}  $\bar{\sigma}(y)=\sigma y, \ y \in \R$, and the discount factor satisfies the relation $ {\rho} > {2 \rho^*} ( \bar{L}   +\sigma^2(\rho^*-1))$. 
     In this case, we also assume that there exists $x_1^*>0$ such that $h_{x_1}(x) \leq \min \{0,-b_1^1\}$ for each $x$ with $x_1 < 2 x_1^*$, that $b^i(x_1,x_i)\geq0$ for $x_1, x_i \geq0$ for each $i=2,...,d$, and that $a_1 \geq 0$.   
\end{enumerate} 
\end{enumerate}  
\end{assumption} 
  
Natural examples in which the conditions above are satisfied are given --after discussing  generalizations of Assumption \ref{assumption}-- in Section \ref{section examples}. These include a relevant class of  \emph{linear-quadratic}  singular stochastic control problems (see Example \ref{example key} and Subsection \ref{subsection example Patrick} below).
Notice that the nature of  problem \eqref{eq definition V} is genuinely multidimensional, as the  components of the dynamics \eqref{SDE dynamics} are interconnected. 

\begin{remark}[On the role of Assumption \ref{assumption}]\label{remark role of assumption} We underline that the particular choice of $p \geq 2$ is motivated by  quadratic running costs (cf.\ Example \ref{example key} in Section \ref{section examples}).  
From Condition \ref{assumption on b} one can see that quite strong requirements are needed in order to treat models with a general $b^i$. However, when $b^i$ has a simpler form, some conditions on the derivatives ${b}_{x_1 }^{i},\, {b}_{x_1 x_{i}}^{i},\, h_{x_{i}},\, h_{x_1 x_{i}}$ can be weakened (see Subsections \ref{subsection Affine drift} and \ref{subsection On Condition 2}).
Also, the assumption on $h_{x_1}$ in Condition \ref{ass sigma geometric} is to enforce that the optimal trajectories  live in the set  $\mathbb{R}_+^d:=\{ x\in \R^d |\, x_i>0, \ i=1,...,d \}$, whenever the initial condition $x \in \mathbb{R}_+^d$ (cf.\ Lemma \ref{lemma geometric X>0} below). 
This condition is a natural substitute, for minimization problems in dimension $d \geq 2$, of the classical Inada condition at 0 (see, e.g., equation (2.5) in \cite{guo.pham2005}). 
The latter, is typically assumed in profit maximization problems, and it is satisfied by Cobb-Douglas production functions. 
Finally, the conditions on the discount factor $\rho$ are in place in order to ensure a suitable ``integrability" of the optimal trajectories, which allows to prove some semiconcavity estimates for the value function $V$ (see steps 2 and 3 in the proof of Theorem \ref{theorem V sol HJB} in Appendix \ref{appendix auxiliary results}).
\end{remark}
Observe that, when Condition \ref{ass sigma constant} is in place, some controlled trajectories can reach the whole space with probability $\mathbb{P}>0$. 
On the other hand, under Condition \ref{ass sigma geometric}, as mentioned in Remark \ref{remark role of assumption}, the natural domain for an optimally-controlled trajectory  is $\mathbb{R}_+^d$. 
This suggests to define a domain $D$ in the following way 
\begin{equation}\label{eq definition domain} 
D:=\mathbb{R}^d   \text{ if Condition \ref{ass sigma constant} holds}, \quad  D:=\mathbb{R}_+^d   \text{ if Condition \ref{ass sigma geometric} holds.} 
\end{equation}
Indeed, the value function $V$ is finite and it is a  convex solution in $W_{loc}^{2;\infty}(D)$ of the  Hamilton-Jacobi-Bellman (HJB, in short) equation \begin{equation} 
\label{HJB v.i.} 
\max \{ \rho V-\mathcal{L} V -h , | V_{x_1}| - 1 \} = 0,  \quad \text{a.e.\ in } D, 
\end{equation}
where 
$
\mathcal{L}V(x):= \bar{b}(x) DV(x) + \frac{1}{2} \sum_{i=1}^d  \bar{\sigma}^2(x_i) V_{x_i x_i}(x), \ x \in D, 
$
is the generator of the uncontrolled SDE  \eqref{SDE dynamics uncontrolled}.
For completeness, a proof of this result is provided in Appendix \ref{appendix auxiliary results} (see Theorem \ref{theorem V sol HJB}). During the proof of Theorem \ref{theorem V sol HJB}, the convergence of a certain penalization method is studied: This convergence will be a useful tool in many of the proofs in this paper. 

Define next the \emph{waiting region} $\mathcal{W}$ as  
\begin{equation}
\label{waiting region definition}
\mathcal{W}:= \{ x \in D \,|\,  |V_{x_1}(x)| < 1 \},
\end{equation}
 and notice that, by the $W_{loc}^{2;\infty}$-regularity of $V$, $\mathcal{W}$ is an open subset of $D$. Also, for each $z\in \mathbb{R}^{d-1}$, we define the sets $$
D_1(z):= \{ y\in \mathbb{R}\, |\, (y,z) \in D \}  \quad \text{and} \quad \mathcal{W}_1(z):= \{ y\in \mathbb{R}\, |\, (y,z) \in \mathcal{W} \}.
 $$ 
 In the sequel, the closure of $\mathcal{W}$ (resp.\ $\mathcal{W}_1 (z)$) in $D$ (resp.\ $D_1(z)$)   will be denoted by $\overline{\mathcal{W}}$ (resp.\ $\overline{\mathcal{W}}_1 (z)$). We state here a technical lemma, whose proof is given in Appendix \ref{appendix proof propositions kruk}. 
 \begin{lemma}\label{lemma waiting region nonempty}
For any ${x}=({x}_1, {z}) \in D$, with ${z} \in \mathbb{R}^{d-1}$, the set $\mathcal{W}_1(z)$ is a nonempty open interval; in particular, $\mathcal{W}$ is nonempty.  
 \end{lemma}
 
 \begin{remark}[Existence and uniqueness of optimal controls]\label{lemma appendix existence optimal controls}  Under Assumption \ref{assumption}, for each $\bar{x} \in D$ there exists a unique optimal control $\bar{v} \in \mathcal{V}$. 
This is a classical result when the drift is affine. 
In the case of a convex drift, it essentially follows from the convexity of $J$ w.r.t.\ $(x,v)$.   
The latter in turn follows from the convexity of the drift, the monotonicity of $h$, and a comparison principle for SDEs ({see Step 1 in the proof of Theorem \ref{theorem V sol HJB} in Appendix \ref{appendix auxiliary results} or} Lemma 3.1 in \cite{Boetius&Kohlmann98}). The argument can be recovered from the proof of Lemma \ref{lemma construction optimal policy} below, which works for any sequence of controls minimizing the cost functional $J$. 
Finally, the uniqueness of the optimal control is a consequence of the strict convexity of $h$ in the variable $x_1$ {(see Step 1 in the proof of Theorem \ref{theorem V sol HJB} in Appendix \ref{appendix auxiliary results})}.
\end{remark}
 
The following is the main result of our paper, characterizing the optimal policies in terms of the waiting region $\mathcal{W}$ and the derivative $V_{x_1}$ in the sense of Definition \ref{def Skorokhod problem}.   
\begin{theorem}
\label{theorem main characterization}
Let $\bar{x}=(\bar{x}_1, \bar{z}) \in D$, with $\bar{z} \in \mathbb{R}^{d-1}$. The following statements hold true: 
\begin{enumerate}  
    \item\label{theorem x inside waiting} If $\bar{x} \in \overline{\mathcal{W}}$, then the optimal control $\bar{v} $ is the unique solution to the modified Skorokhod problem for the SDE \eqref{SDE dynamics} in $\overline{\mathcal{W}}$ starting at $\bar{x}$ with reflection direction $-V_{x_1} e_1$; 
    \item\label{theorem x outside waiting} If $\bar{x} \notin  \overline{\mathcal{W}}$, then the optimal control $\bar{v}$ can be written as $\bar{v} = \bar{y}_1 -\bar{x}_1 + \bar{w}$, where 
     $\bar{y}_1$ is the metric projection of $\bar{x}_1$ into the set $\overline{\mathcal{W}}_1(\bar{z})$, 
    and $\bar{w}$ is the unique solution to the modified Skorokhod problem for the SDE \eqref{SDE dynamics} in $\overline{\mathcal{W}}$ starting at $\bar{y}:= (\bar{y}_1,\bar{z})$ with reflection direction $-V_{x_1}e_1$.   
\end{enumerate} 
\end{theorem}
In Section \ref{section proof sigma constant} we provide a proof of Theorem \ref{theorem main characterization} under Condition \ref{ass sigma constant} in Assumption \ref{assumption}. The strategy of the proof can be resumed in three main steps: 
\begin{itemize}
    \item[Step a.] In Subsection \ref{section a connection with Dynkin games} we study an important monotonicity property of $V_{x_1}$, through a connection with Dynkin games. 
    \item[Step b.] In Subsection \ref{section epsilon optimal policies}, this property  will allow us to construct $\varepsilon$-optimal policies as solutions to Skorokhod problems in domains $\overline{\mathcal{W}}_\varepsilon$ approximating  $\overline{\mathcal{W}}$.
    \item[Step c.] Finally, in Subsection \ref{section proof of the main results} we prove that the $\varepsilon$-optimal policies approximate the optimal policy, and that the latter is a solution to the Skorokhod problem in the original domain  $\overline{\mathcal{W}}$.
\end{itemize}  
The proof of Theorem \ref{theorem main characterization} under Condition \ref{ass sigma geometric} in Assumption \ref{assumption} follows similar rationales, and it is discussed in Section \ref{section proof main result geometric}. In particular, in Subsections \ref{subsection geom preliminary lemma} a preliminary lemma is proved, while in Subsection \ref{subsection sketch of the proof geom} we show how to use this lemma in order to repeat (with minor modifications) the arguments of Section \ref{section proof sigma constant}.  

\section{Proof of Theorem \ref{theorem main characterization} for constant volatility}\label{section proof sigma constant}
In this section we assume that Condition \ref{ass sigma constant} in Assumption \ref{assumption} holds. To simplify the notation, the proof is given for $d=2$, so that $D=\mathbb{R}^2$. The generalization to the case $d>2$ is straightforward.
\subsection{Step a: A connection to Dynkin games and the monotonicity property} 
\label{section a connection with Dynkin games} 
In this subsection we adopt an approach based on the variational formulation of the problem in order to show, in the spirit of \cite{chiarolla&haussmann2000}, a connection between the singular control problem \eqref{eq definition V} and a Dynkin game. 
This connection will enable us to prove a monotonicity property of $V_{x_1}$, which will be then fundamental in order to construct  $\varepsilon$-optimal controls.  

\subsubsection{The related Dynkin game} 
We begin by characterizing $V_{x_1}$ as a $W_{loc}^{2;\infty}$-solution to a \emph{two-obstacle problem}. The proof of the next result borrows arguments from \cite{chiarolla&haussmann2000} (see in particular Theorem 3.9, Proposition 3.10, and Theorem 3.11 therein). However, since in our case $b$ can be convex, the techniques used in \cite{chiarolla&haussmann2000} needs to be  refined, and used along with suitable estimates (described more in detail in the proof of Theorem \ref{theorem V sol HJB} in Appendix \ref{appendix auxiliary results}) on a penalization method. 
We provide a detailed proof for the sake of completeness.
\begin{theorem}
\label{theorem connection Dynkin game PDE} The function $V_{x_1}$ is a $W_{loc}^{2;\infty}(\mathbb{R}^{2})$-solution to the equation \begin{equation}
\label{eq Dynking game HJB}
\\max \{(\rho-b_{1}^1) V_{x_1}-\mathcal{L} V_{x_1} - \hat{h} , | V_{x_1}| - 1 \} = 0,  \quad \text{a.e.\ in } \mathbb{R}^{2},  
\end{equation}
where $\hat{h}:=h_{x_1} + b_{x_1}^2V_{x_2}$.
\end{theorem}
\begin{proof} We organize the proof in two steps.
\smallbreak \noindent
\emph{Step 1.} In this step we show that the function $V_{x_1}$ is a solution to a variational inequality with a local operator, and that $V_{x_1} \in W_{loc}^{2;\infty}(\mathbb{R}^{2})$.   
Fix $B \subset \mathbb{R}^{2}$ open bounded and consider a nonnegative localizing function $\psi \in C_c^\infty (B)$. 
Define the sets 
$$
\mathcal{K}:=\big\{ U \in W_{loc}^{1;2}(\mathbb{R}^{2}) \, | \, |U| \leq 1 \ \text{a.e.} \big\}\quad  \text{and}\quad \mathcal{K}_\psi:=\{ \psi U\, |\, U \in \mathcal{K} \}.  
$$ 
We show in the sequel that the function $W:=V_{x_1} \psi$ is a solution in $\mathcal{K}_\psi$ to the variational inequality 
\begin{equation}
\label{eq variational inequality with psi} 
A_B(W, U - W) \geq \langle \hat{H} , U - W \rangle_B, \quad \text{for each } U \in \mathcal{K}_\psi,
\end{equation}
where $\hat{H} := \hat{h} \psi -   V_{x_1} \mathcal{L} \psi - {\sigma^2} DV_{x_1} D\psi$, the operator $A_B: W^{1;2}(B) \times W^{1;2}(B)  \to \mathbb{R}$ is given by
$$
A_B(\bar{U}, U) := \frac{\sigma^2}{2} \sum_{i=1}^{2} \langle \bar{U}_{x_i}, U_{x_i} \rangle_B -  \langle \bar{b} D\bar{U}, U \rangle_B + (\rho-b_{1}^1)\langle \bar{U}, U \rangle_B \quad \text{for each } \bar{U},  U \in W^{1;2}(B),
$$  
and $\langle \cdot, \cdot\rangle_B$ denotes the scalar product in $L^2(B)$. 
                   
Let us begin by introducing a family of penalized versions of the HJB equation \eqref{HJB v.i.}. Let $\beta \in C^\infty(\mathbb{R})$ be a convex nondecreasing function with $\beta(r)=0$ if $r\leq 0$ and $\beta(r)=2r-1$ if $r\geq 1$. For each $\varepsilon > 0$, let $V^\varepsilon$ be defined as in \eqref{eq control problem penalized}. As in Step 1 in the proof of Theorem \ref{theorem V sol HJB} in Appendix \ref{appendix auxiliary results}, $V^\varepsilon$  is a  $C^2$-solution to the partial differential equation 
\begin{equation}
    \label{penalized HJB} 
     \rho V^\varepsilon-\mathcal{L}V^\varepsilon + \frac{1}{\varepsilon}\beta ( (V_{x_1}^\varepsilon)^2 - 1 ) = h, \quad x \in \mathbb{R}^{2}.
\end{equation}
As in Step 3 in the proof of Theorem \ref{theorem V sol HJB} in Appendix \ref{appendix auxiliary results}, for each $R>0$ there exists a constant $C_R$ such that   
\begin{equation}
\label{eq uniform sobolev estimate V}
\sup_{\varepsilon \in (0,1)} \| V^\varepsilon \|_{W^{2;\infty}(B_R)} \leq C_R.
\end{equation}
Moreover (as in \eqref{eq appendix Sobolev limits} in the proof of Theorem \ref{theorem V sol HJB}), as $\varepsilon \to 0$, on each subsequence we have: 
\begin{align}
\label{eq limits in sobolev}
     & (V^\varepsilon,DV^\varepsilon) \text{ converges to $(V,DV)$ uniformly  in $ B_R$}; \\ \notag
     & D^2V^\varepsilon \text{ converges to $ D^2V$ weakly  in $L^2( B_R)$}.
\end{align}

We now show that $V_{x_1} \in \mathcal{K}$. Since the $W_{loc}^{1;2}$-regularity of $V_{x_1}$ is already known (cf.\ Theorem \ref{theorem V sol HJB} in Appendix \ref{appendix auxiliary results}), we only need to show that $|V_{x_1}| \leq 1 $ in  $\mathbb{R}^{2}$. 
To this end, take $R>0$ and observe that, by \eqref{eq uniform sobolev estimate V} and \eqref{penalized HJB}, we have  
\begin{equation}\label{eq L2 estimate beta}
   \sup_{\varepsilon \in (0,1)} \| \beta ((V_{x_1}^\varepsilon)^2 -1 ) \|_{L^2(B_R)} \leq C_R \varepsilon, 
\end{equation} 
where the constant $C_R>0$ does not depend on $\varepsilon$. Moreover, unless to consider a larger $C_R$, by the estimate \eqref{eq uniform sobolev estimate V} and the definition of $\beta$, we also have the pointwise estimate
\begin{equation}\label{eq pointwise estimate beta}
|\beta ((V_{x_1}^\varepsilon)^2 -1 )| \leq 2((V_{x_1}^\varepsilon)^2 +1 ) \leq C_R, \quad \text{ on $B_R$, for each $\varepsilon \in (0,1)$}. 
\end{equation}   
Therefore, the limits in \eqref{eq limits in sobolev} and the estimates \eqref{eq pointwise estimate beta} allow to invoke the dominated convergence theorem to deduce, thanks to \eqref{eq L2 estimate beta}, that
$$
\| \beta ((V_{x_1})^2 -1 ) \|_{L^2(B_R)} = \lim_{\varepsilon \to 0} \| \beta ((V_{x_1}^\varepsilon)^2 -1 ) \|_{L^2(B_R)} =0.   
$$ 
Since $R$ is arbitrary, we conclude that $|V_{x_1}| \leq 1$ a.e.\ in $\mathbb{R}^{2}$, and therefore that $W\in\mathcal{K}_\psi$.

We continue by proving \eqref{eq variational inequality with psi}. Since $V^\varepsilon$ is a solution to \eqref{penalized HJB}, a standard bootstrapping argument (using Theorem 6.17 at p.\ 109 in \cite{gilbarg2001}) allows to improve the regularity of $V^\varepsilon$ and to prove that $V^\varepsilon \in C^{4}$. 
Therefore, we can differentiate  equation \eqref{penalized HJB} with respect to $x_1$ in order to get an equation for  $V_{x_1}^\varepsilon$. 
That is,
\begin{equation}
    \label{eq penalized derived HJB}
    [(\rho-b_{1}^1)-\mathcal{L}]V_{x_1}^\varepsilon + \frac{2}{\varepsilon}\beta' ( (V_{x_1}^\varepsilon)^2 - 1 )V_{x_1}^\varepsilon V_{x_1 x_1 }^\varepsilon = \hat{h}^\varepsilon , \quad x \in \mathbb{R}^{2},  
\end{equation}
where we have defined $\hat{h}^\varepsilon := h_{x_1} + b_{x_1}^2 V_{x_2}^\varepsilon$. Moreover, by \eqref{eq penalized derived HJB}, the localized function $V_\psi^\varepsilon:= V_{x_1}^\varepsilon \psi$ is a solution to the equation 
\begin{equation}
    \label{eq penalized localized derived HJB}
    [(\rho-b_{1}^1)-\mathcal{L}]V_\psi^\varepsilon + \frac{2}{\varepsilon}\beta' ( (V_{x_1}^\varepsilon)^2 - 1 )V_\psi^\varepsilon V_{x_1 x_1 }^\varepsilon = \hat{H}^\varepsilon , \quad x \in \mathbb{R}^{2},  
\end{equation}
where $\hat{H}^\varepsilon := \hat{h}^\varepsilon \psi - V_{x_1}^\varepsilon \mathcal{L}\psi - {\sigma^2} DV_{x_1}^\varepsilon D\psi$. 

Let now $U \in \mathcal{K}_\psi$. Taking the scalar product of \eqref{eq penalized localized derived HJB} with $U-V_{\psi}^\varepsilon$, an integration by parts gives
\begin{equation}
\label{eq variational inequaliti epsilon psi}
A_B(V_{\psi}^\varepsilon, U - V_{\psi}^\varepsilon) + \frac{2}{\varepsilon} \langle \beta' ( (V_{x_1}^\varepsilon)^2 - 1 )V_{\psi}^\varepsilon V_{x_1 x_1 }^\varepsilon, U - V_{\psi}^\varepsilon \rangle_B = \langle \hat{H}^\varepsilon , U - V_{\psi}^\varepsilon \rangle_B. 
\end{equation} 
Moreover, since $\sigma >0$, the operator 
$
\big( \frac{\sigma^2}{2} \sum_{i=1}^{2} \langle {U}_{x_i}, U_{x_i} \rangle_B \big)^{{1}/{2}}, \ U \in W^{1;2}(B),
$ 
defines a norm on $W_0^{1;2}(B)$. Therefore, such an operator is lower semi-continuous with respect to the  weak convergence in $W_0^{1;2}(B)$. By the limits in \eqref{eq limits in sobolev}, this implies that 
\begin{equation} 
    \label{eq liminf of the norm}
    \liminf_{\varepsilon \to 0} \frac{\sigma^2}{2} \sum_{i=1}^{2} \langle V_{\psi x_i}^\varepsilon,  V_{\psi x_i}^\varepsilon \rangle_B 
    \geq  \frac{\sigma^2}{2} \sum_{i=1}^{2} 
     \langle W_{ x_i}, W_{ x_i} \rangle_B. 
\end{equation}  
Therefore exploiting the convergences in \eqref{eq limits in sobolev} and \eqref{eq liminf of the norm}, taking the liminf as $\varepsilon \to 0$ in \eqref{eq variational inequaliti epsilon psi}, we obtain   
\begin{equation}\label{eq liminf variational inequality}
A_B(W, U - W) + \liminf_{\varepsilon \to 0}  \frac{2}{\varepsilon} \langle \beta' ( (V_{x_1}^\varepsilon)^2 - 1 )V_{\psi}^\varepsilon V_{x_1 x_1 }^\varepsilon, U - V_{\psi}^\varepsilon \rangle_B \geq \langle \hat{H} , U - W \rangle_B. 
\end{equation}  
In order to prove \eqref{eq variational inequality with psi}, it thus only remains to show that the scalar product in \eqref{eq liminf variational inequality} involving $\beta'$ is nonpositive. 
Write $U$ as $U= \psi \bar{U}$, with $\bar{U} \in \mathcal{K}$. If $x\in \mathbb{R}^{2}$ is such that $(V_{x_1}^\varepsilon(x))^2 \leq (\bar{U}(x))^2$, then $\beta'((V_{x_1}^\varepsilon(x))^2 -1)=0$ since $\bar{U} \in \mathcal{K}$. On the other hand, if $(V_{x_1}^\varepsilon(x))^2 > (\bar{U}(x))^2$ then we have $2V_{\psi}^\varepsilon (U - V_{\psi}^\varepsilon) \leq U^2 -(V_{\psi}^\varepsilon)^2 < 0$. Hence, since $V^\varepsilon$ is convex and $\beta'$ nonnegative, in both cases we deduce that
$$
\frac{2}{\varepsilon}\beta' ( (V_{x_1}^\varepsilon)^2 - 1 )V_{\psi}^\varepsilon V_{x_1 x_1 }^\varepsilon( U - V_{\psi}^\varepsilon) \leq 0. 
$$ 
Therefore, we conclude that $W$ is a solution to the variational inequality \eqref{eq variational inequality with psi}. 

Finally, since $\sigma>0$, the $W_{loc}^{2;\infty}$-regularity of $V_{x_1}$ follows from Theorem 4.1 at p.\ 31 in \cite{friedman2010}, slightly modified in order to fit problem \eqref{eq variational inequality with psi} (see Problem 1 at p.\ 44, combined with Problems 2 and 5 at p.\ 29 in \cite{friedman2010}).   

\smallbreak \noindent 
\emph{Step 2.} We now prove that $V_{x_1}$ is a pointwise solution to  \eqref{eq Dynking game HJB}. 
For $B \subset \mathbb{R}^{2}$ open bounded and $\psi \in C_c^\infty (B)$, by Step 1 we have that $V_{x_1}\psi$ is a solution to the variational inequality \eqref{eq variational inequality with psi}.  
Moreover, thanks to the regularity of $V_{x_1}$, an integration by parts in   \eqref{eq variational inequality with psi} reveals that,
\begin{equation*}
\langle \hat{L}\psi, (U-V_{x_1})\psi \rangle_B \geq 0, \ \text{for each }U\in \mathcal{K}, 
\end{equation*}
where $\hat{L}:= [(\rho-b_{1}^1)-\mathcal{L}]V_{x_1} -\hat{h}$. The latter, in turn implies that
\begin{equation}
\label{eq variational inequality integrated by parts}
\langle \hat{L}\psi, (U-V_{x_1})\psi \rangle_B \geq 0, \ \text{for each } U\in \widehat{\mathcal{K}}:= \big\{ U \in L_{loc}^{2}(\mathbb{R}^{2}) \, | \, |U| \leq 1 \ \text{a.e.} \big\}.
\end{equation}
For every $\varepsilon >0$, define the sets
$\widehat{\mathcal{W}}_{\varepsilon} := \{ |V_{x_1}| < 1-\varepsilon \}$ and,
for $N>0$ and $0<\eta<\varepsilon / N$, set  $\hat{\psi}:= - \eta \hat{L} \mathds{1}_{\widehat{\mathcal{W}}_\varepsilon} \mathds{1}_{\{ \hat{L} < N \}  }$. 
Define next  $U:= V_{x_1} + \hat{\psi}$, and observe that $U \in \widehat{\mathcal{K}}$. 
With this choice of $U$,  the inequality \eqref{eq variational inequality integrated by parts} rewrites as 
\begin{equation*}
    0 \leq \int_B \hat{L}(U-V_{x_1}) \psi^2 dx = - \eta \int_{ \mathbb{R}^{2} } \hat{L}^2 \psi^2 \mathds{1}_{\widehat{\mathcal{W}}_\varepsilon} \mathds{1}_{\{ |\hat{L}| < N \}  } dx, 
\end{equation*}
which in turn implies that $\int_{ \mathbb{R}^{2} } \hat{L}^2 \psi^2 \mathds{1}_{\widehat{\mathcal{W}}_\varepsilon} \mathds{1}_{\{ |\hat{L}| < N \}  } dx=0$. Taking limits as $N \to \infty$ and $\varepsilon \to 0$, by monotone convergence theorem, we conclude that $\int_{\mathcal{W} } \hat{L}^2 \psi^2 dx =0 $; that is, $\hat{L}=0$ a.e.\ in $\mathcal{W}$. 

Finally, defining the two regions 
\begin{equation}
\label{definition of I+ and I-}
\mathcal{I}_-:=\{ x \in \mathbb{R}^{2}\, |\, V_{x_1}(x)=-1 \}\quad \text{and} \quad\mathcal{I}_+ :=\{ x \in \mathbb{R}^{2}\, |\, V_{x_1}(x)=1 \},
\end{equation}
we can repeat the arguments above with $\hat{\psi}:= - \eta \hat{L}^+ \mathds{1}_{\mathcal{I}_+ } \mathds{1}_{\{ |\hat{L}| < N \}  }$ and $\hat{\psi}:= - \eta \hat{L}^- \mathds{1}_{\mathcal{I}_-} \mathds{1}_{\{ |\hat{L}| < N \}  }$, in order to deduce that $\hat{L} \leq 0 $ a.e.\ in $\mathcal{I}_{+} \cup \mathcal{I}_-$, and thus completing the proof of the theorem. 
\end{proof} 

Theorem \ref{theorem connection Dynkin game PDE} allows to provide a probabilistic representation of $V_{x_1}$ in terms of a Dynkin game. 
Let $\mathcal{T}$ be the set of $\mathbb{F}$-stopping times, and, for $\tau_1, \tau_2 \in \mathcal{T}$,  define the functional 
$$ 
G(x; \tau_1,\tau_2):= \mathbb{E} \bigg[ \int_0^{\tau_1 \land \tau_2} e^{-\hat{\rho} t}  \hat{h}(X_t^x)   dt 
- e^{-\hat{\rho} \tau_1 } \mathds{1}_{ \{\tau_1 \leq \tau_2,\, \tau_1 < \infty \} } 
+ e^{-\hat{\rho} \tau_2}  \mathds{1}_{ \{\tau_2 < \tau_1 \} } \bigg], 
$$ 
where $\hat{h}=h_{x_1} +b_{x_1}^2 V_{x_2}$ (cf.\ Theorem \ref{theorem connection Dynkin game PDE}), the process $X^x$ denotes the solution to the uncontrolled SDE \eqref{SDE dynamics uncontrolled}, and $\hat{\rho}:=\rho-b_{1}^1$.
Consider the 2-player stochastic differential game of optimal stopping in which Player 1 (resp.\ Player 2) is allowed to choose  a stopping time $\tau_1$ (resp.\ $\tau_2$) in order to maximize (resp.\ minimize) the functional $G$.
 
Recalling the definitions of $\mathcal{I}_-$ and $\mathcal{I}_+$ given in \eqref{definition of I+ and I-}, from Theorem \ref{theorem connection Dynkin game PDE} we obtain the following verification theorem. 
Its proof is based on a generalized version of It\^o's formula (see Theorem 1 at p.\ 122 in \cite{krylov1980}) which can be applied to the process $( e^{-\hat{\rho} t} V_{x_1}(X_t^x) )_{t\geq 0 }$ because $V_{x_1} \in W_{loc}^{2;\infty}(\mathbb{R}^{2})$ by Theorem \ref{theorem connection Dynkin game PDE}. 
Since these arguments are standard, we omit the details in the interest of length. 
\begin{theorem} 
\label{theorem Dynkin game connection} 
For each $x\in \mathbb{R}^{2}$, the profile strategy $(\bar{\tau}_1, \, \bar{\tau}_2)$ given by the stopping times  
$$
\bar{\tau}_1 := \inf \{ t \geq 0 \, | \,  X_t^x \in  \mathcal{I}_- \} \ \ \text{and} \ \ \bar{\tau}_2 := \inf \{ t \geq 0 \,  | \,  X_t^x \in \mathcal{I}_+ \}
$$ 
is a saddle point of the Dynkin game, and its corresponding value equals $V_{x_1}(x)$; that is, 
$$
 G(x;{\tau}_1,  \bar{\tau}_2) \leq V_{x_1}(x)=G(x;\bar{\tau}_1,  \bar{\tau}_2)  \leq G(x;\bar{\tau}_1,  {\tau}_2), \quad \text{for each } \tau_1, \tau_2 \in \mathcal{T}.
$$ 
Moreover, we have
\begin{equation}\label{eq sup inf}
    V_{x_1}(x)= \sup_{\tau_1 } \inf_{\tau_2 } G(x;{\tau}_1,  {\tau}_2) =  \inf_{\tau_2} \sup_{\tau_1 } G(x;{\tau}_1,  {\tau}_2). 
\end{equation}
\end{theorem}

\subsubsection{The monotonicity property} 
We now show how Condition \ref{assumption on b} in Assumption \ref{assumption} together with Theorems \ref{theorem connection Dynkin game PDE} and \ref{theorem Dynkin game connection} lead to an important monotonicity of $V_{x_1}$.
\begin{proposition}
\label{proposition V_x1 nondecreasing in x2} We have $ b_{x_1}^2  V_{x_1 x_2} \geq 0$ in $\mathbb{R}^{2}$. 
\end{proposition} 
\begin{proof} 
Since $b_{x_1}^2 \leq 0$ by Condition \ref{assumption on b} in Assumption \ref{assumption}, it is enough to show that $ V_{x_1 x_2} \leq 0$.
Fix an initial condition $x \in \mathbb{R}^{2}$, take  $r > 0$, and define a new initial condition $x^r \in \mathbb{R}^{2}$ by setting   
$ 
x^r:= x + r e_2 .
$ 
Let $X^{x^r}=(X^{1,x^r},X^{2,x^r})$ be the solution to the uncontrolled dynamics \eqref{SDE dynamics uncontrolled}, with initial condition $x^r$.
By the structure we assumed on the drift, this perturbation of the initial condition will affect only the  second component of $X^{x^r}$. 
Indeed, since $x_{2}^r \geq x_{2}$, a standard comparison principle for SDE (see \cite{ikeda1977}) gives 
$ X_t^{2,x^r} -X_t^{2,x} \geq 0$ $ \text{for each } t \geq 0, \ \mathbb{P}\text{-a.s.}
$, while  $X^{1,x^r}=X^{1,x}$.
Hence, since $h_{x_1 x_2} \leq 0$, we have
\begin{equation}
    \label{eq mon h}
h_{x_1}(X_t^{x^r}) \leq h_{x_1}(X_t^{x}), \quad \text{for each } t\geq0, \ \mathbb{P}\text{-a.s.}
\end{equation}
Moreover, since $b_{x_1}^{2} \leq 0$, we can exploit the convexity of $V$  to obtain
\begin{align}
\label{eq mon b.deltaV_x2}
    b_{x_1}^{2}& (X_t^{x^r})(V_{x_{2}}(X_t^{x^r}) - V_{x_{2}}(X_t^{x})) \\ \notag
    & = b_{x_1}^{2} (X_t^{x^r})(X_t^{2,x^r}-X_r^{2,x}) \int_0^1 V_{x_{2} x_{2}} ( X_t^x + s (X_t^{x^r}-X_r^x)) ds \\ \notag & \leq 0, \quad \text{for each } t\geq0, \ \mathbb{P}\text{-a.s.} 
\end{align}
Let us now prove that $V_{x_{2}}(y) \geq 0$, for each $y \in \mathbb{R}^{2}$. Fix $y \in \mathbb{R}^{2}$ and let $v$ be an optimal control for $y$.
Observe that, for each $\delta >0$ we can still employ a comparison principle to deduce that  
$
X_t^{1,y;v} - X_t^{1,y-\delta e_{2};v}  = 0$, and $
X_t^{2,y;v} - X_t^{2,y-\delta e_{2};v} \geq 0$, for each $t\geq0, \ \mathbb{P}\text{-a.s}$.
This, since $h_{x_2} \geq0 $ and $V\in C^1 (\mathbb{R}^{2})$, in turn implies that
\begin{align} 
\label{eq V_x2 >0}
    V_{x_{2}}(y) & =\lim_{\delta \to 0} \frac{V(y)-V(y-\delta e_{2})}{\delta}\\ \notag
    & \geq \lim_{\delta \to 0} \frac{J(y;v)-J(y-\delta e_{2};v)}{\delta}\\ \notag 
    &=\lim_{\delta \to 0} \frac{1}{\delta} \mathbb{E} \bigg[ \int_0^\infty e^{-\rho t} ( h(X_t^{y;v}) -h(X_t^{y-\delta e_{2};v} )) dt \bigg] \geq 0, \notag 
\end{align}
where we have used that the control $v$ is suboptimal for the initial condition $y-\delta e_{2}$.
Hence, since $b_{x_1 x_2}^{2} \leq 0$, we obtain that
\begin{equation} 
    \label{eq mon delta b V_x2}
   ( b_{x_1}^{2}(X_t^{x^r}) -b_{x_1}^{2}(X_t^{x}) ) V_{x_{2}}(X_t^{x}) \leq 0,  \quad \text{for each } t\geq0, \ \mathbb{P}\text{-a.s.}  
\end{equation} 
Summing now the inequalities \eqref{eq mon h}, \eqref{eq mon b.deltaV_x2} and \eqref{eq mon delta b V_x2}, we find
\begin{equation}\label{eq monotonicity hat h}
h_{x_1}(X_t^{x^r})+ b_{x_1}^{2}(X_t^{x^r}) V_{x_{2}}(X_t^{x^r}) \leq h_{x_1}(X_t^{x})+ b_{x_1}^{2}(X_t^{x}) V_{x_{2}}(X_t^{x}),  \quad \text{for each } t\geq 0, \ \mathbb{P}\text{-a.s.};  
\end{equation} 
that is,   
$ \hat{h}(X^{x^r}) \leq  \hat{h}(X^{x})$.
Therefore, for each stopping time $\tau_1,\, \tau_2 \in \mathcal{T}$, we deduce that 
$$
G(x^r;\tau_1,\tau_2) \leq G(x;\tau_1,\tau_2).  
$$
Taking the supremum over $\tau_1 \in \mathcal{T}$ and the infimum over $\tau_2 \in \mathcal{T}$ in the latter inequality, we deduce, in light of \eqref{eq sup inf} in Theorem \ref{theorem Dynkin game connection}, that  
$V_{x_1}(x^r) \leq V_{x_1}(x)$. Hence,  we conclude that $V_{x_1 x_2} \leq 0$ in $\R^2$, which completes the proof of the proposition. 
\end{proof}

\subsection{Step b: Construction of $\varepsilon$-optimal policies}
\label{section epsilon optimal policies}
For every $\varepsilon > 0$ define the sets
\begin{equation*} 
\mathcal{W}_{\varepsilon} := \{ x \in \mathbb{R}^{2} \, | \, V_{x_1}^2(x) < 1-\varepsilon \}, \quad S_\varepsilon := \partial \mathcal{W}_\varepsilon.
\end{equation*} 
In light of Lemma \ref{lemma waiting region nonempty}, the set $\mathcal{W}_\varepsilon$ is clearly nonempty for $\varepsilon$ sufficiently small.

The proof of the following lemma is obtained combining arguments from \cite{kruk2000} together with the monotonicity property shown in Proposition \ref{proposition V_x1 nondecreasing in x2}.
\begin{lemma}
\label{lemma epsilon Skorokhod problems} For each $\varepsilon>0$ such that $\bar{x} \in \mathcal{W}_\varepsilon$, there exists a solution $v^\varepsilon \in \mathcal{V}$ to the (classical) Skorokhod problem for the SDE  \eqref{SDE dynamics} in  $\overline{\mathcal{W}}_\varepsilon$ starting at $\bar{x}$ with reflection direction $-{V_{x_1}}/{|V_{x_1}|}e_1$.  
\end{lemma} 

\begin{proof} 
Fix $\varepsilon>0$ such that $\bar{x} \in \mathcal{W}_\varepsilon$. In order to employ the results of \cite{lions&sznitman1984} to construct $v^\varepsilon$ as the solution of the Skorokhod problem with reflection along $S_\varepsilon$, we first show that $S_\varepsilon$ is a $C^{3}$ hypersurface.
 
To this end, we begin the proof by showing that 
\begin{equation}
    \label{eq V_x1 x1 >0 }
    V_{x_1 x_1} (x)>0, \quad \text{for each }x\in \mathcal{W}.
\end{equation} 
Take indeed $x\in \mathcal{W}$ and  $\delta >0$ such that $B_\delta (x) \subset \mathcal{W}$. Since $V$ solves the linear equation $\rho V-\mathcal{L}V=h$ in $\mathcal{W}$, from Theorem 6.17 at p.\ 109 in \cite{gilbarg2001} it follows that $V \in C^{4}(\mathcal{W})$. 
Therefore, we can differentiate two times with respect to $x_1$ the HJB equation \eqref{HJB v.i.}, and obtain an equation for $V_{x_1 x_1}$
\begin{equation}
\label{eq HJB secord derivative}
   (\rho- 2 b_{1}^1) V_{x_1 x_1} - \mathcal{L} V_{x_1 x_1} = h_{x_1 x_1}  + 2 b_{x_1}^2 V_{x_1 x_2} + b_{x_1 x_1}^2 V_{x_2}, \quad \text{in } B_\delta (x).
\end{equation} 
Since by assumption $h_{x_1 x_1}>0$, thanks to Proposition \ref{proposition V_x1 nondecreasing in x2} we have that $ h_{x_1 x_1}  + 2 b_{x_1}^2 V_{x_1 x_2}>0$. 
By the inequality \eqref{eq V_x2 >0} in the proof of Proposition \ref{proposition V_x1 nondecreasing in x2}, and the fact that  $b^2$ is convex, we deduce that $b_{x_1 x_1}^2 V_{x_2}\geq 0$. 
Therefore, the right hand side of \eqref{eq HJB secord derivative} is positive. 
Next, by the strong maximum principle (see Theorem 3.5 at p.\ 35 in \cite{gilbarg2001}), $V_{x_1 x_1}$ cannot achieve a nonpositive local  minimum in $B_\delta (x)$, unless it is constant. If $V_{x_1 x_1}$ is constant in $B_\delta (x)$, then by \eqref{eq HJB secord derivative} we obtain $V_{x_1 x_1}>0$ as desired. If $V_{x_1 x_1}$ attains its minimum at the boundary $\partial B_\delta (x)$, then by convexity of $V$ we still have 
$$  
V_{x_1 x_1}(y) > \min_{\partial B_\delta (x)} V_{x_1 x_1} \geq 0, \quad \text{for each } y \in B_\delta(x),
$$
which also proves \eqref{eq V_x1 x1 >0 } 

Next, define $\bar{\nu} (x):= V_{x_1}(x) / |V_{x_1}(x)| e_1$ for each $x \in S_\varepsilon$, and $w(y) := |V_{x_1}(y)|^2$ for each $y \in \mathcal{W}$. Notice that
$ \sqrt{w(y)}= |\partial_{\bar{\nu}} V(y) | $. 
{For $R>0$, by compactness of $\overline{\mathcal{W}}_{\varepsilon/2}^\text{\tiny{$R$}}:=\overline{\mathcal{W}}_{\varepsilon/2} \cap \overline{B}_R$, in light of \eqref{eq V_x1 x1 >0 } we can find a constant $c_\varepsilon^\text{\tiny{$R$}} >0$ such that
\begin{equation} 
\label{eq inf V_x1 x1}
    \inf_{x \in \overline{\mathcal{W}}_{\varepsilon/2}^\text{\tiny{$R$}}} V_{x_1 x_1}(x)\geq c_{\varepsilon }^\text{\tiny{$R$}}>0.
\end{equation}
Therefore, for $x \in S_\varepsilon$ and $R$ large enough, by  \eqref{eq inf V_x1 x1}, we have  
\begin{equation*} 
    \sqrt{w(x+\lambda {\bar{\nu}})} =| \partial_{\bar{\nu}} V (x+\lambda {\bar{\nu}}) | \geq  \partial_{\bar{\nu}} V (x+\lambda {\bar{\nu}}) \geq \partial_{\bar{\nu}} V(x) + \lambda c_\varepsilon^\text{\tiny{$R$}} /2 = \sqrt{w(x)} +  \lambda c_\varepsilon^\text{\tiny{$R$}} /2,  
\end{equation*}
and hence 
\begin{equation}
\label{eq gradient non tangential}
\partial_{\bar{\nu}} \sqrt{w(x)} \geq c_\varepsilon^\text{\tiny{$R$}} /2.
\end{equation} 
It thus follows that $\partial_{\bar{\nu}} w \ne 0$ on $S_\varepsilon$. This implies, by the implicit function theorem
,  that $S_\varepsilon$ is a $C^{3}$-hypersurface.  \\ \indent
Now, by \eqref{eq gradient non tangential}, arguing as in Lemma 2.7 in \cite{kruk2000}, we have that the vector $-{\bar{\nu}}$ is not tangential to $S_\varepsilon$, and, by definition of $\mathcal{W}_\varepsilon$ and of ${\bar{\nu}}$, we observe that the vector $-{\bar{\nu}}$ points inside $\mathcal{W}_\varepsilon$. Therefore, we can employ a version of Theorem 4.4 in \cite{lions&sznitman1984} for unbounded domains}  in order to find a solution $v^\varepsilon \in \mathcal{V}$ to the Skorokhod problem for the SDE  \eqref{SDE dynamics} in  $\overline{\mathcal{W}}_\varepsilon$ starting at $\bar{x}$, with reflection direction $-{V_{x_1}}/{|V_{x_1}|}e_1$. 
\end{proof} 

We conclude this section with the following lemma. We omit its proof since this can be established as in the proof of Lemma 2.8 in \cite{kruk2000}.
\begin{lemma}
\label{lemma epsilon optimality}
For each $\bar{x} \in \mathcal{W}$ and $\varepsilon>0$ such that $\bar{x}\in \mathcal{W}_\varepsilon$, let the control $v^\varepsilon$ be as in Lemma \ref{lemma epsilon Skorokhod problems}. Then $J(\bar{x}; v^\varepsilon) \to V(\bar{x})$ as $\varepsilon \to 0$. 
\end{lemma}

\subsection{Step c: Characterization of the optimal control}\label{section proof of the main results}  
Thanks to the results of Subsections \ref{section a connection with Dynkin games} and \ref{section epsilon optimal policies} we can now prove Theorem \ref{theorem main characterization}. We provide a separate proof for each of the two claims.      
\subsubsection{Proof of Claim \ref{theorem x inside waiting}}\label{section proof theorem x inside waiting} 
We will first prove Claim \ref{theorem x inside waiting} for $\bar{x}\in {\mathcal{W}}$, and then, at the end of this subsection, we will give a proof for a general  $\bar{x}\in \overline{\mathcal{W}}$.
Fix $\bar{x}\in {\mathcal{W}}$ and a sequence $(\varepsilon_n)_{n\in \mathbb{N}}$ converging to zero. To simplify the notation, according to Lemma \ref{lemma epsilon Skorokhod problems} we define the processes 
$$
X^n:=X^{\bar{x};v^{\varepsilon_n}}, \ v^n:=v^{\varepsilon_n}, \  \xi^n:= | v^{\varepsilon_n}|, \quad \text{for each } n \in \mathbb{N}.
$$  
Bear in mind that the processes $v^n$ and $\xi^n$ depend on the initial condition $\bar{x}$, and that, according to Lemma \ref{lemma epsilon optimality}, the sequence of controls $(v^n)_{n\in \mathbb{N}}$ is a minimizing sequence; that is, $J(\bar{x}; v^n) \to V(\bar{x})$ as $n \to \infty$.

We begin with the following estimate. 
\begin{lemma}\label{lemma estimate SDE}
Let $p':=(2 p-1)/2$. We have
\begin{equation*}   
\sup_n \int_0^\infty e^{-\rho t} (\mathbb{E} [ |X_t^{1,n}|^{p} ] +\mathbb{E} [ |X_t^{n}|^{{p'}} ] )dt   \leq C (1+|\bar{x}|^p).  
\end{equation*}   
\end{lemma}  

\begin{proof}
Denoting by $X^{\bar{x}}$ the solution to \eqref{SDE dynamics uncontrolled},  a standard use of Gr\"onwall's inequality and of Burkh\"older-Davis-Gundy's inequality leads to the classical estimate 
$$
\mathbb{E}[|X_t^{\bar{x}}|^p] \leq C e^{p\, \bar{L} t}(1+|\bar{x}|^p) \quad \text{for each $t\geq0$,} 
$$  
where $\bar{L}$ is the Lipschitz constant of $\bar{b}$ and $C >0$ is a generic constant.
Therefore, since the control constantly equal to zero is not necessarily optimal for $\bar{x}$, from the latter estimate and the growth rate of $h$ we obtain
\begin{align*}
V(\bar{x}) \leq \mathbb{E} \bigg[ \int_0^\infty e^{-\rho t} h(X_t^{\bar{x}}) dt \bigg] &\leq C \int_0^\infty e^{-\rho t}(1+ \mathbb{E}[|X_t^{\bar{x}}|^p])dt \\
& \leq  C \int_0^\infty e^{ -( \rho-p\bar{L})t}(1+|\bar{x}|^p)dt  \leq  C (1+|\bar{x}|^p),
\end{align*} 
where we have used that, by Condition \ref{ass sigma constant} in Assumption \ref{assumption}, $\rho> p \, \bar{L}$.
Therefore, since $v^n$ is a minimizing sequence, for all $n$ big enough we find the estimate
\begin{equation*}
\kappa_1 \int_0^\infty e^{-\rho t} \mathbb{E} [ |X_t^{1,n}|^{p} ] dt - \kappa_2 \leq J(\bar{x};v^n)  \leq C (1+|\bar{x}|^p),  
\end{equation*}
from which   
\begin{equation}\label{eq estimate xi n}
 \int_0^\infty e^{-\rho t} \mathbb{E} [ |X_t^{1,n}|^{p} ] dt   \leq C ( 1+|\bar{x}|^p ).
\end{equation}

Next, using again Gr\"onwall's inequality and  Burkh\"older-Davis-Gundy's inequality, we find
$$
\mathbb{E}[|X_t^{2,n}|^{p'}] \leq C e^{p' \, \bar{L} t} \bigg( 1+|\bar{x}|^{p'}+p_t + p_t \int_0^t \mathbb{E}[|X_s^{1,n}|^{p'}] ds \bigg), \quad \text{for each $t\geq0$,} 
$$ 
where $p_t$ is a suitable (deterministic) polynomial in $t$, not depending on $n$. 
{Therefore, we can write
\begin{align*}
   \int_0^\infty e^{-\rho t} \mathbb{E}[|X_t^{2,n}|^{p'}] dt & \leq  C \int_0^\infty e^{(p'\bar{L}-\rho) t}(1+|\bar{x}|^{p'}+p_t) dt \\ \notag
   & \quad +C \int_0^\infty e^{[p'\bar{L}-\rho (1-p'/p)] t} p_t \int_0^t e^{-\rho (p'/p) s } \mathbb{E}[|X_s^{1,n}|^{p'}] ds dt. \notag
\end{align*}  
Also, by using  H\"older's inequality with exponent $q=p/p'$, and denoting by $q^*$ the conjugate of $q$, we obtain
$$
\int_0^t e^{-\rho (p'/p) s } \mathbb{E}[|X_s^{1,n}|^{p'}] ds 
\leq \bigg( \int_0^t 1 ds \bigg)^{\frac{1}{q^*}} \bigg( \int_0^t e^{-\rho  s} \mathbb{E}[|X_s^{1,n}|^{p}] ds \bigg)^{ \frac{1}{q}}
= t ^{\frac{1}{q^*}} \bigg( \int_0^t e^{-\rho  s} \mathbb{E}[|X_s^{1,n}|^{p}] ds \bigg)^{ \frac{p'}{p}},
$$ 
so that
\begin{align}\label{eq estimate for X2}
   \int_0^\infty e^{-\rho t} \mathbb{E}[|X_t^{2,n}|^{p'}] dt 
   & \leq  C \int_0^\infty e^{(p'\bar{L}-\rho) t}(1+|\bar{x}|^{p'}+p_t) dt \\ \notag
   & \quad +C \int_0^\infty e^{[p'\bar{L}-\rho (1-p'/p)] t}\, p_t\, t ^{\frac{1}{q^*}} \bigg( \int_0^\infty e^{-\rho  s} \mathbb{E}[|X_s^{1,n}|^{p}] ds \bigg)^{ \frac{p'}{p}} dt.  \notag 
\end{align}
} 
After noticing that Condition \ref{ass sigma constant} in Assumption \ref{assumption} implies  $p'\bar{L}-\rho<0$ and $p'\bar{L}-\rho (1-p'/p)<0$, using  \eqref{eq estimate xi n} in \eqref{eq estimate for X2}, we conclude that
\begin{equation*}
    \sup_n \int_0^\infty e^{-\rho t} \mathbb{E}[|X_t^{2,n}|^{p'}] dt \leq  C (1+|\bar{x}|^p),
\end{equation*}
which, together with \eqref{eq estimate xi n} {and the fact that $p'<p$,} completes the proof of the lemma.  
\end{proof}

\begin{lemma}
\label{lemma construction optimal policy} Let $\bar{v}\in \mathcal{V}$ be the unique optimal control for $\bar{x}$. We have that
$$ 
X_t^n \to X_t^{\bar{x};\bar{v}} \quad \text{and} \quad v^n \to \bar{v}, \quad \mathbb{P}\otimes dt\text{-a.e.\ in } \Omega \times [0,\infty), \quad \text{as } n\to \infty. 
$$
\end{lemma}  
\begin{proof} The proof employs arguments as those in the proof of Theorem 8 in \cite{Menaldi&Taksar89}, that however need to be suitably adapted  in order to accommodate our more general convex setting. 

We organize the proof in two steps.
\smallbreak \noindent 
\emph{Step 1.}
Arguing by contradiction, in this step we prove that the sequence $X^n$ is Cauchy w.r.t.\ the convergence in the measure $\mathbb{P}\otimes e^{-\rho t}dt$; that is, for each $\delta >0$ we have
\begin{equation}
    \label{eq cauchy sequence X in measure}
    \mathbb{E}\bigg[ \int_0^\infty e^{-\rho t}\mathds{1}_{ \{|X_t^n -X_t^m|>\delta \} } dt \bigg] \to 0, \quad \text{as }n,m\to \infty. 
\end{equation} 
{Observe first that, if the sequence $X^{1,n}$ is Cauchy w.r.t.\ the convergence in the measure $\mathbb{P} \otimes e^{-\rho t}dt$, then it converges (in the measure $\mathbb{P}\otimes e^{-\rho t}dt$) to a process $X^1$.  
Then, we can employ the estimate in Lemma \ref{lemma estimate SDE} in order to deduce that the sequence
$X^{1,n}$ converges in $\mathbb L ^q (\Omega \times [0, \infty) ;\mathbb{P}\otimes e^{-\rho t}dt)$, for some $3/2<q<p$, to the same process $X^1$, so that it is Cauchy in $\mathbb L ^q (\Omega \times [0, \infty) ;\mathbb{P}\otimes e^{-\rho t}dt)$.
Moreover, similarly to \eqref{eq estimate for X2}, we can set $q':= (2q -1)/2$ and use Gr\"onwall's inequality to obtain
\begin{align*} 
   \int_0^\infty e^{-\rho t} \mathbb{E}[|X_t^{2,n} -X_t^{2,m}|^{q'}] dt 
   & \leq C \int_0^\infty e^{[q'\bar{L} - \rho (1-q'/q)] t} p_t t^{\frac{1}{q*}} \bigg( \int_0^\infty e^{-\rho  s} \mathbb{E}[|X_s^{1,n} - X_s^{1,m}|^{q}] ds \bigg)^{ \frac{q'}{q}} dt,  \notag 
\end{align*}  
where $p_t$ is a suitable (deterministic) polynomial in $t$, not depending on $n$, and where $q^*$ denotes the conjugate of $p/p'$. 
Notice that, for $q$ sufficiently close to $p$,  Condition \ref{ass sigma constant} in Assumption \ref{assumption} implies  that $q'\bar{L}-\rho (1-q'/q)<0$, so that the right-hand side of the latter display inequality converges to $0$ as $n,m \to \infty$.
Therefore, since $q'>1$, we obtain that the sequence $X^{2,n}$ is Cauchy in $\mathbb L ^1 (\Omega \times [0, \infty) ;\mathbb{P}\otimes e^{-\rho t}dt)$, which in turn implies that it is Cauchy w.r.t.\ the convergence in the measure $\mathbb{P}\otimes e^{-\rho t}dt$.    
In conclusion, if the sequence $X^{1,n}$ is Cauchy w.r.t.\ the convergence in the measure $\mathbb{P} \otimes e^{-\rho t}dt$, then the sequence $X^{n}$ is Cauchy w.r.t.\ the convergence in the measure $\mathbb{P} \otimes e^{-\rho t}dt$.
For this reason, in order to contradict \eqref{eq cauchy sequence X in measure}, we assume that, for some $\delta >0$ and for a subsequence (not relabelled), one has 
\begin{equation}\label{eq absurde delta xn xm}
    \mathbb{E}\bigg[ \int_0^\infty e^{-\rho t}\mathds{1}_{ \{|X_t^{1,n} -X_t^{1,m}|>\delta \} } dt \bigg] \geq \delta_0>0, \quad \text{for each }n,m\ \in \mathbb{N}, 
\end{equation}
for a certain constant $\delta_0 >0$.} 

Fix $\lambda \in (0,1)$. We begin by defining the processes
$$
Y^{n,m}:= X^{\bar{x};  \lambda v^n+(1-\lambda) v^m} \quad \text{and} \quad Z^{n,m} := \lambda X^n+(1-\lambda) X^m, \quad \text{for each }n,m\ \in \mathbb{N}. 
$$ 
We first need to show that 
\begin{equation}
    \label{eq X < Z}
    Y_t^{n,m} \leq  Z_t^{n,m}, \quad \text{for each } t\geq 0, \ \mathbb{P}\text{-a.s.}
\end{equation}
Since the drift $\bar{b}^1$ is affine, we have  $Y^{1;n,m} = Z^{1;n,m}$.
Moreover, since $b^{2}$ is convex,  we find 
\begin{align*}
Z_t^{2;n,m} & = \bar{x}_{2} +\int_0^t ( \lambda b^{2}( X_s^{n}) + (1-\lambda) b^{2}( X_s^{m} )) dt + \sigma W_t^{2} \\ \notag & \geq \bar{x}_{2} +\int_0^t b^{2}(Z_s^{1;n,m},Z_s^{2;n,m} ) dt + \sigma W_t^{2}   \\ \notag 
& = \bar{x}_{2} +\int_0^t b^{2}(Y_s^{1;n,m},Z_s^{2;n,m} ) dt + \sigma W_t^{2},
\end{align*}
while $Y_t^{2;n,m}= \bar{x}_{2} + \int_0^t b^{2}(Y_s^{1;n,m},Y_s^{{2};n,m}) ds + \sigma W_t^{2}$.  This, by the comparison principle for SDE (see \cite{ikeda1977}), implies that $ Y_t^{{2};n,m} \leq Z_t^{2;n,m}, \ \text{for each } t\geq 0, \ \mathbb{P}\text{-a.s.}$, and \eqref{eq X < Z} follows.

Next, in light of \eqref{eq X < Z}, by the monotonicity of $h$ in $x_2$ we find 
\begin{align} 
    \label{eq comb convex J first}
     \lambda J(\bar{x}; v^n) &+ (1-\lambda) J(\bar{x}; v^m) - J(\bar{x}; \lambda v^n+(1-\lambda) v^m)\\ \notag  
    &= \mathbb{E}\bigg[ \int_0^\infty e^{-\rho t} (\lambda h(X_t^n)+(1-\lambda) h(X_t^m) - h(Y_t^{n,m}) ) dt \\ \notag
    & \quad \quad \quad \quad + \int_{[0,\infty)} e^{-\rho t} (\lambda d|v^n|_t + (1-\lambda) d|v^m|_t -d |\lambda v^n+(1-\lambda) v^m|_t) \bigg] \\ \notag  
    & \geq \mathbb{E}\bigg[ \int_0^\infty e^{-\rho t} (\lambda h(X_t^n)+(1-\lambda) h(X_t^m) - h(Z_t^{n,m}) ) dt \bigg], 
\end{align}  
as we have that $|\lambda v^n+(1-\lambda) v^m|_t  \leq \lambda |v^n|_t + (1-\lambda) |v^m|_t$, and that $e^{-\rho t}$ is positive and decreasing.

Then, using \eqref{eq absurde delta xn xm}, for $M>0$ we observe that   
\begin{align*}
&\mathbb{E}\bigg[ \int_0^\infty e^{-\rho t}\mathds{1}_{ \{|X_t^{1,n} -X_t^{1,m}|>\delta \} }\mathds{1}_{ \{|X_t^n| \leq M, \, |X_t^m|\leq M \} } dt \bigg] \\
&\geq \delta_0- \mathbb{E}\bigg[\int_0^\infty e^{-\rho t}\mathds{1}_{ \{|X_t^n| > M \} }dt \bigg]- \mathbb{E}\bigg[\int_0^\infty e^{-\rho t}\mathds{1}_{ \{|X_t^m| > M \} }dt \bigg].
\end{align*}
Moreover, the estimate in Lemma \ref{lemma estimate SDE} and an application of Chebyshev's inequality  yield
 $$
 \mathbb{E} \bigg[  \int_0^\infty e^{-\rho t} \mathds{1}_{ \{ |X_t^n|>M \} } dt \bigg] \leq \frac{C (1+|\bar{x}|^p)}{M^{p'}},  \quad \text{for each }n\in \mathbb{N}, 
 $$
so that we can find $M$ big enough such that
$$
\mathbb{E}\bigg[ \int_0^\infty e^{-\rho t}\mathds{1}_{ \{|X_t^{1,n} -X_t^{1,m}|>\delta \} }\mathds{1}_{ \{|X_t^n| \leq M, \, |X_t^m|\leq M \} } dt \bigg] \geq \frac{\delta_0}{2}, \quad \text{for each }n,m\in \mathbb{N}.
$$ 
Combining the latter inequality with \eqref{eq comb convex J first}, we obtain
\begin{align}
\label{eq comb convex J 1}
   \lambda J(\bar{x}; v^n) &+ (1-\lambda) J(\bar{x}; v^m) - J(\bar{x}; \lambda v^n+(1-\lambda) v^m) \\ \notag
   & \geq \delta_M\mathbb{E}\bigg[ \int_0^\infty e^{-\rho  t}\mathds{1}_{ \{|X_t^{1,n} -X_t^{1,m}|>\delta \} }\mathds{1}_{ \{|X_t^n| \leq M, \, |X_t^m|\leq M \} } dt \bigg] \\ \notag
    & \geq \delta_M\frac{\delta_0}{2}, 
\end{align}
where, by strict convexity of $h$ in the variable $x_1$, we have
$$
\delta_M := \inf \big\{ \lambda h(x)+(1-\lambda) h(y) - h(\lambda x+(1-\lambda) y ) \big| \,  |x_1-y_1| > \delta, |x|,|y|\leq M \big\} >0. 
$$
On the other hand, by Lemma \ref{lemma epsilon optimality}, $J(\bar{x};v^n)$ converges to $V(\bar{x})$ as $n \to \infty$. Therefore, from \eqref{eq comb convex J 1}, we can find $\bar{n}\in\mathbb{N}$ such that
$$
V(\bar{x})\geq \delta_M \frac{\delta_0}{4} + J(\bar{x};\lambda v^n+(1-\lambda) v^m), \quad \text{for each }n,m\geq \bar{n}, 
$$
which contradicts the definition of $V$, completing the proof of \eqref{eq cauchy sequence X in measure}. 
\smallbreak \noindent 
\emph{Step 2.}
By the previous step, there exists a limit process $\hat{X}$ and, unless to consider a subsequence, we can assume that   
\begin{equation}
\label{eq pointwise convergence of X}
X_t^n \to \hat{X}_t\quad  \mathbb{P}\otimes dt\text{-a.e. in } \Omega \times [0,\infty), \quad \text{as } n\to \infty. 
\end{equation} 
Next, defining the process $v_t:=\hat{X}_t^1 -\bar{x}^1- \int_0^t \bar{b}^1(\hat{X}_s^1)ds - \sigma W_t^1$, using the estimate from Lemma \ref{lemma estimate SDE} and \eqref{eq pointwise convergence of X} we find  
\begin{equation*}
|v_t^n -v_t| \leq |X_t^{1,n} - \hat{X}_t^1| + \bar{L} \int_0^t|X_s^{1,n} - \hat{X}_s^1| ds \to 0 \quad  \mathbb{P}\otimes dt\text{-a.e. in } \Omega \times [0,\infty), 
\end{equation*}
which implies that 
\begin{equation}
\label{eq pointwise convergence of v} 
v_t^n \to v_t\quad  \mathbb{P}\otimes dt\text{-a.e. in } \Omega \times [0,\infty), \ \text{as } n\to \infty.
\end{equation}
We also observe that, by using Lemma 3.5 in \cite{K}, we can assume the processes $\hat{X}^1$ and $v$ to be right-continuous. 
Also, denoting with $\xi$ the total variation of $v$, from \eqref{eq pointwise convergence of v} we easily find
\begin{equation}
\label{eq pointwise convergence of total variation}
\xi_t \leq \liminf_n \xi_t^n \quad  \text{for each } t \geq 0.
\end{equation}  
Next, exploiting the limits in \eqref{eq pointwise convergence of X}, the Lipschitz continuity of $b^2$ and the estimate from Lemma \ref{lemma estimate SDE}, we can see that the process $\hat{X}^2$ is continuous and it solves the SDE 
$
d\hat{X}_t^{2}=b^2(\hat{X}_t^{1},\hat{X}_t^{2})dt + \sigma dW_{t}^2, \ t \geq 0,\ \hat{X}_{0-}^{2}=\bar{x}_2
$. This, together with the definition of $v$, implies that
\begin{equation}\label{eq hat X = X}
\hat{X}=X^{\bar{x};v}.    
\end{equation}

Finally, thanks to the limits in \eqref{eq pointwise convergence of X}, \eqref{eq pointwise convergence of v} and \eqref{eq pointwise convergence of total variation}, to the identity \eqref{eq hat X = X}, and to the continuity of $h$, we invoke Fatou's lemma and, with an integration by parts (see, e.g., Corollary 2 at p.\ 68 in \cite{Protter05}), we find  
\begin{align}\label{eq J < limi inf J^n}
J(\bar{x};v) &=   \mathbb{E} \bigg[ \int_0^\infty e^{-\rho t}  h(X_t^{\bar{x};v}) dt +\rho \int_0^\infty e^{-\rho t} \xi_t dt  \bigg] \\ \notag
& \leq \liminf_n \mathbb{E} \bigg[ \int_0^\infty e^{-\rho t} h(X_t^n) dt + \rho \int_0^\infty e^{-\rho t} \xi_t^n dt  \bigg] =\liminf_n J(\bar{x};v^n) =V(\bar{x}), 
\end{align}
where we have used that the sequence $(v^n)_{n\in \mathbb{N}}$ is minimizing for $\bar{x}$, according to Lemma \ref{lemma epsilon optimality}.
Thus, the process $v$ has locally bounded variation, and $v \in \mathcal{V}$. Also, from \eqref{eq J < limi inf J^n} we deduce that the control $v$ is optimal for $\bar{x}$, and, by uniqueness of optimal controls (see Remark \ref{lemma appendix existence optimal controls}), we conclude that $v=\bar{v}$ and $\hat{X}=X^{\bar{x};\bar{v}}$, completing the proof of the lemma. 
\end{proof}  

The proofs of the next two propositions follow by employing arguments similar to those employed in Sections 2.3 and 2.4 in \cite{kruk2000} (we provide details here in order to recall these arguments in the sequel).
\begin{proposition}
\label{prop X lives in W} We have that $\mathbb{P}[X_t^{\bar{x};\bar{v}}\in \overline{\mathcal{W}}, \, \forall t \geq 0 ]=1$.
\end{proposition}
\begin{proof} By Lemma \ref{lemma construction optimal policy}, $X_t^n \to X_t^{\bar{x};\bar{v}}$,   $\mathbb{P}\otimes dt $-a.e.\ in $\Omega \times [0,\infty)$, and, by Lemma \ref{lemma epsilon Skorokhod problems}, $\mathbb{P}[X_t^{n}\in \overline{\mathcal{W}}, t\geq 0 ]=1$, as $\overline{\mathcal{W}}_{\varepsilon_n} \subset \overline{ \mathcal{W}}$ for each $n\in \mathbb{N}$. Therefore, it is clear that $X_t^{\bar{x};\bar{v}} \in \overline{\mathcal{W}}$, $\mathbb{P}\otimes dt $-a.e.\ in $\Omega \times [0,\infty)$, which, by right-continuity, implies that $\mathbb{P}[X_t^{\bar{x};\bar{v}}\in \overline{\mathcal{W}}, t\geq 0 ]=1$.
\end{proof}

\begin{proposition}\label{proposition the cont parts acts on the boundary}
We have $d\bar{v}=\bar{\gamma} d|\bar{v}|$ with  
$$
|\bar{v}|_t=\int_0^t \mathds{1}_{ \{X_{s-}^{\bar{x};\bar{v}} \in S, \,-V_{x_1}(X_{s-}^{\bar{x};\bar{v}}) = \bar{\gamma}_s \} } d |\bar{v}|_s, \quad \text{for each $t\geq 0, \ \mathbb{P}$-a.s.} $$
\end{proposition}
\begin{proof}
Take $R>0$ such that $\bar{x} \in B_R$ and define $\tau_R:= \inf \{ t \in [0,\infty) | X_s^{\bar{x};\bar{v}} \notin B_R \}$. For each $\varepsilon >0$, let $V^\varepsilon$ be as in \eqref{eq control problem penalized}. As in Step 1 in the proof of Theorem \ref{theorem V sol HJB} in Appendix \ref{appendix auxiliary results}, $V^\varepsilon$ is a convex $C^2$-solution to \eqref{eq appendix penalized HJB}. By It\^o's formula for semimartingales (see, e.g., Theorem 33 at p.\ 81  in  \cite{Protter05}), applied on the process $(e^{-\rho t} V^\varepsilon(X_t^{\bar{x};\bar{v}}))_{t \geq 0}$ on the time interval $[0,\tau_R]$, we find 
\begin{align*}  
    \mathbb{E} [e^{-\rho \tau_R} V^\varepsilon(X_{\tau_R}^{\bar{x};\bar{v}}) ] = V^\varepsilon(\bar{x}) &+ \mathbb{E}\bigg[ \int_0^{\tau_R} e^{-\rho t}(\mathcal{L} V^\varepsilon  - \rho V^\varepsilon) (X_t^{\bar{x};\bar{v}}) dt +  \int_{[0,\tau_R)} e^{-\rho t} V_{x_1}^\varepsilon(X_{t-}^{\bar{x};\bar{v}}) \bar{\gamma}_t d|\bar{v}|_t \\
    & + \sum_{0 \leq t \leq \tau_R } e^{-\rho t}(V^\varepsilon(X_{t}^{\bar{x};\bar{v}})- V^\varepsilon(X_{t-}^{\bar{x};\bar{v}}) - V_{x_1}^\varepsilon(X_{t-}^{\bar{x};\bar{v}}) \bar{\gamma}_t (|\bar{v}|_t - |\bar{v}|_{t-})) \bigg].
\end{align*}  
By the convexity of $V^\varepsilon$, the last sum above is nonnegative.
Also, since the function $\beta$ in \eqref{eq appendix penalized HJB} in nonnegative, we have $\rho V^\varepsilon-\mathcal{L}V^\varepsilon \leq h$ a.e.\ in $\mathbb{R}^{2}$. 
{
Hence from the latter equality we have
\begin{equation*}
V^\varepsilon(\bar{x})  \leq \mathbb{E} \bigg[e^{-\rho \tau_R} V^\varepsilon(X_{\tau_R}^{\bar{x};\bar{v}}) + \int_0^{\tau_R} e^{-\rho t} h(X_{t}^{\bar{x};\bar{v}}) dt - \int_{[0,{\tau_R})} e^{-\rho t} V_{x_1}^\varepsilon(X_{t-}^{\bar{x};\bar{v}}) \bar{\gamma}_t d |\bar{v}|_t \bigg].
\end{equation*} 
Taking limits as $\varepsilon \to 0$ (using the monotonicity of $V^\varepsilon$ in $\varepsilon$  as in \eqref{eq appendix penalized HJB} and the monotone convergence theorem, together with \eqref{eq appendix Sobolev limits} and the dominated convergence theorem), we obtain
\begin{equation}
\label{eq suboptimality epsilon}
V(\bar{x})  \leq \mathbb{E} \bigg[e^{-\rho \tau_R} V(X_{\tau_R}^{\bar{x};\bar{v}}) + \int_0^{\tau_R} e^{-\rho t} h(X_{t}^{\bar{x};\bar{v}}) dt - \int_{[0,{\tau_R})} e^{-\rho t} V_{x_1}(X_{t-}^{\bar{x};\bar{v}}) \bar{\gamma}_t d |\bar{v}|_t \bigg]. 
\end{equation}
Moreover, by the tower rule and strong Markov property, we find
\begin{equation*} 
\lim_{R \to \infty} \mathbb [ e^{-\rho\tau_R} V( X_{\tau_R}^{\bar{x};\bar{v}})] \leq \lim_{R \to \infty} \bigg( V(\bar{x}) - \mathbb{E} \bigg[ \int_0^{\tau_R} e^{-\rho t} h( X_t^{\bar{x};\bar{v}} ) dt + \int_{[0,\tau_R]} e^{-\rho t} d|\bar{v}|_t  \bigg] \bigg) =0,  
\end{equation*}
so that, taking limits in \eqref{eq suboptimality epsilon} as $R \to \infty$ (using the monotone convergence theorem and the dominated convergence theorem), we obtain
\begin{equation} 
\label{eq suboptimality}
V(\bar{x})  \leq \mathbb{E} \bigg[ \int_0^{\infty} e^{-\rho t} h(X_{t}^{\bar{x};\bar{v}}) dt - \int_{[0,\infty)} e^{-\rho t} V_{x_1}(X_{t-}^{\bar{x};\bar{v}}) \bar{\gamma}_t d |\bar{v}|_t \bigg].
\end{equation}}
Next, by the optimality of $\bar{v}$, we have that $V(\bar{x}) = J(\bar{x};\bar{v})$, and, from \eqref{eq suboptimality}, it follows that
\begin{equation}\label{eq to compare support of v}
 \mathbb{E} \bigg[  \int_{[0,\infty)} e^{-\rho t} ( 1 + V_{x_1}(X_{t-}^{\bar{x};\bar{v}}) \bar{\gamma}_t ) d |\bar{v}|_t \bigg] \leq 0.
\end{equation}
This in turn implies, {using  $|\bar \gamma _t| = 1$ and $0\leq 1-|V_{x_1}| \leq 1+V_{x_1} \gamma$ for all $\gamma \in \mathbb{R}$ with $|\gamma|=1$}, that  
$$
0 \leq \mathbb{E} \bigg[  \int_{[0,\infty)} e^{-\rho t}(1 -|V_{x_1}(X_{t-}^{\bar{x};\bar{v}})| ) d|\bar{v}|_t \bigg] \leq \mathbb{E} \bigg[  \int_{[0,\infty)} e^{-\rho t} (1+V_{x_1}(X_{t-}^{\bar{x};\bar{v}}) \bar{\gamma}_t  ) d|\bar{v}|_t \bigg] \leq 0.
$$ 
From the latter chain of inequalities we deduce that the support of the random measure $d |\bar{v}|$ is $\mathbb{P}$-a.s.\ contained in the set $\{ (\omega,t) \in \Omega \times [0,\infty)  \, |\,  X_{t-}^{\bar{x};\bar{v}}(\omega) \in \partial \mathcal{W}, \bar{\gamma}_t(\omega) = -V_{x_1}(X_{t-}^{\bar{x};\bar{v}}(\omega)) \} $, which completes the proof of the proposition. 
\end{proof}

The proof of the next proposition also follows by employing the arguments in \cite{kruk2000}. Details are provided in Appendix \ref{appendix proof propositions kruk} for the sake of completeness. 
\begin{proposition}
\label{propo charact jumps} We have that,
$\mathbb{P}$-a.s., a possible jump of the process $X^{\bar{x};\bar{v}}$ at time $t \geq 0$ occurs on some interval $I\subset \partial \mathcal{W}$ parallel to the vector field $-V_{x_1} e_1$, i.e., such that $-V_{x_1}(x)e_1$ is parallel to $I$ for each $x\in I$. If $X^{\bar{x};\bar{v}}$ encounters such an interval $I$, it instantaneously jumps to its endpoint in the direction $-V_{x_1} e_1$ on $I$.  
\end{proposition}

Combining then the Propositions \ref{prop X lives in W}, \ref{proposition the cont parts acts on the boundary} and \ref{propo charact jumps}, we see that, for $\bar{x} \in {\mathcal{W}}$, the optimal control $\bar{v} \in \mathcal{V}$ is a solution to the modified Skorokhod problem for the SDE \eqref{SDE dynamics} in $\overline{\mathcal{W}}$ starting at $\bar{x}$ with reflection direction $-V_{x_1} e_1$.

Take next $\bar{x} \in \overline{\mathcal{W}}$. By definition, there exists a sequence $(x^k)_{k \in \mathbb{N}} \subset \mathcal{W}$ such that $x^k\to \bar{x}$ as $k\to \infty$. For each $k$, let $w^k$ be the optimal control for $x^k$, and consider the controls 
$x^k -\bar{x} + w^k$, which consist in  following the policy $w^k$ after an initial jump from $\bar{x}$ to $x^k$.
Using the fact that $x^k \in \mathcal{W}$, from Proposition \ref{prop X lives in W} we have that $\mathbb{P}[X_t^{x^k;w^k}\in \overline{\mathcal{W}}, t\geq 0 ]=1$.
Observe, moreover, that $X^{x^k; w^k} = X^{\bar{x} ; x^k -\bar{x} + w^k}$, and that
$|J(\bar{x}; x^k -\bar{x} + w^k) - J(x^k;w^k)| = |\bar{x} -x^k|$.
By the continuity of $V$, we now see that
\begin{equation*}
    V(\bar{x}) = \lim_k V(x^k) = \lim_k J(x^k;w^k) = \lim_k J(\bar{x}; x^k -\bar{x} + w^k).  
\end{equation*}
Therefore, the sequence of controls $(x^k -\bar{x} + w^k)_{k\in \mathbb{N}}$ is a minimizing sequence for the initial condition $\bar{x}$. Repeating the proof of Lemma \ref{lemma construction optimal policy} with the sequence of controls $(x^k-\bar{x} + w^k)_{k\in \mathbb{N}}$, we see that
$ X_t^{x^k; w^k} \to X_t^{\bar{x}; \bar{v}} $, $\mathbb{P}\otimes dt $-a.e.\ in $\Omega \times [0,\infty)$. This allows to repeat the arguments in the proofs of Propositions \ref{prop X lives in W}, \ref{proposition the cont parts acts on the boundary} and \ref{propo charact jumps} in order to conclude that, also for $\bar{x} \in \overline{\mathcal{W}}$, the optimal control $\bar{v} \in \mathcal{V}$ is a solution to the modified Skorokhod problem for the SDE \eqref{SDE dynamics} in $\overline{\mathcal{W}}$ starting at $\bar{x}$ with reflection direction $-V_{x_1} e_1$.

Finally, through a verification theorem (which can be proved by using It\^{o}'s formula as in the proof of Proposition \ref{proposition the cont parts acts on the boundary}), it is easy to show that any solution to the modified Skorokhod problem for the SDE \eqref{SDE dynamics} in $\overline{\mathcal{W}}$ starting at $\bar{x}$ with reflection direction $-V_{x_1} e_1$ is an optimal control. This, by uniqueness of the optimal control (see Remark \ref{lemma appendix existence optimal controls}) implies that such a solution is unique,  completing the proof of Claim \ref{theorem x inside waiting} of Theorem \ref{theorem main characterization}.   
 
\subsubsection{Proof of Claim \ref{theorem x outside waiting}}\label{subsection proof x outside waiting} Fix  $\bar{x}=(\bar{x}_1, \bar{z}) \notin \overline{\mathcal{W}}$ and denote again by $\bar{v}$ the optimal control for $\bar{x}$. Let $\bar{y}_1 \in \mathbb{R}$ be the metric projection of $\bar{x}_1$ into the set $\overline{\mathcal{W}}_1(\bar{z})$. The set $\overline{\mathcal{W}}_1(\bar{z})$ is a closed  interval (cf.\ Lemma \ref{lemma waiting region nonempty}), hence the point $\bar{y}_1$ is uniquely determined. 
Set then $\bar{y}:= (\bar{y}_1,\bar{z})$ and observe that $\bar{y} \in \partial \mathcal{W}$. Let $\bar{w}$ be the optimal control for $\bar{y}$.
Notice that, since $V_{x_1}$ is pointing outside $\overline{\mathcal{W}}_1(\bar{z})$, we have $ V_{x_1}(\bar{y}) (\bar{x}_1 -\bar{y}_1 ) = |\bar{x}_1 -\bar{y}_1| $. 
Therefore, since $(\bar{y}_1 + \lambda (\bar{x}_1 - \bar{y}_1), \bar{z}) \notin \mathcal{W}$ for each $\lambda  \in (0,1)$, we get
\begin{align*}
V(\bar{x})= V(\bar{y}_1, \bar{z}) + \int_0^1 V_{x_1}(\bar{y}_1 + \lambda (\bar{x}_1 - \bar{y}_1), \bar{z}) (\bar{x}_1 -\bar{y}_1) d\lambda =  V(\bar{y}) +  |\bar{x}_1 -\bar{y}_1|. 
\end{align*}
This means that $V(\bar{x}) = J(\bar{y};\bar{w}) + |\bar{x}_1 -\bar{y}_1| = J(\bar{x}; \bar{x}_1 -\bar{y}_1 +\bar{w})$, which, by uniqueness of the optimal control, implies that $\bar{v}=\bar{x}_1 -\bar{y}_1 +\bar{w}$. Moreover, since $\bar{y} \in \overline{\mathcal{W}}$ and $\bar{w}$ is optimal for $\bar{y}$, by Claim \ref{theorem x inside waiting} we have that  $\bar{w}$ is the unique  solution to the modified Skorokhod problem for the SDE \eqref{SDE dynamics} in $\overline{\mathcal{W}}$ starting at $\bar{y}$ with reflection direction $-V_{x_1}e_1$.  This completes the proof of Claim \ref{theorem x outside waiting} and therefore also of Theorem \ref{theorem main characterization}. 

\section{On the proof of Theorem \ref{theorem main characterization} for linear volatility}\label{section proof main result geometric}
In this section we assume that Condition \ref{ass sigma geometric} in Assumption \ref{assumption} holds. To simplify the notation, also this proof is given for $d=2$, so that $D=\mathbb{R}_+^{2}=\{x \in \mathbb{R}^{2}\, | \, x_1, \, x_{2} >0 \}$. The generalization to the case $d>2$ is straightforward.
  
\subsection{A preliminary lemma}\label{subsection geom preliminary lemma}  Define the set      
$$\mathcal{V}_+^x:=\{v\in \mathcal{V}\, | \, X_t^{1,x;v}, \,X_t^{2,x;v}>0 \text{ for each }t\geq0, \  \mathbb{P}\text{-a.s.} \}. $$
Under Condition \ref{ass sigma geometric}, as mentioned in Remark \ref{remark role of assumption}, the natural domain for an optimally-controlled trajectory  is $\mathbb{R}_+^d$. The following lemma formalizes this statement.
\begin{lemma} 
\label{lemma geometric X>0}
We have
$
V(x)= \min_{v\in \mathcal{V}_+^x} J(x;v),  \text{ for each } x \in \mathbb{R}_+^{2}.
$
\end{lemma}
\begin{proof}
The idea of the proof is as follows. {We will show that, if the negative component of the optimal control never acts when the optimal trajectory lies in the region $(-\infty, x_1^*] \times \mathbb R$, then such a trajectory always remains in $\mathbb R _+^2$. On the other hand, arguing by contradiction, we show that, if the optimal control acts when the optimal trajectory lies in the region $(-\infty, x_1^*] \times \mathbb R$, then it is possible to construct an admissible control which performs better. This then contradicts the uniqueness of the optimal control.}

Let $v\in \mathcal{V}$ be an optimal control for $x\in \mathbb{R}_+^{2}$, and denote by $(\xi^+,\xi^-)$ its minimal decomposition. 
In order to simplify the notation, set $X:=X^{x;v}$. 
Assuming that $v_s=0$ for each $s<0$ and recalling $x_1^*$ from Condition \ref{ass sigma geometric} in Assumption \ref{assumption}, define the random variable
$$
\tau:= \inf \{t\geq 0 \, |\, (X_t^1, \xi_{t+1/k}^- - \xi_{t-}^-) \in (-\infty,x_1^*) \times (0,\infty), \ \text{for any}\ k \in \mathbb N \}. 
$$
It can be easily shown that $\tau$ is an $\mathbb{F}$-stopping time. 
Also, such a definition of $\tau$ is such that the negative part $\xi^-$ of $v$ acts at time $\tau$; that is, $\tau$ is in the support of the measure $\xi^-$.

If $\mathbb{P}[\tau < \infty]=0$, then the control $\xi^-$ never acts when the state process $X^1$ lies in the region $(-\infty,x_1^*)$. 
Since $a_1\geq 0$ and $b^2\geq 0$, this is enough to ensure that $X_t^{1,x;v}, \,X_t^{2,x;v}>0 \text{ for each }t\geq0, \  \mathbb{P}\text{-a.s.}$, which in turn implies that $v \in \mathcal{V}_+^x$.

Arguing by contradiction, suppose that $\mathbb{P}[\tau < \infty]>0$. 
Define the control $\tilde{v}_t:=\mathds{1}_{ \{t<\tau \} } v_t +\mathds{1}_{ \{ t \geq \tau \}} (\xi_t^+  - \xi_{\tau-}^- + \min\{ \frac{3}{2} x_1^* - X_{\tau -}^1, 0\} \mathds{1}_{\{ \Delta \xi_\tau^- >0 \} })$, and the process $\tilde{X}:=X^{x;\tilde{v}}$. Define next the stopping time $\bar{\tau}:= \inf \{t \geq \tau \, | \, \tilde{X}_t^1 \geq 2x_1^* \}$, the control 
$
\bar{v}_t:= \mathds{1}_{ \{t < \bar{\tau} \}} \tilde{v} + \mathds{1}_{ \{t \geq \bar{\tau} \} } (X_{\bar{\tau}}^1 - \tilde{X}_{\bar{\tau}-}^1 + v_t - v_{\bar{\tau}})
$
and the process  $\bar{X}:=X^{x;\bar{v}}$. 
{Notice that, on $\{ \tau < \infty \}$, we have $\tau < \bar{\tau}$.}
Also, by the definition of $\bar{v}$ and of $\tau$, for $k$ such that $\tau + 1/k < \bar{\tau}$, on $\{ \tau < \infty \}$ we have 
\begin{align*}
\bar{v}_{\tau + 1/k}-v_{\tau + 1/k} 
& = \xi^- _{\tau + 1/k}-\xi^- _{\tau -} + \min \{ \tfrac{3}{2} x_1^* - X_{\tau-}^1, 0 \}  \mathds{1}_{\{ \Delta \xi_\tau^- >0 \} } \\
& = \xi^- _{\tau + 1/k}-\xi^- _{\tau} + \mathds{1}_{\{ \Delta \xi_\tau^- >0\}} \Delta \xi_\tau^- +  ( \tfrac{3}{2} x_1^* - X_{\tau-}^1) \mathds{1}_{ \big\{ \Delta \xi_\tau^- >0, \, \tfrac{3}{2} x_1^* < X_{\tau-}^1 \big\} }\\
& \geq \xi^- _{\tau + 1/k}-\xi^- _{\tau} + \mathds{1}_{ \big\{ \Delta \xi_\tau^- >0, \, \tfrac{3}{2} x_1^* \geq X_{\tau-}^1 \big\} } \Delta \xi_\tau^- +  \tfrac{x_1^*}{2} \mathds{1}_{ \big\{ \Delta \xi_\tau^- >0, \,  \tfrac{3}{2} x_1^* < X_{\tau-}^1 \big\} }>0,
\end{align*}
so that the processes $v$ and $\bar{v}$ are not indistinguishable.
Moreover, $v$ and $\bar{v}$ are such that, on $\{ \tau < \infty \}$, we have
\begin{equation}
\label{eq relation X barX}
\begin{cases} 
    X_t^1=\bar{X}_t^1 \text{ for } t\in [0,\tau) \cup [\bar{\tau}, \infty), \\  X_t^1 \leq \bar{X}_t^1 \text{ for } t\in [\tau,\bar{\tau}). 
\end{cases}
\end{equation}
After some manipulations, from \eqref{eq relation X barX} we deduce that
\begin{align}\label{eq bar v better than v hold} 
J(x;v) - J(x;\bar{v})  
& =   \mathbb{E}\bigg[ \mathds{1}_{\{\tau < \infty \}} \bigg( \int_{(\tau, \bar{\tau})} e^{-\rho t} d \xi_t^- + \int_\tau^{\bar{\tau}} e^{-\rho t} Dh(\hat{X}_t) (X_t - \bar{X}_t) dt \bigg) \bigg] \\ \notag 
& \quad + \mathbb{E}[ \mathds{1}_{\{\tau < \infty \}} e^{-\rho \tau} ( |X_{\tau}^1- X_{\tau-}^1| -| \bar{X}_{\tau}^1 -\bar{X}_{\tau -}^1| )] \\ \notag 
& \quad + \mathbb{E}[\mathds{1}_{\{\tau < \infty \}}  e^{-\rho \bar{\tau}} ( |X_{\bar{\tau}}^1- X_{\bar{\tau}-}^1| -| \bar{X}_{\bar{\tau}}^1 -\bar{X}_{\bar{\tau} -}^1| ) ], \notag 
\end{align}  
for $\hat{X}_t = \lambda_t \bar{X}_t  + (1-\lambda_t) {X}_t^{x;v} \in  (-\infty, 2 x_1^*) \times \R$, and suitable choice of $\lambda_t(\omega) \in [0,1]$. 
We point out that, the expectations in \eqref{eq bar v better than v hold} are well defined also for $\bar \tau = \infty$. Indeed, since $v$ is optimal, we have $\lim_{T\to \infty } \mathbb E [ e^{-\rho T} |v|_T ] =0 $, so that $ e^{-\rho \bar{\tau}} ( |X_{\bar{\tau}}^1- X_{\bar{\tau}-}^1| -| \bar{X}_{\bar{\tau}}^1 -\bar{X}_{\bar{\tau} -}^1| )=0$ $\mathbb P$-a.s.\ on $\{\bar \tau = \infty \}$.

{
Noticing that
\begin{equation}
\label{eq jump 1}
 e^{-\rho \tau} ( |X_{\tau}^1- X_{\tau-}^1| -| \bar{X}_{\tau}^1 -\bar{X}_{\tau -}^1|)  = e^{-\rho {\tau}} ( \bar{X}_{\tau}^1- {X}_{\tau}^1) \geq 0, 
\end{equation}
from \eqref{eq bar v better than v hold}, we write
\begin{align}\label{eq il gran finale}
J(x;v) - J(x;\bar{v}) = \mathbb E[\mathds{1}_{\{\tau < \infty \}} \Psi  ],
\end{align}
with 
\begin{align}\label{eq bar v better than v}
\Psi :=  \int_{(\tau, \bar{\tau})} e^{-\rho t} d \xi_t^- & + \int_\tau^{\bar{\tau}} e^{-\rho t} Dh(\hat{X}_t) (X_t - \bar{X}_t) dt \\ \notag
& +e^{-\rho {\tau}} ( \bar{X}_{\tau}^1- {X}_{\tau}^1) +  e^{-\rho \bar{\tau}} ( |X_{\bar{\tau}}^1- X_{\bar{\tau}-}^1| -| \bar{X}_{\bar{\tau}}^1 -\bar{X}_{\bar{\tau} -}^1| ). \notag
\end{align}
}
Now, if $\bar{X}_{\bar{\tau}}^1  \geq \bar{X}_{\bar{\tau}-}^1$, then using \eqref{eq relation X barX} we find
\begin{align}\label{eq jump 2 easy case}
 |X_{\bar{\tau}}^1 & - X_{\bar{\tau}-}^1| -| \bar{X}_{\bar{\tau}}^1 -\bar{X}_{\bar{\tau} -}^1|  \geq    \bar{X}_{\bar{\tau} -}^1 -{X}_{\bar{\tau} -}^1 \geq 0.
\end{align} 
Therefore, plugging  \eqref{eq jump 2 easy case} into \eqref{eq bar v better than v} and taking the expectation, using \eqref{eq jump 1} we obtain the inequality
\begin{align}\label{eq contradiction J v 1}
 \mathbb E [ \mathds{1}_{\{\tau < \infty , \, \bar{X}_{\bar{\tau}}^1  \geq \bar{X}_{\bar{\tau}-}^1\} } \Psi ]
 &\geq  \mathbb E \bigg[ \mathds{1}_{\{\tau < \infty ,\, \bar{X}_{\bar{\tau}}^1  \geq \bar{X}_{\bar{\tau}-}^1\} } \int_\tau^{\bar{\tau}} e^{-\rho t}h_{x_1}(\hat{X}_t) ({X}_t^1 - \bar{X}_t^1 )dt \bigg] \\ \notag 
 & \quad + \mathbb E \bigg[ \mathds{1}_{\{\tau < \infty ,\, \bar{X}_{\bar{\tau}}^1  \geq \bar{X}_{\bar{\tau}-}^1\} } \int_\tau^{\bar{\tau}} e^{-\rho t}h_{x_2}(\hat{X}_t) ({X}_t^2 - \bar{X}_t^2 ) dt \bigg] \geq 0, 
\end{align} 
where we have also used \eqref{eq relation X barX}, Condition \ref{ass sigma geometric} in Assumption \ref{assumption}, and that, due to the monotonicity of $b^2$ in the variable $x_1$, via a comparison principle {we have ${X}_t^2 - \bar{X}_t^2 \geq 0 $ for $t \in (\tau, {\bar{\tau}})$.}
On the other hand, if $\bar{X}_{\bar{\tau}}^1  \leq \bar{X}_{\bar{\tau}-}^1 $, 
from \eqref{eq relation X barX} we obtain
\begin{align}\label{eq jump 2}
 |X_{\bar{\tau}}^1 &- X_{\bar{\tau} -}^1| -| \bar{X}_{\bar{\tau}}^1 -\bar{X}_{\bar{\tau}- }^1|  \geq  {X}_{\bar{\tau} -}^1 -  \bar{X}_{\bar{\tau} -}^1 \\ \notag
 & = {X}_{{\tau} }^1 - \bar{X}_{{\tau} }^1 + \int_{\tau}^{\bar{\tau} - } b_1^1 ({X}_t^1 - \bar{X}_t^1 ) dt + \int_{\tau}^{\bar{\tau} - }  \sigma ({X}_t^1 - \bar{X}_t^1 ) dW_t^1  - \int_{(\tau, \bar{\tau})} d\xi_t^-.  
\end{align}
{
By estimates as in Lemma \ref{lemma estimate SDE}, we have that
$\mathbb E [ \int_0^\infty e^{-\rho t} |X_t^1|^2 dt ] \leq C(1+|x|^2)$. 
Also, by definition of $\bar X$ and $\bar \tau$, using \eqref{eq relation X barX} we have that $ X^1_t \leq \bar X^1_t \leq |X_t^1|+ 2x_1^*$. 
Hence,  
\begin{align*}
   &\sup_{t\geq 0} \mathbb E \bigg[  \bigg| \int_0^t  e^{-\rho \bar \tau} \mathds{1}_{\{ b_1^1 \leq 0 , \, \tau < \infty, \, \bar{X}_{\bar{\tau}}^1  \leq \bar{X}_{\bar{\tau}-}^1\} }  \mathds{1} _{(\tau, \bar \tau)}(s) \sigma ({X}_s^1 - \bar{X}_s^1 ) dW_s^1 \bigg|^2 \bigg] \\
   & \quad  \leq \sup_{t\geq 0} \mathbb E \bigg[   \int_0^t e^{-\rho s} \sigma (|{X}_s^1|^2 + |\bar{X}_s^1|^2 )ds \bigg]   \leq C(1+|x|^2 + |x_1^*|^2) < \infty,
\end{align*}
which, by a version of the martingale convergence theorem (see, e.g., Problem 3.20 at p.\ 18 in \cite{karatzas&shreve1998}), yields
\begin{align*}
    \label{eq martingale conv theorem}
   &  \mathbb E \bigg[  \int_{\tau}^{\bar \tau} e^{-\rho \bar \tau} \mathds{1}_{\{ b_1^1 \leq 0 , \, \tau < \infty, \, \bar{X}_{\bar{\tau}}^1  \leq \bar{X}_{\bar{\tau}-}^1\} } \sigma ({X}_s^1 - \bar{X}_s^1 ) dW_s^1  \bigg] \\ \notag
   & \quad = \lim_{t\to\infty}  \mathbb E \bigg[  \int_0^t e^{-\rho \bar \tau} \mathds{1}_{\{ b_1^1 \leq 0 , \, \tau < \infty, \, \bar{X}_{\bar{\tau}}^1  \leq \bar{X}_{\bar{\tau}-}^1\} }  \mathds{1} _{(\tau, \bar \tau)}(s) \sigma ({X}_s^1 - \bar{X}_s^1 ) dW_s^1  \bigg] = 0. \notag
\end{align*}
By using the latter equality and by substituting \eqref{eq jump 2} into \eqref{eq bar v better than v}, as in \eqref{eq contradiction J v 1} we obtain
\begin{equation}\label{eq contradiction J v2}
  \mathbb E [ \mathds{1}_{\{ b_1^1 \leq 0 , \, \tau < \infty, \, \bar{X}_{\bar{\tau}}^1  \leq \bar{X}_{\bar{\tau}-}^1\} } \Psi] \geq 0 .
\end{equation}}
Similarly, for $b_1^1 \geq 0$ we find
\begin{align}\label{eq contradiction J v3}
\mathbb E [ & \mathds{1}_{\{ b_1^1 \geq 0, \, \tau < \infty, \,   \bar{X}_{\bar{\tau}}^1  \leq \bar{X}_{\bar{\tau}-}^1\} } \Psi ] \\ \notag
& \quad \quad  \geq 
\mathbb E \bigg[ \mathds{1}_{\{ b_1^1 \geq 0, \, \tau < \infty, \,  \bar{X}_{\bar{\tau}}^1  \leq \bar{X}_{\bar{\tau}-}^1\} }   \int_\tau^{\bar{\tau}} e^{-\rho t} (h_{x_1}(\hat{X}_t)+b_1^1) ({X}_t^1 - \bar{X}_t^1 ) dt\bigg] \\ \notag 
& \quad \quad \quad + 
\mathbb E \bigg[ \mathds{1}_{\{ b_1^1 \geq 0, \, \tau < \infty, \, \bar{X}_{\bar{\tau}}^1  \leq \bar{X}_{\bar{\tau}-}^1\} }  \int_\tau^{\bar{\tau}} e^{-\rho t} h_{x_2}(\hat{X}_t) ({X}_t^2 - \bar{X}_t^2 ) dt \bigg] \geq 0. 
\end{align}
Finally, adding the inequalities \eqref{eq contradiction J v 1}, \eqref{eq contradiction J v2} and \eqref{eq contradiction J v3}, and using \eqref{eq il gran finale} and \eqref{eq jump 1} we obtain
\begin{align*} 
J(x;v) - J(x;\bar{v})   \geq  0, 
\end{align*}
which contradicts the uniqueness of the optimal control $v$, completing the proof of the lemma.   
\end{proof}

\subsection{Sketch of the proof of Theorem \ref{theorem main characterization}}\label{subsection sketch of the proof geom} Since we are interested in characterizing the optimal control for any given $\bar{x} \in \R_+^2$, thanks to Lemma \ref{lemma geometric X>0} we can restrict the domain of the HJB equation to the set $\R_+^2$.
We observe that, upon exploiting the ellipticity of the operator $\mathcal{L}$ in the domain $\mathbb{R}_+^{2}$ (and, in particular, the uniform ellipticity of $\mathcal{L}$ on each ball $B\subset\mathbb{R}_+^{2}$), all the results from Sections \ref{section a connection with Dynkin games} and \ref{section epsilon optimal policies} can be recovered, with minimal adjustments of the arguments therein. 

For $\bar{x}\in {\mathcal{W}}$ we can consider the processes  
$
X^n:=X^{\bar{x};v^n}, \ v^n  \text{  for } n  \in \mathbb{N},
$
with $(v^n)_{n\in \mathbb{N}}$ minimizing sequence of solutions to the Skorokhod problems on domains $\overline{\mathcal{W}}_n$, according to Lemma \ref{lemma epsilon optimality} (here $\overline{\mathcal{W}}_n$ denotes the closure of $\mathcal{W}_n$ in $\R_+^2$). 

Estimates as those of Lemma \ref{lemma estimate SDE} can now be proved as follows. 
Denoting by $X^{\bar{x}}$ the solution to \eqref{SDE dynamics uncontrolled}, by standard results (see, e.g., Theorem 4.1 at p.\ 59 in \cite{mao2008}) we have
$\mathbb{E}[|X_t^{\bar{x}}|^p] \leq C e^{p\,(2 \bar{L} + \sigma^2(p-1)) t}(1+|\bar{x}|^p)$ for each $t\geq0$. 
Hence, arguing as in the proof of Lemma \ref{lemma estimate SDE} and using the requirement on $\rho$ from Condition \ref{ass sigma geometric} in Assumption \ref{assumption}, we find
\begin{equation}\label{eq  geom estimate xi n}
\sup_n \int_0^\infty e^{-\rho t} \mathbb{E} [ |X_t^{1,n}|^{p} ] dt   \leq C ( 1+|\bar{x}|^p ).
\end{equation}
Next, for $p':=(2 p-1)/2$, we use \eqref{eq  geom estimate xi n} to estimate $|X^{2,n}|^{p'}$. We underline that, since  $\overline{\mathcal{W}}_n \subset {\mathcal{W}}$, we have $X_t^n > 0 \  \mathbb{P}\otimes dt\text{-a.e.\ in } \Omega \times [0,\infty) $.
For each $n\in \mathbb{N}$, define the process $\Lambda^n$ as the solution to the SDE
$$
d\Lambda_t^n = \bar{L}(1+|X_t^{1,n}| + \Lambda_t^n)dt + \sigma \Lambda_t^n dW_t^2, \ t \geq0, \quad  \Lambda_0^n=\bar{x}_2. 
$$
Since $X_t^{2,n} \leq \bar{x}_2 + \int_0^t \bar{L}(1+|X_s^{1,n}| + |X_s^{2,n}|)ds + \sigma \int_0^t X_s^{2,n} dW_s^2,$ by a comparison principle we obtain $X^{2,n} \leq \Lambda^n$. 
{Therefore, using that $\Lambda_t^n= \hat{E}_t [\bar{x}_2 + \int_0^t \bar{L}(1+|X_s^{1,n}|)\hat{E}_s^{-1} ds ] $, with $\hat{E}_t:= \exp[(\bar{L}-\sigma^2/2)t + \sigma W_t^2]$, we find, for a suitable (deterministic) polynomial $p_t$,
\begin{align}\label{eq original}
\int_0^\infty e^{-\rho t} \mathbb{E}[ |X_t^{2,n}|^{p'}] dt  \leq \int_0^\infty e^{-\rho t} \mathbb{E}[ |\Lambda_t^{n}|^{p'}] dt & \leq C \int_0^\infty  e^{-\rho t} \mathbb{E}\bigg[ \hat{E}_t^{p'}\bar{x}_2^{p'} + p_t \int_0^t \hat{E}_t^{p'}\hat{E}_s^{-p'} ds\bigg] \\ \notag
& + C \int_0^\infty p_t e^{-\rho t}  \int_0^t \mathbb{E} \big[ |X_s^{1,n}|^{p'}({\hat{E}_t}/{\hat{E}_s})^{ p'} \big] ds \,  dt.
\end{align} 
By using H\"{o}lder's inequality with exponent $q=p/p'$ ($q^*$ denoting the conjugate of $q$), we estimate the integrand of the second time-integral in the right-hand side of \eqref{eq original} so to obtain  
\begin{align}\label{eq geom estimate X2}
e^{-\rho t} & \int_0^t \mathbb{E} \big[ |X_s^{1,n}|^{p'}({\hat{E}_t}/{\hat{E}_s})^{ p'}  \big] ds \\ \notag
& \leq e^{-\rho(1-\frac{1}{q})t}  \int_0^t \mathbb{E} \big[ e^{-\frac{\rho}{q}s} |X_s^{1,n}|^{p'}({\hat{E}_t}/{\hat{E}_s})^{ p'} \big] ds \\ \notag
& \leq  C e^{-\rho(1-\frac{1}{q})t} \bigg( \int_0^{t} e^{-\rho s}\mathbb{E}[ |X_s^{1,n}|^{p}] ds \bigg)^{\frac{1}{q}} \bigg( \int_0^t \mathbb{E}[ ({\hat{E}_t}/{\hat{E}_s})^{ p' q^*}] ds \bigg)^{\frac{1}{q^*}} \\ \notag
& \leq  C e^{-\rho(1-\frac{1}{q})t} \bigg( \int_0^{\infty} e^{-\rho s}\mathbb{E}[ |X_s^{1,n}|^{p}] ds \bigg)^{\frac{1}{q}} \bigg( \int_0^t \mathbb{E}[ ({\hat{E}_t}/{\hat{E}_s})^{ p' q^*}] ds \bigg)^{\frac{1}{q^*}}.
\end{align}
Hence,  substituting \eqref{eq  geom estimate xi n} into \eqref{eq geom estimate X2}, and then feeding the result back into \eqref{eq original}, we have
\begin{align}\label{eq geom estimate X2 second}
\int_0^\infty e^{-\rho t} \mathbb{E}[ |X_t^{2,n}|^{p'}] dt & \leq C \int_0^\infty  e^{-\rho t} \mathbb{E}\bigg[ \hat{E}_t^{p'}\bar{x}_2^{p'} + p_t \int_0^t \hat{E}_t^{p'}\hat{E}_s^{-p'} ds\bigg] \\ \notag
& + C ( 1+|\bar{x}|^p ) \int_0^\infty p_t e^{-\rho(1-\frac{1}{q})t}  \bigg( \int_0^t \mathbb{E}[ ({\hat{E}_t}/{\hat{E}_s})^{ p' q^*}] ds \bigg)^{\frac{1}{q^*}} dt.
\end{align} 
Furthermore, exploiting the requirement on $\rho$ made in Condition \ref{ass sigma geometric} in Assumption \ref{assumption}, after elementary computations one can see that
\begin{align}\label{eq estimate porca}
     & \int_0^\infty  e^{-\rho t} \mathbb{E}\bigg[ \hat{E}_t^{p'}\bar{x}_2^{p'} + p_t \int_0^t \hat{E}_t^{p'}\hat{E}_s^{-p'} ds\bigg]  \leq C (1 + |\bar{x}|^p), \\ \notag
     & \int_0^\infty  p_t e^{-\rho(1-\frac{1}{q})t}\bigg( \int_0^t \mathbb{E}[ ({\hat{E}_t}/{\hat{E}_s})^{ p' q^*}] ds \bigg)^{\frac{1}{q^*}} dt < \infty.
\end{align} 
Finally, substituting  \eqref{eq estimate porca} in \eqref{eq geom estimate X2 second}, we conclude that
$$
\sup_n \int_0^\infty e^{-\rho t} \mathbb{E}[ |X_t^{2,n}|^{p'}] dt \leq C (1 + |\bar{x}|^p), 
$$
which, combined with \eqref{eq  geom estimate xi n} {(and using that $p' < p$)}, gives
\begin{equation}\label{eq geom estimate SDE full}
\sup_n \int_0^\infty e^{-\rho t} (\mathbb{E} [ |X_t^{1,n}|^{p} ] +\mathbb{E} [ |X_t^{n}|^{{p'}} ] )dt   \leq C (1+|\bar{x}|^p).
\end{equation} 
}

Thanks to the estimate \eqref{eq geom estimate SDE full}, the arguments of Step 1 in the proof of Lemma \ref{lemma construction optimal policy} can be recovered, so that (up to a subsequence)
\begin{equation}\label{eq geom limits of X^n to hat X}
X_t^n \to \hat{X}_t \quad  \mathbb{P}\otimes dt\text{-a.e.\ in } \Omega \times [0,\infty), \quad \text{as } n\to \infty,
\end{equation}
for an adapted process $\hat{X}$. 
Using again \eqref{eq geom estimate SDE full} and the assumption $p \geq 2$, a standard use of Banach-Saks' theorem allows to find a subsequence of indexes $(n_j)_{j\in \mathbb{N} }$ such that the Cesàro means of $(X^{1,n_j})_{j\in \mathbb{N}}$ converge in $\mathbb{L}^2$ to the process $\hat{X}^1$; that is,  
\begin{equation}
\label{cesaro.convergence X}
\bar{X}^{1,m}:=\frac{1}{m}\sum_{j=1}^{m} X^{1,n_j} \rightarrow \hat{X}^1,\text{ as } m\to \infty, \quad \text{in } \mathbb{L}^2(\Omega\times [0,T]; \mathbb{P}\otimes dt ), \text{ for each $T>0$.}
\end{equation}
Next, defining the process $v_t:=\hat{X}_t^1 - \bar{x}_1 - \int_0^{t} (a_1 +b_1^1 \hat{X}_s^1) ds - \sigma \int_0^{t} \hat{X}_s^1 dW_s^1$, and exploiting the $\mathbb{L}^2$ convergence in \eqref{cesaro.convergence X} and the linearity of the dynamics for the first component, we deduce that
\begin{equation}
\label{cesaro.convergence v}
\bar{v}^{m}:=\frac{1}{m}\sum_{j=1}^m v^{n_j} \rightarrow v,\text{ as } m\to \infty, \quad \text{in } \mathbb{L}^2(\Omega\times [0,T]; \mathbb{P}\otimes dt ), \text{ for each $T>0$,}
\end{equation}
{where the processes $v^{n_j}$ were introduced at the beginning of this proof.}
Again, by using Lemma 3.5 in \cite{K}, we can assume the processes $\hat{X}^1$ and $v$ to be right-continuous.
Next, observe that the processes $X^{2,n}$ can be expressed as
$$
X_t^{2,n}= E_t \big[ \bar{x}_2 + \begin{matrix} \int_0^t b^2(X_s^n)/E_s ds  \big], \quad \text{with}\quad E_t:= \exp \big( \sigma W_t^2 - \frac{\sigma^2}{2} t \big),\end{matrix} \quad t \geq 0.
$$ 
Hence, taking limits as $n\to \infty$ in the latter equality (exploiting \eqref{eq geom limits of X^n to hat X} and  the uniform integrability deriving from \eqref{eq geom estimate SDE full}), we deduce that 
$$
\hat{X}_t^{2}= E_t \big[ \bar{x}_2 + \begin{matrix} \int_0^t b^2(\hat{X}_s)/E_s ds \end{matrix} \big], \quad  t \geq 0,
$$
so that, thanks also to the very definition of $v$, we have $\hat{X}=X^{\bar{x};v}$. 
Overall, from \eqref{eq geom limits of X^n to hat X}, \eqref{cesaro.convergence v} and the latter equality, we have 
\begin{equation}\label{eq geom final limits}
\bar{X}^{m}:=\frac{1}{m}\sum_{j=1}^{m} X^{n_j} \rightarrow X^{\bar{x};v}, \text{ and }\bar{v}^{m} \rightarrow v,  \quad  \mathbb{P}\otimes dt\text{-a.e.\ in } \Omega \times [0,\infty), \text{ as } m\to \infty.
\end{equation}
It is however worth noticing that $\bar{X}^m$ is not the solution of the SDE controlled by $\bar{v}^m$, unless $b^2$ is affine. 
Similarly to \eqref{eq pointwise convergence of total variation}, using the fact that the sequence of controls $v^n$ is minimizing, and exploiting the limits in \eqref{eq geom final limits} and the convexity of $h$, we find   
\begin{align*}
J(\bar{x};v) &=   \mathbb{E} \bigg[ \int_0^\infty e^{-\rho t}  h(X_t^{\bar{x};v}) dt +\rho \int_0^\infty e^{-\rho t} |v|_t dt  \bigg] \\ \notag
& \leq \liminf_m \mathbb{E} \bigg[ \int_0^\infty e^{-\rho t} h(\bar{X}_t^m) dt + \rho \int_0^\infty e^{-\rho t} |\bar{v}^m|_t dt  \bigg] \\
& \leq \liminf_m \frac{1}{m} \sum_{j=1}^m \mathbb{E} \bigg[ \int_0^\infty e^{-\rho t} h(X_t^{n_j}) dt + \rho \int_0^\infty e^{-\rho t} |{v}^{n_j}|_t dt  \bigg] =V(\bar{x}),
\end{align*}
so that the control $v$ has locally bounded variation and it is optimal. By uniqueness of the optimal control,
we deduce that $\bar{v}=v$ and $\hat{X}=X^{\bar{x}; \bar{v}}$.
  
Finally, thanks to the properties of $(X^n, v^n)$, by repeating the arguments leading to Propositions \ref{prop X lives in W}, \ref{proposition the cont parts acts on the boundary} and \ref{propo charact jumps} (see Appendix \ref{appendix proof propositions kruk}), the optimal control $\bar{v}$ for $\bar{x}\in {\mathcal{W}}$ can be characterized as the unique solution to the modified Skorokhod problem for the SDE \eqref{SDE dynamics} in $\overline{\mathcal{W}}$ starting at $\bar{x}$ with reflection direction $-V_{x_1} e_1$. On the other hand, for $\bar{x}\in \overline{\mathcal{W}}$, we can repeat the rationale at the end of Subsection \ref{section proof theorem x inside waiting}, which yields that the optimal control can be characterized also for $\bar{x}\in \overline{\mathcal{W}}$,  completing the proof of Claim \ref{theorem x inside waiting} of Theorem \ref{theorem main characterization}.  
  
When $\bar{x} \notin \overline{\mathcal{W}}$, following the arguments of Subsection \ref{subsection proof x outside waiting}, one can characterize the initial jump of $\bar{v}$. This completes the proof of Theorem \ref{theorem main characterization} under Condition \ref{ass sigma geometric} in Assumption \ref{assumption}. 
 
\section{Comments, extensions and examples}\label{section examples}
 
\subsection{Refinements of Assumption  \ref{assumption}}\label{section refinements}  Assumption  \ref{assumption} can be improved as follows. 
\subsubsection{Affine drift}\label{subsection Affine drift}
{
When the drift of the dynamics is affine, some of the monotonicity conditions in Assumption \ref{assumption} can be relaxed.
Indeed, Theorem \ref{theorem main characterization} holds if Assumption \ref{assumption} is replaced with the following conditions: 
\begin{assumption} For $p=2$, assume that: 
    \label{assumption affine drift}
    \begin{enumerate}
        \item The running cost $h$ satisfies Condition \ref{assumption on h} in Assumption \ref{assumption};
        \item $\bar{b}(x):=a + bx$, for a vector $a \in \mathbb{R}^{d}$ and a matrix $b\in\mathbb{R}^{d\times d}$ such that the vector
        $
        \beta := ( 0,b_1^2,...,b_1^d)^{\text{\tiny {$\top$}}} \in \mathbb{R}^{d}$ is an eigenvector of $b$ and  $h_{x_1 \beta} \geq 0$ (here the vector $( 0,b_1^2,...,b_1^d)^{\text{\tiny {$\top$}}}$ is the first column of $b$, with $b_1^1$ replaced by $0$, while $h_{x_1 \beta}$ denotes the $\beta$-directional derivative of $h_{x_1}$);
        \item \label{assumption affine drift.hro} $\bar{\sigma}=\sigma$ for a constant $\sigma > 0$ and  $
        \rho > 2 \Lambda(b)$, with $ \Lambda(b):= \max \{ \text{Re}(\lambda) \, | \, \lambda \text{ eigenvalue of $b$} \}.  
        $
    \end{enumerate}
\end{assumption}
In this case, for $x \in \mathbb{R}^{d}$, $r > 0$ and $x^r:= x + r  \beta $,  the solution $X^{x^r}$ of \eqref{SDE dynamics uncontrolled} writes (see, e.g., p.\ 99 in \cite{mao2008})  as  $ X_t^{x^r}=e^{bt} x^r + P_t $, where $P_t$ does not depend on $x^r$. 
Hence, since the vector $\beta$ is by assumption an eigenvector of the matrix $b$ with eigenvalue $\lambda$, we find
$ X_t^{x^r}-X_t^x =r \,  e^{tb}\beta = r e^{t \lambda} \beta,  \text{ for each } t\geq 0, \ \mathbb{P}$-a.s. This easily allows to repeat the arguments in the proof of Proposition  \ref{proposition V_x1 nondecreasing in x2}, so that $V_{x_1 \beta} \geq 0$, while all of the other results  in this paper still hold (often with less technical proofs).   
We refer to Lemma 2.2 and Theorem 2.3 in \cite{chiarolla&haussmann2000} for more details on the sufficiency of the requirement on $\rho$ in Condition \ref{assumption affine drift.hro} in Assumption \ref{assumption affine drift} (the case $p>2$ can be treated as well, for  $\rho$ large enough).
We also underline that all the results in this paper apply for a constant volatility matrix $\bar{\sigma}$ such that $\bar{\sigma} \bar{\sigma}^{\text{\tiny {$\top$}}}$ is positive definite, $\bar{\sigma}^{\text{\tiny {$\top$}}}$ denoting the transpose of $\bar{\sigma}$.
}
\subsubsection{On Condition \ref{assumption on b}}\label{subsection On Condition 2} A careful look into the proofs of Proposition \ref{proposition V_x1 nondecreasing in x2} and of Lemma \ref{lemma construction optimal policy}  reveals that the results in this paper remain valid if the drift coefficients $b^i$ in Condition  \ref{assumption on b} in Assumption \ref{assumption} satisfy one of the following more general requirements. 
        \begin{enumerate}
            \item Under Condition \ref{ass sigma constant}, for $i=2,...,d$, either of the following is satisfied:
        \begin{enumerate} 
            \item $b^i$ is convex, $h_{x_i} \geq0$, and either $ {b}_{x_1 }^{i},\, {b}_{x_1 x_{i}}^{i},\, h_{x_1 x_{i}} \leq 0 $ or $ {b}_{x_1 }^{i},\, {b}_{x_1 x_{i}}^{i},\, h_{x_1 x_{i}} \geq 0 $; 
            \item $b^i$ is concave, $h_{x_i} \leq0$, and either $ {b}_{x_1 }^{i},\, - {b}_{x_1 x_{i}}^{i},\,  h_{x_1 x_{i}} \leq 0 $ or $ {b}_{x_1 }^{i},\, - {b}_{x_1 x_{i}}^{i},\,  h_{x_1 x_{i}} \geq 0 $.
       \end{enumerate}  
        \item Under Condition \ref{ass sigma geometric}, for  $i=2,...,d$, either of the following is satisfied:
        \begin{enumerate}
            \item $b^i$ is convex, $h_{x_i} \geq0$, and  $ {b}_{x_1 }^{i},\, {b}_{x_1 x_{i}}^{i},\, h_{x_1 x_{i}} \leq 0 $; 
        \item $b^i$ is concave, $h_{x_i} \leq0$, and $ {b}_{x_1 }^{i},\, - {b}_{x_1 x_{i}}^{i},\,  h_{x_1 x_{i}} \leq 0 $.
        \end{enumerate}
\end{enumerate}
We point out that the conditions to deal with a linear volatility need to be compatible with the arguments in the proof of Lemma \ref{lemma geometric X>0} and are, for this reason, more restrictive.     
\subsubsection{On the lower-growth of $h$} We underline that the lower-growth requirement on $h$ in Condition \ref{assumption on h} in Assumption \ref{assumption} can be improved in some particular settings:
If the drift is affine and the volatility is constant, for $p\leq2$ it is sufficient to assume $h\geq - \kappa_2$. Indeed, in this case, the proof of the estimate \eqref{eq semiconcavity} in Step 2 in the proof of Theorem \ref{theorem V sol HJB} in Appendix   \ref{appendix auxiliary results} simplifies {(in particular, in \eqref{estimate A + B}, we would have $M_2=0$ since the processes $Z$ and $ X^{x^\lambda; \alpha}$ coincide, by linearity of the dynamics)} and it can be provided without relying on Lemma \ref{lemma estimate SDE}. Also, for any $x\in \mathbb{R}^d$ and any sequence of minimizing controls $(v^n)_{n \in\mathbb{N}}$, we have the estimate
$$
\sup_n \begin{matrix} \mathbb{E}\big[ \int_{[0,\infty)} e^{-\rho t} d |v^n|_t \big] \end{matrix} \leq C(1+|x|^p), 
$$
which, combined with $ \mathbb{E}[ |X_t^{x;v^n}| ] \leq C(1+|x|^p + \mathbb{E}[ |v^n|_t  ])e^{\bar{L} t} $, gives 
$$
\sup_n \begin{matrix} \mathbb{E}\big[ \int_{[0,\infty)} e^{-(\rho + \bar{L}) t}  |X_t^{x;v^n}|  dt \big] \end{matrix} \leq \sup_n  C \Big(1+|x|^p + \begin{matrix} \mathbb{E}\big[ \int_{0}^\infty e^{-\rho t}  |v^n|_t  dt \big] \end{matrix}\Big) \leq C(1+|x|^p). 
$$
Therefore, a limit process $\hat{X}$ such that $X_t^{x;v^n} \to \hat{X}_t$ $\mathbb{P}\otimes dt$-a.e.\ as $n\to \infty$ can be found, by adapting  the reasoning in Step 1 in the proof of Lemma  \ref{lemma construction optimal policy}. Also, using Lemma 3.5 in \cite{K}, in the spirit of what has been done in Subsection \ref{subsection sketch of the proof geom}, we can exploit the convexity of $h$ and the fact that $\bar{b}$ is affine in order to prove that $\hat{X}=X^{x;v}$,  with $v$ optimal control for the given $x$. This allows to recover Lemma \ref{lemma construction optimal policy} and to characterize the optimal control $v$.  
\subsection{Some remarks}\label{subsection some remark} We provide here some extensions of the results contained in this paper.
\begin{remark}[Asymmetric costs of action] Unless to slightly modify some of the arguments in this paper, Theorem \ref{theorem main characterization} extends to the case in which increasing the first component of the state process has a different cost than decreasing it; that is, to the cost functional
$$ 
J_{\kappa_1, \kappa_2}(x;v):=\mathbb{E}\bigg[ \int_0^\infty e^{-\rho t} h(X_t^{x;v}) dt+ \kappa_1\int_{[0,\infty)}e^{-\rho t} d\xi_t^{+} + \kappa_2\int_{[0,\infty)}e^{-\rho t} d\xi_t^{-} \bigg], \quad \kappa_1,\,\kappa_2 >0.
$$
In this case, the value function $V$ solves the HJB equation 
\begin{equation*} 
\max \{ \rho V-\mathcal{L} V -h , - V_{x_1} - \kappa_1,  V_{x_1} - \kappa_2 \} = 0,  \quad \text{a.e.\ in } \mathbb{R}^{2}.
\end{equation*}
This can be shown by employing arguments similar to those in the proof of Theorem \ref{theorem V sol HJB} in Appendix \ref{appendix auxiliary results}, by replacing the penalizing term in \eqref{eq appendix penalized HJB} with an ``asymmetric'' penalization
$[ \beta(-V_{x_1} - \kappa_1) + \beta(V_{x_1} - \kappa_2) ]/{\varepsilon}$. Most of the arguments in this paper remains essentially unchanged, and the optimal control can be characterized as the solution to a Skorokhod problem on the domain $\mathcal{W}_{\kappa_1, \kappa_2}:=\{ y \in \mathbb{R}^{d}\, | \, - \kappa_1 < V_{x_1}(y) < \kappa_2 \}$. 
\end{remark} 

\begin{remark}[Monotone controls]\label{remark monotone controls}
The approach in this paper allows also to characterize optimal controls for stochastic singular control problems where the minimization problem is formulated over the set of monotone controls; that is, when  $$V(x):= \inf_{\xi \in \mathcal{V}_{\uparrow} } J(x;\xi) \quad \text{with}\quad \mathcal{V}_{\uparrow}:=\{ \xi \in \mathcal{V}, \, \xi \text{ nondecreasing}\}.$$
In this case, $V$ solves the HJB equation $\max \{ \rho V-\mathcal{L} V -h , - V_{x_1} - 1 \} = 0,$ a.e.\ in $D$, and its derivative $V_{x_1}$ is the value function of an optimal stopping problem (rather than a Dynkin game).  
The arguments in this paper can be easily adapted, and the optimal control can be characterized as the solution to a Skorokhod problem on the domain $\mathcal{W}_+:=\{ y \in \mathbb{R}^{d}\, | \,  - 1 < V_{x_1}(y) \}$.
We stress that, in this case, the additional requirements on $h$ and $\bar{b}$ in Condition \ref{ass sigma geometric} in Assumption \ref{assumption} are not anymore needed (see Remark \ref{remark role of assumption}). 
\end{remark}
\begin{remark}[Finite time horizon] 
A characterization result analogous to Theorem \ref{theorem main characterization} could also be investigated for an optimal control problem over a finite time-horizon. For example, when $d=2$ and $b$ is affine, a connection with Dynkin games is known from \cite{chiarolla&haussmann1998}. Therefore, it seems possible to use this connection in order to investigate the monotonicity of the value of the game (as in Proposition \ref{proposition V_x1 nondecreasing in x2}), and to use this monotonicity in order to construct $\varepsilon$-optimal controls $v^\varepsilon$.  
In this case, building on the results in \cite{boryc&kruk2016}, one can try to study the limit as $\varepsilon \to 0$ of $(v^\varepsilon)_{\varepsilon>0}$, in order to provide a characterization of the optimal control.
\end{remark}

\subsection{Examples}\label{subsection examples}     

For the sake of illustration, we begin with the following:

\begin{example}\label{example key} 
{Let $d=2$, $\rho$ large enough, $\phi$ be a convex nonincreasing function  and $f$ be a convex nondecreasing function. In light of the discussion in Section \ref{section refinements}, the optimal control can be then characterized in the following settings: 
\begin{enumerate}
    \item $\bar{\sigma}$ as in Condition \ref{ass sigma constant} and
\begin{enumerate}
    \item\label{example linear quadratic} $b^2(x) = a^2 + b_1^2 x_1 + b_2^2 x_2$, $h(x)= Q_1 x_1^2 + Q_{12} x_1 x_2 + Q_2 x_2^2$ with $Q_1, Q_2 \geq 0$ and $Q_{12} b_1^2 \geq 0$ (in this case, Assumption \ref{assumption affine drift} is satisfied);
    \item $b^2(x) =  \phi(x_1) + b_2^2 x_2$, $h(x)=|x_1|^2 + f(x_2)$ (in this case, the conditions discussed in Subsection \ref{subsection On Condition 2} are met); 
\end{enumerate}
\item $\bar{\sigma}$ and $a_1$ as in Condition \ref{ass sigma geometric}, $\phi$ is nonnegative, $b_2^2 \geq 0$, $b_1^1 \leq 0$,  $x_1^*>0$ and
$b^2(x) =  \phi(x_1) + b_2^2 x_2$, $h(x)=|x_1 - x_1^*|^2 + f(x_2)$ (in this case, the conditions discussed in Subsection \ref{subsection On Condition 2} are met).
\end{enumerate}
In particular, the setting of Point (\ref{example linear quadratic}) encompasses a relevant class of \emph{linear-quadratic} singular stochastic  control problems, and it can be thought of as the leading example of this work.}    
\end{example}

\begin{example}
Here we discuss a model of pollution control. In the sequel, $x\in \R_+^2$ is the given and fixed initial condition of the state variable. Consider a company that can increase via an irreversible investment plan $\xi \in \mathcal{V}_\uparrow$ (cf.\ Remark \ref{remark monotone controls}) its production capacity $X^{1,x;\xi}$. The latter depreciates at constant rate $\delta>0$ and is randomly fluctuating, e.g.\ because of technological uncertainty. Production leads to emissions of pollutants and thus impacts the level of a state process $X^{2,x;\xi}$ which summarizes one or more stocks of environmental pollutants (such as the average concentration of CO2 in the atmosphere). We assume that such an externality of production on the stock of pollutants is measured by a positive, convex, increasing, Lipschitz continuous function $\phi$ that has bounded second order derivative. Overall, the dynamics of $X^{x;\xi}$ is given by \begin{equation*}
 \begin{cases}
     dX_t^{1,x;\xi}= -\delta X_t^{1,x;\xi} dt +\sigma_1 X_t^{1,x;\xi} d W_t^1 + d\xi_t, \\
     dX_t^{2,x;\xi}= ( \phi(X_t^{1,x;\xi})- X_2^{2,x;\xi}) dt +\sigma_2 X_t^{2,x;\xi} d W_t^2. 
 \end{cases}  
\end{equation*}

The company aims at choosing a production plan that minimizes the sum of different costs: the cost of not meeting a given production level $\theta$; the penalty of leading to a level of pollution that exceeds some environmental target $\vartheta$; the proportional costs of investment. That is,
\begin{equation*}
    V(x) = \inf_{\xi \in \mathcal{V}_\uparrow}\mathbb{E} \bigg[ \int_0^\infty e^{ -\rho t} \big( (X_t^{1,x;\xi}-\theta)^2 + c(X_t^{2,x;\xi} - \vartheta) \big) dt + \int_{[0,\infty)}e^{ -\rho t} d\xi_t \bigg]. 
\end{equation*} 
Here, $c \in C_{loc}^{2;1}(\R)$ is a nonnegative, nondecreasing, convex, Lipschitz continuous function such that $c(y)=0$ for $y\leq 0$, and with bounded second order derivative. 
In light of the discussion in Subsections \ref{section refinements} and \ref{subsection some remark},  the optimal control for $V$ can be characterized as the solution to its related Skorokhod problem.
\end{example} 

We next turn our focus to 
examples of bounded-variation problems treated in the literature and for which our results apply.
\begin{example}\label{example chiarolla haussmann}
 We discuss the model studied in  \cite{chiarolla&haussmann2000}. For $d=2$, consider the singular control problem with running cost $h(x_1, x_2)=\nu x_1^2 + x_2^2$, for $\nu>0$, and drift $\bar{b}(x)=a+bx$, for a constant vector $a\in \mathbb{R}^2$ and a matrix 
$$
b = \begin{pmatrix}
b_1^1 & b_2^1 \\
b_1^2 & b_2^2
\end{pmatrix} \in \mathbb{R}^{2\times 2 },
$$
Observe that the requirements discussed in Subsection \ref{subsection Affine drift}, are satisfied by assuming $b_2^1=0$ and $\rho > 2 \Lambda(b)$.
Therefore, Theorem \ref{theorem main characterization} gives the optimal control as the solution of the related Skorokhod problem. 
The same result was obtained in \cite{chiarolla&haussmann2000} only under the additional assumption of a global Lipschitz-continuous  free boundary. 
\end{example}

\begin{example} 
 Another example of setup similar to ours has been studied in \cite{yang2014}, where a multidimesional singular control problem with $d \geq 2$ and constant drift and volatility is considered. There, the author shows the $C^2$-regularity of the value function, allowing for the characterization of the optimal policy as a solution to the related Skorokhod problem (even in the case of a state dependent cost of intervention). It is easy to see that, when the drift $\bar{b}$ is assumed to be constant, no monotonicity of the running cost $h$ is required in order to obtain our Theorem \ref{theorem main characterization}.
In comparison with \cite{yang2014}, our main result (cf.\ Theorem \ref{theorem main  characterization})  allows to characterize the optimal policy even in cases in which the dynamics are interconnected (at the cost of additional structural conditions on the running cost $h$).  
\end{example}

\subsection{An example with degenerate  dynamics}\label{subsection example Patrick} A more involved discussion is required to treat the degenerate singular control problem studied in \cite{federico&ferrari&schuhmann2019} (see also  \cite{federico&ferrari&schuhumann2020b}).  

In this subsection, we take $d=2$, $h$ satisfying Condition \ref{assumption on h} in Assumption \ref{assumption}, $\bar{b}(x) =a+bx =(\bar{b}^1(x),\bar{b}^2(x))^{\textbf{\tiny{$\top$}}} $, and
\begin{equation}\label{eq dynamics patrick} 
a = \begin{pmatrix}
0  \\
 a^2
\end{pmatrix}, \quad
b = \begin{pmatrix}
0 & 0 \\
b_1^2 & b_2^2
\end{pmatrix},  \quad \sigma = \begin{pmatrix} 
0 & 0 \\ 
0 & \eta
\end{pmatrix}, \quad b_1^2,\, \eta,\, \rho >0, \quad b_2^2 < 0,  \quad h_{x_1 x_2} \geq 0.  
\end{equation} 
{In other words, $\bar{b}^1(x)=0$ and $\bar{b}^2(x)= a^2 + b_1^2 x_1 + b_2^2 x_2$ for any $x \in \R^2$.}
{In order to simplify the analysis of this example, assume  $ p = 2$. Observe that, in this case, all the requirements of Assumption \ref{assumption affine drift} are satisfied with $\Lambda(b)=0$, except from the nondegeneracy condition on ${\sigma}$.}    
The analysis of this subsection can be repeated also for $b_2^2=0$ and for a general $p \geq 1$.   
    
Despite in this example the matrix $ \sigma  \sigma ^{\textbf{\tiny{$\top$}}}$ is degenerate, the arguments in this paper can be employed in order to characterize the optimal control. However, some extra care is needed in order to prove the regularity of the value function inside the waiting region, which in fact follows from the properties of the free boundary proved in \cite{federico&ferrari&schuhmann2019} and \cite{federico&ferrari&schuhumann2020b}.

We begin the discussion by observing that results analogous to the ones contained in Appendix \ref{appendix auxiliary results} hold. In particular, Theorem \ref{theorem V sol HJB} can be shown by using a suitable perturbation of the matrix $\sigma$ (see the Appendix A in \cite{federico&ferrari&schuhmann2019}, for more details). The connection with Dynkin games holds as well (see Theorem 3.1 in \cite{federico&ferrari&schuhmann2019}), so that the arguments leading to Proposition \ref{proposition V_x1 nondecreasing in x2} (which make no use of the non-degeneracy of $\sigma \sigma^{\textbf{\tiny{$\top$}}}$) can be recovered. 

\subsubsection{Regularity of V in $\mathcal{W}$} We enforce an additional hypothesis, which is satisfied by $h(x)=|x|^2$ or $h(x)=(x_1+x_2)^2$. 
\begin{assumption}
\label{assumption additional for Patrick} \ 
\begin{enumerate}
\item $\lim_{x_2 \to \pm \infty} h_{x_2}(x_1,x_2)=\pm \infty$ for any $x_1 \in \mathbb{R}$;
\item One of the following holds true:
\begin{enumerate}
    \item $h_{x_1}(x_1, \cdot) $ is strictly increasing for any $x_1 \in \mathbb{R}$;
    \item $h_{x_1 x_2} = 0$ and $h(x_1,\cdot)$ is strictly convex for any $x_1 \in \mathbb{R}$.
\end{enumerate}
\end{enumerate}
\end{assumption}
As in Proposition 5.8 in \cite{federico&ferrari&schuhumann2020b} (see otherwise Proposition 4.25 at p.\ 92 in \cite{Schuhmann2021}),  under the additional  Assumption \ref{assumption additional for Patrick},  there exist two nonincreasing locally Lipschitz  continuous functions $g_1, \, g_2: \mathbb{R} \to \mathbb{R}$ such that $g_1 < g_2$ and
\begin{equation}  
\label{eq representation free boundary} 
\mathcal{I}_- = \{ x \in \mathbb{R}^2 \, | \, x_2 \leq g_1(x_1) \} \quad \text{and} \quad \mathcal{I}_+ = \{ x \in \mathbb{R}^2 \, | \, x_2 \geq g_2(x_1) \}.
\end{equation}
For each $x \in \mathbb{R}^2$,  recall the definition of $\bar{\tau}_1, \, \bar{\tau}_2$ given in Theorem \ref{theorem Dynkin game connection} and define the stopping times
\begin{equation}
    \bar{\tau}_1^\delta:=\inf \{ t\geq 0 \,|\, X_t^{x+\delta e_1} \in \mathcal{I}_- \}, \quad \bar{\tau}_2^\delta:=\inf \{ t\geq 0 \,|\, X_t^{x+\delta e_1} \in \mathcal{I}_+ \}, \quad \delta\in \mathbb{R}.
\end{equation} 
The Lipschitz continuity of $g_1$ and of $g_2$ allows to prove the following lemma. 
\begin{lemma}\label{lemma continuity stopp times} Under the additional Assumption \ref{assumption additional for Patrick}, for $x\in \mathbb{R}^2$, we have
\begin{equation*}
    \lim_{\delta \to 0} \bar{\tau}_1^\delta = \bar{\tau}_1, \quad \text{and} \quad 
    \lim_{\delta \to 0} \bar{\tau}_2^\delta = \bar{\tau}_2, \quad \mathbb{P}\text{-a.s.}
\end{equation*} 
\end{lemma}
\begin{proof} We only prove the first of the two limits for $\delta \to 0^+$,  since the same limit for $\delta \to 0^-$ follows by identical arguments, and the second limit can be proved in the same way. 
We first observe that, since $g_1$ is finite, we have $\mathbb{P}[\bar{\tau}_1<\infty]=1$. 
{Also, when $\bar \delta >\delta >0$, we have, by convexity of $V$ and by Proposition \ref{proposition V_x1 nondecreasing in x2}, that 
\begin{align*} 
V_{x_1}(x_1 + \bar \delta , X_t^{2, x + \bar \delta e_1}) 
& \geq V_{x_1}(x_1  +  \delta, X_t^{2, x + \bar \delta e_1}) \\ 
& \geq 
V_{x_1}(x_1 + \delta , X_t^{2, x+\delta e_1}) \geq V_{x_1}(x_1, X_t^{2, x+\delta e_1}) \geq V_{x_1}(x_1 , X_t^{2, x}),
\end{align*}
from which we deduce that 
\begin{equation}\label{eq tau delta > tau}
\bar{\tau}_1^{\bar \delta}  \geq \bar{\tau}_1^\delta \geq \bar{\tau}_1, \quad \mathbb{P}\text{-a.s.}
\end{equation}
We continue the proof arguing by contradiction. In light of \eqref{eq tau delta > tau}, suppose that there exists $E \in \mathcal{F}$, with $\mathbb{P}[E]>0$, such that for each $\omega \in E$ there exists $\varepsilon(\omega)>0$ and a sequence $(\delta_j)_{j\in \mathbb{N}}$ (not depending on $\omega$) with $\delta_j >0$ and $\delta_j \to 0$ as $j\to \infty$, for which $\bar{\tau}_1^{\delta_j}(\omega) > \bar{\tau}_1(\omega) + \varepsilon(\omega) $ for each $j \in \mathbb{N}$.} 
Using the representation in \eqref{eq representation free boundary},  (dropping the dependence on $\omega$ to simplify the notation) this is equivalent to
\begin{equation}\label{eq assume by contradiction}
\text{$X_{\bar{\tau}_1}^{2,x}\leq g_1(x_1) $ and $X_{\bar{\tau}_1+s}^{2,x+{\delta_j} e_1} > g_1(x_1 +{\delta_j})$, for each $s \in [0,\varepsilon], \, j \in \mathbb{N}$.} 
\end{equation} 
Notice that, due to the particular structure of the dynamics, we have
\begin{equation}\label{eq patrick explicit SDE}
X_{s}^{2,x+{\delta_j} e_1}=X_{s}^{2,x} +{\delta_j} {b_1^2}\big(e^{b_2^2s}-1 \big)/{b_2^2}, \quad s \geq 0, \ j \in \mathbb{N},
\end{equation}
from which we can write
\begin{align*}
X_{\bar{\tau}_1}^{2,x}&= (X_{\bar{\tau}_1}^{2,x}-X_{\bar{\tau}_1+s}^{2,x})+X_{\bar{\tau}_1+s}^{2,x}\\
&=-\int_0^s(a^2 +b_1^2x_1 +b_2^2X_{\bar{\tau}_1+r}^{2,x})dr - \eta (W_{\bar{\tau}_1+s} - W_{\bar{\tau}_1}) +X_{\bar{\tau}_1+s}^{2,x+{\delta_j} e_1} -{\delta_j} {b_1^2}\big(e^{b_2^2(\bar{\tau}_1+s)}-1 \big)/{b_2^2},  
\end{align*}  
From the latter equality, using \eqref{eq assume by contradiction}, by Lipschitz continuity of $g_1$ {(with Lipschitz constant $L_{g_1}$)}, and pathwise boundedness of $X^{2,x}$ and of $\bar{\tau}_1$, we obtain  
\begin{equation}\label{eq estimate to get contradiction}
 X_{\bar{\tau}_1}^{2,x} \geq 
 -{\delta_j} C - sC - \eta (W_{\bar{\tau}_1+s} - W_{\bar{\tau}_1}) +g_1(x_1) - L_{g_1} {\delta_j}, \text{ for each $s \in [0,\varepsilon], \, j \in \mathbb{N}$,} 
\end{equation}
where the constant $C$ depends on $\bar{\tau}_1$ (which is finite, by assumption) and on $\sup_{r\in [0,\varepsilon]} X_{\bar{\tau}_1 + r}^{2,x} $, but it is independent of $s$ and $j$.
Next, by the law of iterated logarithm (see, e.g., Theorem 9.23 at p.\ 112 in \cite{karatzas&shreve1998}) we find a sequence $(s_k)_{k\in \mathbb{N}}$ converging to zero and $\bar{k}\in \mathbb{N}$ (depending on $\omega$) such that
\begin{equation}\label{eq law iterated logarithm}
- (W_{\bar{\tau}_1+s_k} - W_{\bar{\tau}_1}) \geq \sqrt{s_k} \sqrt{ \log \log (1/s_k)} \geq  \sqrt{s_k}, \quad \text{for each $k\geq\bar{k}$}. 
\end{equation}
Finally, from \eqref{eq estimate to get contradiction} and \eqref{eq law iterated logarithm}, for suitable choice of ${\delta_j}$ and $s_k$, we conclude that
$$
 X_{\bar{\tau}_1}^{2,x} \geq 
 -{\delta_j} (C+L_{g_1})  +  \sqrt{s_k} (\eta  -C \sqrt{s_k}) +g_1(x_1) > g_1(x_1), 
$$
which contradicts \eqref{eq assume by contradiction}, and therefore completes the proof of the lemma. 
\end{proof}
 
\begin{lemma}\label{lemma V C2 patrick}
 Under the additional Assumption \ref{assumption additional for Patrick}, we have $V \in C^2(\mathcal{W})$.
\end{lemma}
\begin{proof} We split the proof in two steps. 
\smallbreak \noindent
\emph{Step 1.} 
Take $z\in \mathcal{W}$ and  $\varepsilon >0$ such that $B_\varepsilon^1 (z) \times B_\varepsilon^2(z)  \subset \mathcal{W}$, where $B_\varepsilon^1 (z):=\{ x_1 \in \mathbb{R} \, | \, |z_1-x_1|<\varepsilon \}$ and $B_\varepsilon^2 (z):=\{ x_2 \in \mathbb{R} \, | \,  |z_2-x_2|<\varepsilon \}$. We prove that $V_{x_2 x_2}, V_{x_1 x_2}$ are locally Lipschitz in $B_\varepsilon^1(z) \times B_\varepsilon^2(z)$ and that  $V_{x_1 x_1}(x_1,\cdot)$ is locally Lipschitz in $B_\varepsilon^2(z)$ for each $x_1 \in B_\varepsilon^1(z)$.

We begin by observing that, under \eqref{eq dynamics patrick}, the HJB equation can be regarded as a second order ordinary differential equation (ODE, in short) in the variable $x_2 \in \mathbb{R}$ depending on the parameter $x_1 \in \mathbb{R}$. In particular, $V$ solves the equation
\begin{equation}\label{eq HJB parametric}
\rho V - \bar{b}^2 V_{x_2} - ({\eta^2}/{2}) V_{x_2 x_2}=h, \quad \text{for a.a.\ $ x_2 \in B_\varepsilon^2(z)$, for each fixed $x_1 \in B_\varepsilon^1(z)$}.    
\end{equation}
Therefore we have $V(x_1,\cdot) \in C^{4;1}(B_{\varepsilon}^2(z))$, for each $x_1 \in  B_\varepsilon^1(z)$. Next, for any $y_1, \, x_1 \in B_\varepsilon^1(z)$ we define the function $W(x_2):=V(y_1,x_2)-V(x_1,x_2), \ x_2 \in B_\varepsilon^2(z)$, which satisfies the ODE
\begin{equation*}
\rho W - \bar{b}^2(y_1,\cdot)W_{x_2} - ({\eta^2}/{2}) W_{x_2 x_2} = F, \quad x_2 \in B_\varepsilon^2(z),  
\end{equation*} 
where 
$
 F=h(y_1,\cdot)-h(x_1,\cdot) + b_1^2 V_{x_2}(x_1,\cdot)(y_1-x_1).  
$
Therefore, by employing Schauder interior estimates (see Theorem 6.2 at p.\ 90 in \cite{gilbarg2001}), we obtain
$$
\| W \|_{C^{2;1}(B_{{\varepsilon}/{2}}^2(z))}\leq C ( \|W \|_{C^{0}(B_{\varepsilon}^2(z))} + \| F \|_{C^{0;1}(B_{{\varepsilon}}^2(z))}).
$$
Moreover, by the $W_{loc}^{2;\infty}$-regularity of $V$ (cf. Theorem \ref{theorem V sol HJB} in Appendix \ref{appendix auxiliary results}), the function $F$ is Lipschitz in $B_\varepsilon^2(z)$ (uniformly for $y_1,x_1 \in B_\varepsilon^1 (z)$). 
Thus, the latter estimate implies that 
$$
\|V(y_1,\cdot)-V(x_1,\cdot)\|_{C^{2;1}(B_{{\varepsilon}/{2}}^2(z))}\leq C |y_1-x_1|,
$$
for a constant $C$ which is independent from $y_1$ and $ x_1$, as long as they are elements of $B_\varepsilon^1(z)$. 
Hence, the functions $V, \, V_{x_2}, \, V_{x_2 x_2} $ are Lipschitz continuous in $B_\varepsilon^1 (z) \times B_{{\varepsilon}/{2}}^2(z)$.

We can therefore compute the weak derivative of \eqref{eq HJB parametric} with respect to $x_1$, obtaining, for each fixed $x_1 \in B_\varepsilon^1(z)$, the ODE 
\begin{equation}\label{eq HJB_x1 parametric}
\rho V_{x_1} - \bar b ^2 V_{x_1 x_2 } - ({\eta^2}/{2}) V_{x_1 x_2 x_2 }=h_{x_1} + b_1^2V_{ x_2}, \quad \text{for a.a.\ $ x_2 \in B_{\varepsilon/2}^2(z)$.}    
\end{equation} 
Since $V_{x_2 x_2}$ is Lipschitz, we have $V_{x_1}(x_1,\cdot) \in C^{3;1}(B_{\varepsilon/2}^2(z))$, for each $x_1 \in  B_\varepsilon^1(z)$. Also, we can again define a  function $W^1(x_2):=V_{x_1}(y_1,x_2)-V_{x_1}(x_1,x_2), \ x_2 \in B_{\varepsilon/2}^2(z)$, which satisfies the elliptic equation 
\begin{equation*}
\rho W^1 - \bar b ^2(y_1,\cdot)W_{x_2}^1 - ({\eta^2}/{2}) W_{x_2 x_2}^1 = F^1, \quad x_2 \in B_{\varepsilon/2}^2(z),  
\end{equation*}
where 
$
F^1=h_{x_1 }(y_1,\cdot)-h_{x_1}(x_1,\cdot) +  b_1^2 (V_{ x_2}(y_1,\cdot)- V_{ x_2}(x_1,\cdot)) + b_1^2V_{x_1 x_2}(x_1,\cdot)  (y_1-x_1).   
$
By employing again Schauder interior estimates, we obtain 
$$
\| W^1 \|_{C^{2;1}(B_{{\varepsilon}/{3}}^2(z))}\leq C ( \|W^1 \|_{C^{0}(B_{\varepsilon/2}^2(z))} + \| F^1 \|_{C^{0;1}(B_{{\varepsilon/2}}^2(z))}).
$$
This, by the local Lipschitz continuity of $V_{x_2}$ and  $V_{x_1 x_2}$ (since we have shown that $V_{x_1 x_2 x_2}$ exists bounded) in the variable $x_2$,  implies that
$$
\|V_{x_1}(y_1,\cdot)-V_{x_1}(x_1,\cdot)\|_{C^{2;1}(B_{{\varepsilon}/{2}}^2(z))}\leq C |y_1-x_1|;
$$
that is, the functions $V_{x_1}, \, V_{x_1 x_2}, \, V_{x_1 x_2 x_2} $ are Lipschitz continuous in $B_\varepsilon^1 (z) \times B_{\varepsilon/3}^2(z)$.

This allows to compute once more  the weak derivative w.r.t.\ $x_1$ in equation \eqref{eq HJB_x1 parametric}, obtaining for each fixed $x_1 \in B_\varepsilon^1(z)$, the ODE 
\begin{equation}\label{eq HJB_(x_1 x_1) parametric}
\rho V_{x_1 x_1} -\bar b ^2 V_{x_1 x_1 x_2} - ({\eta^2}/{2}) V_{x_1 x_1 x_2 x_2 } = h_{x_1 x_1} + 2 b_1^2 V_{x_1 x_2}, \quad \text{for a.a.\ $ x_2 \in B_{\varepsilon/3}^2(z)$.}    
\end{equation}
Therefore, since we have shown that $V_{x_1 x_2}$ is Lipschitz, after employing one more time Schauder interior estimates, we obtain
$$
\| V_{x_1 x_1} \|_{C^{2;1}(B_{{\varepsilon}/{4}}^2(z))}\leq C ( \| V_{x_1 x_1} \|_{C^{0}(B_{\varepsilon/3}^2(z))} + \| h_{x_1 x_1} + 2 b_1^2 V_{x_1 x_2} \|_{C^{0;1}(B_{{\varepsilon/3}}^2(z))}) \leq C, \quad x_1 \in B_\varepsilon^1(z),
$$
for $C$ large enough, not depending on $x_1$.
In particular we deduce that $V_{x_1 x_1}(x_1, \cdot)$ is Lipschitz in $B_{\varepsilon/4}^2(z)$, with Lipschitz constant uniformly bounded for $x_1 \in B_\varepsilon^1(z)$.

\smallbreak \noindent
\emph{Step 2.} We now prove that $V_{x_1 x_1}(\cdot,x_2)$ is continuous in $\mathcal{W}_1(x_2)$ (see Lemma \ref{lemma waiting region nonempty}), for each $x_2 \in \mathbb{R}$. This is done by employing a direct computation to find an expression for $V_{x_1 x_1}$.  

Fix $x \in \mathcal{W}$ and let $\hat{h}:=h_{x_1} + b_1^2 V_{x_2}$ be as in Theorem \ref{theorem connection Dynkin game PDE}.  
For $\delta>0$, from \eqref{eq tau delta > tau} in the proof of Lemma \ref{lemma continuity stopp times}, we have $\bar{\tau}_1^\delta \geq \bar{\tau}_1$. Then, from \eqref{eq patrick explicit SDE} and Theorem \ref{theorem Dynkin game connection}, we write  
\begin{align}\label{eq direct computetion rapporto incrementale}
& \frac{V_{x_1}(x+\delta e_1) -V_{x_1}(x) }{\delta} \leq \frac{G(x+\delta e_1;\bar{\tau}_1^\delta,\bar{\tau}_2)-G(x;\bar{\tau}_1^\delta,\bar{\tau}_2)}{\delta} \\ \notag
 & = \mathbb{E}\bigg[  \int_0^{\bar{\tau}_1^\delta \land \bar{\tau}_2} e^{-\rho t} \bigg(\frac{\hat{h}( X_t^{x+\delta e_1}) - \hat{h}(X_t^x) }{\delta} \bigg) dt \bigg]\\ \notag 
&= \mathbb{E}\bigg[  \int_0^{\bar{\tau}_1 \land \bar{\tau}_2}  \int_0^1e^{-\rho t} \Big( \hat{h}_{x_1}( Z_t^{\delta,r}) +\hat{h}_{x_2}( Z_t^{\delta,r}) {b_1^2}(e^{b_2^2 t} - 1) /{b_2^2} \Big)dr dt \bigg] \\ \notag
&\quad + \mathbb{E}\bigg[  \int_{\bar{\tau}_1 \land \bar{\tau}_2}^{\bar{\tau}_1^\delta \land \bar{\tau}_2}  \int_0^1e^{-\rho t} \Big( \hat{h}_{x_1}( Z_t^{\delta,r}) +\hat{h}_{x_2}( Z_t^{\delta,r}) {b_1^2}(e^{b_2^2 t} - 1) /{b_2^2}\Big) dr dt \bigg]=:M_1^\delta + M_2^\delta, \notag
\end{align} 
where $Z_t^{\delta,r}:=X_t^x+ r(X_t^{x+\delta e_1} - X_t^x)$.
Next, in order to study $M_1^\delta$ and $M_2^\delta$, define
\begin{equation}\label{eq def H}
H(t,y):=\hat{h}_{x_1}( y) +\hat{h}_{x_2}(y) {b_1^2}(e^{b_2^2 t} - 1)/{b_2^2},\quad t\geq 0, \ y\in \mathbb{R}^2.
\end{equation}
Notice that, by \eqref{eq dynamics patrick}, Proposition \ref{proposition V_x1 nondecreasing in x2} (see the discussion in Subsection \ref{subsection Affine drift}) and the convexity of $V$ we have $h_{x_1 x_1},\, b_1^2 h_{x_1 x_2},\, b_1^2 V_{x_1 x_2},\, V_{x_2 x_2} \geq 0$, and hence 
\begin{equation}\label{eq H >0}
H \geq 0.
\end{equation}
Moreover, since $p=2$, from Proposition 2.4 in  \cite{federico&ferrari&schuhmann2019}, for each $\bar{y},\, y \in \mathbb{R}^2$, and $\lambda \in [0,1]$, we have 
\begin{equation}\label{eq estimate on the value function from patrick}
\lambda V(\bar{y}) +(1-\lambda) V(y) -V( \lambda \bar{y} + (1-\lambda) y) \leq K \lambda(1-\lambda)|\bar{y}-y|^2,
\end{equation} 
for some $K>0$. 
{
This semiconcavity estimate, together with  Condition \ref{assumption on h} in Assumption \ref{assumption}, implies that (see, e.g., Proposition 1.1.3 at p.\ 2 in \cite{cannarsa&sinestrani2004}) both $V$ and $h$ have bounded (weak) second order derivatives. 
Then, using the definition of $\hat h$ and the fact that $0 \leq {b_1^2}(e^{b_2^2 t} - 1)/{b_2^2} \leq C$ (since by assumption $b_1^2 > 0$ and $b_2^2<0$),   we obtain
$$
H(t,y)=h_{x_1 x_1}(y) +  b_1^2 V_{x_2 x_1}(y) +(h_{x_1 x_2}(y) + b_1^2 V_{x_2 x_2}(y)) {b_1^2}(e^{b_2^2 t} - 1)/{b_2^2} \leq C,
$$
for any $t\geq 0, \ y\in \mathbb{R}^2$.
Hence, using \eqref{eq H >0}  we conclude that
\begin{equation}\label{eq estimates on second derivatives}
0 \leq H(t,y) \leq C.
\end{equation}} 
By Step 1, the function $H(t,\cdot)$ is continuous in $\mathcal{W}$. 
Moreover, since  $Z^{\delta,r} \to X^x$ for $\mathbb{P}\otimes dt \otimes dr$-a.a.\ $(\omega,t,r) \in \Omega \times [0,\infty) \times (0,1)$, as $\delta \to 0$, we deduce that $H(t,Z_t^{\delta,r}) \to H(t,X_t^x)$, $\mathbb{P} \otimes dt \otimes dr$-a.e.\ as $\delta \to 0$. Therefore, thanks to \eqref{eq estimates on second derivatives}, by the dominated convergence theorem we have 
\begin{equation}\label{eq limit of M1 delta}
\lim_{\delta \to 0^+} M_1^\delta = \mathbb{E}\bigg[ \int_0^{\bar{\tau}_1 \land \bar{\tau}_2} e^{-\rho t} \Big( \hat{h}_{x_1}( X_t^x) +\hat{h}_{x_2}( X_t^x) {b_1^2}(e^{b_2^2 t} - 1)/{b_2^2} \Big)   dt \bigg].
\end{equation}  
Also, by Lemma \ref{lemma continuity stopp times} we have $\mathds{1}_{({\bar{\tau}_1 \land \bar{\tau}_2}, {\bar{\tau}_1^\delta \land \bar{\tau}_2})} \to 0$, $\mathbb{P}$-a.s.\ as $\delta \to 0$. Therefore we can again employ \eqref{eq estimates on second derivatives} and the dominated convergence theorem to conclude that 
\begin{equation}\label{eq limit M2 delta} 
\lim_{\delta \to 0} M_2^\delta = 0.
\end{equation}
Hence, since we already know that $V_{x_1 x_1}$ exists a.e., \eqref{eq direct computetion rapporto incrementale}, \eqref{eq limit of M1 delta} and \eqref{eq limit M2 delta} implies that
\begin{equation}\label{eq derivative from above}
V_{x_1 x_1}(x) \leq  \mathbb{E}\bigg[ \int_0^{\bar{\tau}_1 \land \bar{\tau}_2} e^{-\rho t} \Big( \hat{h}_{x_1}( X_t^x) +\hat{h}_{x_2}( X_t^x)  {b_1^2}(e^{b_2^2 t} - 1)/{b_2^2} \Big)   dt \bigg], \quad \text{a.e.\ in }\mathcal{W}.  
\end{equation}
Also, arguments similar to the one leading to \eqref{eq derivative from above}, allow to estimate $V_{x_1 x_1}$ from below, obtaining 
$$
V_{x_1 x_1}(x) \geq  \mathbb{E}\bigg[ \int_0^{\bar{\tau}_1 \land \bar{\tau}_2} e^{-\rho t} \Big( \hat{h}_{x_1}( X_t^x) +\hat{h}_{x_2}( X_t^x)  {b_1^2}(e^{b_2^2 t} - 1)/{b_2^2} \Big)   dt \bigg], \quad \text{a.e.\ in }\mathcal{W}, 
$$
which, together with \eqref{eq derivative from above}, implies that
\begin{equation}\label{eq V_x1x1 explicit}
V_{x_1 x_1}(x) = \mathbb{E}\bigg[ \int_0^{\bar{\tau}_1 \land \bar{\tau}_2} e^{-\rho t} \Big( \hat{h}_{x_1}( X_t^x) +\hat{h}_{x_2}( X_t^x)  {b_1^2}(e^{b_2^2 t} - 1)/{b_2^2} \Big)   dt \bigg], \quad \text{a.e.\ in }\mathcal{W}.   
\end{equation} 
{
We can finally study the continuity of $V_{x_1 x_1}$ in the variable $x_1$. From \eqref{eq V_x1x1 explicit} we have
\begin{align}\label{eq estimate continuity of Vx1x1}
| V_{x_1 x_1}(x+\delta e_1) & - V_{x_1 x_1}(x)| \\ 
& \quad \quad 
 =  \bigg| \mathbb{E}\bigg[ \int_0^{\bar{\tau}_1^\delta \land \bar{\tau}_2^\delta} e^{-\rho t} H(t,X_t^{x+\delta e_1}) dt  - \int_0^{\bar{\tau}_1 \land \bar{\tau}_2 } e^{-\rho t} H(t,X_t^{x}) dt \bigg] \bigg|=:N^\delta, \notag
\end{align}  
with $H$ defined in \eqref{eq def H}.
Following arguments similar to the ones leading to \eqref{eq limit of M1 delta} and \eqref{eq limit M2 delta}, we can show that
$\lim_{\delta \to 0} N^\delta = 0$.} 
Therefore, taking limits as $\delta \to 0$ in \eqref{eq estimate continuity of Vx1x1}, we deduce that $V_{x_1 x_1}$ is a.e.\ equal to a function which is continuous in the variable $x_1$.   
 
By Step 1, the function $V_{x_1 x_1}(x_1,\cdot)$ is locally Lipschitz continuous, uniformly in $x_1$. 
Thus, by the continuity of $V_{x_1 x_1}(\cdot, x_2)$, we conclude that the function $V_{x_1 x_1}$ is jointly continuous in both variables in $\mathcal{W}$. 
This completes the proof of the lemma. 
\end{proof}

\subsubsection{Characterization of the optimal control} 
In light of Lemma \ref{lemma V C2 patrick}, under the additional Assumption \ref{assumption additional for Patrick}, we can  construct the $\varepsilon$-optimal policies. 
Indeed, by employing the comparison principle to the second order ODE \eqref{eq HJB_(x_1 x_1) parametric} (regarded as an equation in the variable $x_2$, depending on the parameter $x_1$), one still obtains that $V_{x_1 x_1}>0$ in $\mathcal{W}$. 
This, together with the fact that $V_{x_1} \in C^1(\mathcal{W})$ (by Lemma \ref{lemma V C2 patrick}), allows to show that $S_\varepsilon$ is a $C^1$ curve in $\mathbb{R}^2$ and that the vector field $-e_1V_{x_1}/|V_{x_1}|$ is $C^1$ on $S_\varepsilon$, and nontangential to $S_\varepsilon$. 
All the assumptions in CASE 2 at p.\ 557 in  \cite{dupuis&ishii1993} (up to the boundedness of $\mathcal{W}$) are then satisfied, and we can therefore employ (a suitable extension to unbounded domains of) Theorem 5.1 at p.\ 572 in \cite{dupuis&ishii1993} in order to find the $\varepsilon$-optimal controls as in
Lemma \ref{lemma epsilon Skorokhod problems}. Finally, all the arguments in Section \ref{section proof of the main results} can be repeated in the case in which $\sigma \sigma^{\textbf{\tiny{$\top$}}}$ is degenerate. Overall, we have proved the following result. 

\begin{theorem}\label{theorem degenerate}
Consider the degenerate singular control problem described in \eqref{eq dynamics patrick}, with $h$ satisfying Condition \ref{assumption on h} in Assumption \ref{assumption} and Assumption \ref{assumption additional for Patrick}. Then, the thesis of Theorem \ref{theorem main characterization} holds. 
\end{theorem}

Concluding, with respect to \cite{federico&ferrari&schuhmann2019}, we require in addition that $h_{x_1 x_1}>0$ and that Assumption \ref{assumption additional for Patrick} is satisfied. 
In this case, Theorem \ref{theorem degenerate} applies, and the construction of the optimal control discussed in Section 7 in  \cite{federico&ferrari&schuhmann2019} can be provided. We underline that in \cite{federico&ferrari&schuhmann2019} a construction of an optimal control is given in weak formulation, under a quite strong requirement on the running cost $h$. We refer to Proposition 7.3 in \cite{federico&ferrari&schuhmann2019} for more details.

\appendix
\section{On the HJB equation}\label{appendix auxiliary results}
In this subsection we prove that $V$ is a {convex} solution (in the a.e.\ sense) to the related HJB equation. 
The argument of the proof exploits the penalization method introduced in \cite{evans1979} for bounded domains (see also \cite{hynd2013} and the references therein), which we extend to $D$ thanks to suitable semiconcavity estimates
, in the spirit of \cite{buckdahn&cannarsa2010}. Although this result is somehow classical, we have not been able to find versions that exactly fit our setting, and we therefore provide its proofs in the following.    

\begin{theorem}
\label{theorem V sol HJB}
The value function $V$ is a {convex} $W_{loc}^{2;\infty}(D)$-solution to the equation  
\begin{equation}
\label{eq appendix HJB equation}
\max \{ \rho V-\mathcal{L}V - {h} , | V_{x_1}| - 1 \} = 0,  \quad \text{a.e.\ in } D.   
\end{equation}
\end{theorem}
\begin{proof}  
We divide the proof in five steps. 
\smallbreak \noindent 
\textit{Step 1.}
{
We first prove the convexity of $V$. 
Since the argument is similar to the one already used in the proof of Lemma \ref{lemma construction optimal policy}, we limit ourself to provide a sketch of the proof. \\ \indent
Take $x, \, \bar x \in D$ and $v, \, \bar v \in \mathcal V$. 
For $\lambda \in [0,1]$, set $x^\lambda := \lambda x + (1-\lambda) \bar x$ and  $v^\lambda := \lambda v + (1-\lambda) \bar v$. 
Define the processes
$$
Y:= X^{ x^\lambda;  v^\lambda} \quad \text{and} \quad Z := \lambda X^{x;v}+(1-\lambda) X^{\bar x; \bar v }. 
$$ 
Since the drift $\bar{b}^1$ is affine, we have  $Y^1 = Z^1$. 
Moreover, by convexity of $b^i$,  we find 
\begin{align*} 
Z_t^{i} 
& ={x}_{i}^\lambda +\int_0^t(\lambda b^i(X_s^{x;v})+(1-\lambda) b^i( X_s^{\bar x; \bar v }) ) ds + \int_0^t \bar \sigma (Z_s^{i}) d W_s^{i} \\
& \geq  {x}_{i}^\lambda +\int_0^t b^{i}(Y_s^{1},Z_s^{i} ) ds +\int_0^t \bar \sigma (Z_s^{i}) d W_s^{i}, 
\end{align*}
while $Y_t^{i}= {x}_{i}^\lambda + \int_0^t b^{i}(Y_s^{1},Y_s^{i}) ds +  \int_0^t \bar \sigma (Y_s^{i}) d W_s^{i}$.  
This, by the comparison principle for SDE (see \cite{ikeda1977}), implies that $ Y_t^{i} \leq Z_t^{i}, \ \text{for each } t\geq 0, \ \mathbb{P}\text{-a.s.}$.
Hence, by the monotonicity of $h$ in $x_i$ we find 
\begin{equation}\label{eq J is convex}
    J(x^\lambda; \lambda v +(1-\lambda) \bar v) \leq \lambda J(x; v) + (1-\lambda) J(\bar{x}; \bar v), 
\end{equation}
where we have also used that $h$ is convex and that $|\lambda v+(1-\lambda) \bar v |_t  \leq \lambda |v|_t + (1-\lambda) | \bar v |_t$, and that $e^{-\rho t}$ is positive and decreasing. 
Notice that, if either $x _1 \ne \bar x_1 $ or $v \ne \bar v$ and $\lambda \in (0,1)$, then, by strict convexity of $h$ in the variable $x_1$, we would obtain a strict inequality in \eqref{eq J is convex}. 
In particular, $J$ is strictly convex in  $v$ (from which the uniqueness of the optimal control, claimed in Remark \ref{lemma appendix existence optimal controls}, follows). \\ \indent
Next, take $\delta >0$ and assume the processes $v$ and $\bar v$ to be $\delta$-optimal for $x$ and $\bar x$, respectively;
i.e., assume $J(x;v) \leq V(x) +\delta$ and $J(\bar x; \bar v) \leq V( \bar x) +\delta$.
By \eqref{eq J is convex}, we obtain
$$
V(x^\lambda) \leq J(x^\lambda; \lambda v +(1-\lambda) \bar v) \leq \lambda J(x; v) + (1-\lambda) J(\bar{x}; \bar v) \leq \lambda V(x) +  (1-\lambda) V(\bar x) + \delta,
$$
which, by arbitrariness of $\delta$, gives the convexity of $V$.
}

\smallbreak \noindent 
\textit{Step 2.}
Let us start by introducing a family of penalized versions of the HJB equation \eqref{eq appendix HJB equation}. Let $\beta \in C^\infty(\mathbb{R})$ be a convex nondecreasing function with $\beta(r)=0$ if $r\leq 0$ and $\beta(r)=2r-1$ if $r\geq 1$. 
For each $\varepsilon > 0$, let $V^\varepsilon$ be the 
the value function of the penalized control problem 
\begin{equation}
\label{eq control problem penalized}
V^\varepsilon(x):=\inf_{ \alpha \in \mathcal{U}_\varepsilon } J_{\varepsilon}(x;\alpha) := \inf_{ \alpha \in \mathcal{U}_\varepsilon } \mathbb{E} \bigg[ \int_0^\infty  e^{-\rho t}(h(X_t^{x;\alpha}) + |\alpha_t^1| + \alpha_t^2 ) dt \bigg], \quad x\in D, 
\end{equation}
where $\mathcal{U}_\varepsilon$ is the set of $E_\varepsilon$-valued $\mathbb{F}$-progressively measurable processes, with
$
E_\varepsilon:= \{ \alpha=(\alpha^1, \alpha^2)\in \mathbb{R}\times [0, \infty)  \,| \, |\alpha^1|r - \frac{1}{\varepsilon}\beta(r(r+2))\leq \alpha^2\leq \frac{1}{\varepsilon}, \ \forall r>0 \}  
$. 
Here, with slight abuse of notation, $X^{x;\alpha}$ denotes the solution to
$
dX_t^{x;\alpha} = (b(X_t^{x;\alpha})+ e_1 \alpha_t^1) dt + \sigma dW_t  , \ t \geq 0, \ X_0^{x:\alpha}=x.
$
We point out that, under Condition \ref{ass sigma geometric} in Assumption \ref{assumption}, a result analogus to Lemma \ref{lemma geometric X>0} holds.
Arguing as in \cite{hynd2013} (throught a localization argument), we have that $V^\varepsilon$ is a $C^2(D)$
 solution to the partial differential equation 
\begin{equation}\label{eq appendix penalized HJB}
    \rho V^\varepsilon-\mathcal{L}V^\varepsilon + \frac{1}{\varepsilon}\beta ( (V_{x_1}^\varepsilon)^2 - 1 ) = h, \quad \text{in } D.
\end{equation} 
Moreover, the family $(V^\varepsilon)_{\varepsilon \in (0,1)}$ provides a monotone approximation of $V$; that is,
\begin{equation}
\label{eq app pointwise limit to V}
\lim_{\varepsilon \to 0} V^\varepsilon(x) = V(x) \quad \text{ and $\quad V^\varepsilon(x) \geq  V^{\bar{\varepsilon}}(x),$ for $\bar \varepsilon < \varepsilon, \quad $for each $x\in D$.}
\end{equation}
Take indeed $x\in D$. Observe that, for each $\varepsilon >0$, we have $V^\varepsilon(x) \geq V (x)$, as $\alpha^2 \geq 0$. 
Moreover, as in Theorem 2.2.\ in \cite{chow&menaldi&robin1985}, one can show that for each $\delta >0$ there exists a Lipschitz admissible process $w\in \mathcal{V}$ such that $J(x;w) \leq V(x) + \delta/2$. 
Since $w$ is Lipschitz, we have $dw_t=\alpha_t^1 dt$, for some bounded progressively measurable process $\alpha^1$. Then, defining $\alpha^2_t=\rho \delta /2$, we can find $\bar{\varepsilon}>0$ such that $\alpha:=(\alpha^1,\alpha^2) \in \mathcal{U}_\varepsilon$ for each $\varepsilon \in(0,\bar{\varepsilon})$. Moreover, with this choice of $\alpha$, we have that $J_\varepsilon(x;\alpha) \leq J(x;w) + \delta/2 \leq V(x) + \delta$, for each $\varepsilon \in (0, \bar{\varepsilon})$, completing the proof of \eqref{eq app pointwise limit to V}.
\smallbreak \noindent
\textit{Step 3.} In this step we show that, under Condition \ref{ass sigma constant} in Assumption \ref{assumption}, for each $R>0$, there exists a constant $C_R$ such that
\begin{equation}
\label{eq semiconcavity}
0 \leq \lambda V^\varepsilon(\bar{x}) + (1-\lambda)V^\varepsilon( x)-V^\varepsilon( \lambda \bar{x} +(1-\lambda)x)  \leq C_R \lambda (1-\lambda) |\bar{x}-x|^2, 
\end{equation} 
for each $\lambda \in [0,1]$,  $\bar{x}, x \in B_R$ and $\varepsilon >0$. 
By the same arguments leading to the convexity of $V$ {(see Step 1 in this proof)}, we have that, for each $\varepsilon>0$, the function $V^\varepsilon$ is convex.  Therefore, we only need to prove the last inequality in \eqref{eq semiconcavity}. 
Take  $\bar{x},\,x \in B_R$, $\lambda \in  [0,1]$ and set $x^\lambda:=\lambda \bar{x} + (1-\lambda)x$. 
Fix $\varepsilon >0$, an arbitrary $\delta>0$, and let $\alpha \in \mathcal{U}_\varepsilon$ be a $\delta$-optimal control for the problem \eqref{eq control problem penalized} with initial condition $x^\lambda$; that is, $J_\varepsilon(x^\lambda;\alpha) \leq V^\varepsilon (x^\lambda)+\delta$. 
Since $\alpha$ is not necessarily optimal for $x$ or $\bar{x}$, we have
\begin{align*}
\lambda V^\varepsilon(\bar{x})  & + (1-\lambda)V^\varepsilon( x)-V^\varepsilon( x^\lambda ) -\delta \\
& \leq  \lambda J_\varepsilon ( \bar{x};\alpha) + (1-\lambda)J_\varepsilon( x; \alpha)-J_\varepsilon(x^\lambda; \alpha)  \\
& \leq \mathbb{E} \bigg[ \int_0^\infty e^{-\rho t} \big( \lambda h(X_t^{ \bar{x};\alpha}) + (1-\lambda)h(X_t^{ x; \alpha})-h(X_t^{x^\lambda; \alpha}) \big) dt \bigg].
\end{align*}
Setting $Z_t := \lambda X_t^{ \bar{x};\alpha} + (1-\lambda) X_t^{ x; \alpha}$, using Condition \ref{assumption on h} in Assumption \ref{assumption} , we continue the latter chain of estimates to find  
\begin{align}\label{estimate A + B} 
\lambda V^\varepsilon(\bar{x})  & + (1-\lambda)V^\varepsilon( x)-V^\varepsilon( x^\lambda ) -\delta \\ \notag
& \leq \mathbb{E} \bigg[ \int_0^\infty e^{-\rho t} \big( \lambda h(X_t^{ \bar{x};\alpha}) + (1-\lambda)h(X_t^{ x; \alpha})-h(Z_t) \big) dt \bigg] \\ \notag
& \quad + \mathbb{E} \bigg[ \int_0^\infty e^{-\rho t} \big( h(Z_t)-h(X_t^{x^\lambda; \alpha}) \big) dt \bigg] \\ \notag 
& \leq C \lambda (1-\lambda) \mathbb{E} \bigg[ \int_0^\infty e^{-\rho t} \big(1+ \big| X_t^{x; \alpha} \big|^{p-2} + \big| X_t^{\bar{x}; \alpha}\big|^{p-2} \big) \big| X_t^{ \bar{x};\alpha} - X_t^{x;\alpha} \big|^2 dt \bigg] \\ \notag 
& \quad + C \mathbb{E} \bigg[ \int_0^\infty e^{-\rho t} \big(1+ \big| Z_t \big|^{p-1} + \big| X_t^{x^\lambda; \alpha}\big|^{p-1} \big) \big| Z_t- X_t^{x^\lambda; \alpha} \big| dt \bigg] \\ \notag 
& =: M_1 + M_2.  
\end{align}   
We will now estimate $M_1$ and $M_2$ separately.

First of all, by a standard use of Gr\"onwall's inequality, we find
\begin{equation}\label{estimate B}
 \big| X_t^{ \bar{x};\alpha} - X_t^{ x; \alpha} \big|  \leq C e^{\bar{L} t} |\bar{x} - x|. 
\end{equation} 
When $p= 2$, from \eqref{estimate B} and our assumptions on $\rho$, we immediately deduce that   
\begin{equation}
\label{eq esimates M2 finale}
M_1 \leq  C_R \lambda (1-\lambda) |\bar{x} - x|^2,
\end{equation} 
as desired. On the other hand, if $p > 2$, set $p':=(2p-1)/2$. Defining $q:=p'/(p-2)$ and denoting by $q^*$ its conjugate, we can employ H\"older's inequality and obtain
\begin{align*}
M_1 &\leq C \lambda(1-\lambda)|\bar{x} - x|^2 \bigg( \mathbb{E} \bigg[ \int_0^\infty e^{(2 \bar{L}-\rho(1-\frac{1}{q}) )q^* t} dt \bigg] \bigg)^{\frac{1}{q^*}} \bigg( \mathbb{E} \bigg[ \int_0^\infty e^{-\rho t}  \big(1+ \big| X_t^{x; \alpha} \big|^{p'} + \big| X_t^{\bar{x}; \alpha} \big|^{p'} \big) dt \bigg]  \bigg)^{\frac{1}{q}} \\
& \leq C \lambda(1-\lambda)(1+|x|^p + |\bar{x}|^p)^{\frac{1}{q}} |\bar{x} - x|^2 \leq C_R \lambda (1-\lambda) |\bar{x} - x|^2,
\end{align*} 
where we have used the requirements on $\rho$ in Condition \ref{ass sigma constant} in Assumption \ref{assumption}, and the estimate {as in Lemma \ref{lemma estimate SDE}}, which holds also for the penalized problem. 
      
We next estimate $M_2$. Since the gradient $D \bar b$ is Lipschitz we have the estimate  (see, e.g., Proposition 1.1.3 at p.\ 2 in \cite{cannarsa&sinestrani2004})
\begin{equation*}
    |\lambda \bar b(\bar{y}) +(1-\lambda) \bar b(y) - \bar b( \lambda \bar{y} + (1-\lambda) y) |  \leq C \lambda(1-\lambda)|\bar{y}-y|^2, \quad \text{for each $\bar{y},\, y\in \mathbb{R}^{2}.$} 
\end{equation*}
This, together with the  Lipschitz property of $b$, allows to obtain 
\begin{align}\label{eq z-x dinamics}
\big| X_t^{x^\lambda; \alpha} - Z_t  \big|  \leq &   \int_0^t \big| \bar b(X_s^{x^\lambda; \alpha}) -\lambda \bar b(X_s^{ \bar{x};\alpha}) - (1-\lambda) \bar b(X_s^{ x; \alpha})\big| ds \\ \notag
&\leq  \bar{L}  \int_0^t  \big( \big| X_s^{x^\lambda; \alpha} - Z_s  \big| + C   \lambda(1-\lambda ) \big| X_s^{ \bar{x};\alpha} - X_s^{ x; \alpha}\big|^2 \big) ds , \\ \notag
& \leq  \bar{L}  \int_0^t  \big( \big| X_s^{x^\lambda; \alpha} - Z_s  \big| + C  \lambda(1-\lambda )|\bar{x} - x|^2  e^{2 \bar{L} s}  \big)   ds , \\ \notag 
& \leq  C  \lambda(1-\lambda )|\bar{x} - x|^2  e^{2 \bar{L} t }   + \bar{L} \int_0^t   \big| X_s^{x^\lambda; \alpha} - Z_s  \big|    ds .  \notag 
\end{align} 
The latter estimate, after employing Gr\"onwall's inequality, leads to
\begin{equation}\label{estimate X -Z}
 \big|X_t^{x^\lambda; \alpha} - Z_t  \big|  \leq C \lambda(1-\lambda ) e^{3 \bar{L} t } |\bar{x} -x |^2.  
\end{equation}
Defining $q:=p'/(p-1)$ and denoting by $q^*$ its conjugate, we can again employ H\"older's inequality and \eqref{estimate X -Z} in order to obtain
\begin{align*} 
M_2 &\leq  C \lambda(1-\lambda )  |\bar{x} -x |^2 \mathbb{E} \bigg[ \int_0^\infty e^{(3\bar{L}-\rho) t} \big(1+ \big| Z_t \big|^{p-1} + \big| X_t^{x^\lambda; \alpha}\big|^{p-1} \big) dt \bigg]\\
& \leq C \lambda(1-\lambda)|\bar{x} - x|^2 \bigg( \mathbb{E} \bigg[ \int_0^\infty e^{(3 \bar{L}-\rho(1-\frac{1}{q}))q^* t} dt \bigg] \bigg)^{\frac{1}{q^*}} \bigg( \mathbb{E} \bigg[ \int_0^\infty e^{-\rho t}  \big(1+ \big| X_t^{x^\lambda; \alpha} \big|^{p'} + \big| Z_t \big|^{p'} \big) dt \bigg]  \bigg)^{\frac{1}{q}} \\
& \leq C \lambda(1-\lambda)(1+|x|^p + |\bar{x}|^p)^{\frac{1}{q}} |\bar{x} - x|^2 \leq C_R \lambda (1-\lambda) |\bar{x} - x|^2,
\end{align*}  
where we have used the estimate as in Lemma \ref{lemma estimate SDE} and the requirements on $\rho$ in Condition \ref{ass sigma constant} in Assumption \ref{assumption}.
This, together with \eqref{eq esimates M2 finale} and \eqref{estimate A + B}, thanks again to the arbitrariness of $\delta$, completes the proof of \eqref{eq semiconcavity}.

\smallbreak \noindent
\textit{Step 4.} We now prove the estimate \eqref{eq semiconcavity} under Condition \ref{ass sigma geometric} in Assumption \ref{assumption}.  To simplify the notation, we assume $d=2$, the generalization to $d>2$ being straightforward. 
We proceed from \eqref{estimate A + B}, and we estimate $M_1$ and $M_2$ from above. To this end, define the processes 
$$
{E}_t:= \exp[(b_1^1-\sigma^2/2)t + \sigma  W_t^1] \quad \text{and}\quad \hat{E}_t:= \exp[(\bar{L}-\sigma^2/2)t + \sigma W_t^2].
$$

We first estimate $M_1$. Observe that 
\begin{equation}\label{eq geom appe delta x1} 
    |X_t^{1,\bar{x};\alpha} -X_t^{1,{x};\alpha}|=|\bar{x}_1 -x_1| E_t,
\end{equation}
which we will use to estimate $|X_t^{2,\bar{x};\alpha} -X_t^{2,{x};\alpha}|$.  Define the process $\Delta$ as the solution to the SDE
$$
d\Delta_t = \bar{L}(|X_t^{1,\bar{x};\alpha} -X_t^{1,{x};\alpha}| + \Delta_t)dt + \sigma \Delta_t dW_t^2, \quad t \geq0, \quad  \Delta_0=|\bar{x}_2 -x_2|.  
$$  
Through a comparison principle, it is easy to check that 
$|X_t^{2,\bar{x};\alpha} -X_t^{2,{x};\alpha}|\leq \Delta_t,$ so that, using \eqref{eq geom appe delta x1} and the explicit expression for $\Delta$, we get
\begin{equation}\label{eq geom appe delta x2} 
    |X_t^{2,\bar{x};\alpha} -X_t^{2,{x};\alpha}| \leq C |\bar{x} -x| \hat{E}_t \big[ 1 + \begin{matrix} \int_0^t \end{matrix} E_s/\hat{E}_s ds \big]=:C |\bar{x} -x| P_t. 
\end{equation}
When $p=2$, the  estimate of $M_1$ can be easily deduced from \eqref{eq geom appe delta x1} and \eqref{eq geom appe delta x2}.
For $p>2$, by employing H\"older's inequality with exponent $q =p'/(p-2)$, we find
\begin{align}\label{eq geom app semifinal M1}
&\mathbb{E} \bigg[ \int_0^\infty e^{-\rho t} \big(1+ \big| X_t^{x; \alpha} \big|^{p-2} + \big| X_t^{\bar{x}; \alpha}\big|^{p-2} \big) (E_t^2 +P_t^2) dt \bigg] \\ \notag
& \leq C \bigg( \int_0^\infty e^{-\rho t} \mathbb{E}\big[1+\big| X_t^{x; \alpha} \big|^{p'} + \big| X_t^{\bar{x}; \alpha}\big|^{p'} \big] dt \bigg)^\frac{1}{q} \bigg(  \int_0^\infty e^{-\rho(1-\frac{1}{q})q^* t} \mathbb{E} \big[ E_t^{2q*}+ P_t^{2q*}\big] dt\bigg)^{\frac{1}{q^*}} \\ \notag
& \leq C ( 1+|x|^p + |\bar x|^p)^\frac{1}{q} \bigg(  \int_0^\infty e^{-\rho(1-\frac{1}{q})q^* t} \mathbb{E} \big[ E_t^{2q*}+ P_t^{2q*}\big] dt\bigg)^{\frac{1}{q^*}} \leq C_R < \infty.
\end{align} 
Here, we have also used \eqref{eq geom estimate SDE full}, while  the finiteness of the latter integral follows, after some elementary computations, from the requirements on $\rho$ in Condition \ref{ass sigma geometric} in Assumption \ref{assumption}.
Finally, by \eqref{eq geom appe delta x1}, \eqref{eq geom appe delta x2} and \eqref{eq geom app semifinal M1}, we obtain 
\begin{equation}\label{eq geo M1}
M_1 \leq  C_R \lambda (1-\lambda)|\bar{x} -x|^2.    
\end{equation}

We next estimate $M_2$. Since $\bar{b}^1$ is affine, we have $Z^1 = X^{1,x^\lambda; \alpha}$. Similarly to \eqref{eq z-x dinamics}, one has 
$$
Z_t^2 - X_t^{2,x^\lambda; \alpha}  \leq \int_0^t (C\lambda(1-\lambda )|X_s^{2,\bar{x}; \alpha} - X_s^{2,x; \alpha}|^2 + \bar{L} \big| X_s^{x^\lambda; \alpha} - Z_s  \big|)    ds + \sigma \int_0^t   (Z_s^2 - X_s^{2,x^\lambda; \alpha} )dW_s^2 . 
$$
Therefore, employing again a comparison principle and using \eqref{eq geom appe delta x2}, we see  that
\begin{equation}\label{eq geom app estimate z-x} 
|Z_t^2 - X_t^{2,x^\lambda; \alpha} | \leq C\lambda (1-\lambda) \hat{E}_t \int_0^t \frac{|X_s^{2,\bar{x}; \alpha} - X_s^{2,x; \alpha}|^2}{\hat{E}_s} ds  \leq  C\lambda (1-\lambda) |\bar{x} -x|^2  \int_0^t \frac{\hat{E}_t}{\hat{E}_s} P_s^2 ds. 
\end{equation}
Also,  H\"older's inequality with exponent $q =p'/(p-1)$ yields
\begin{align}\label{eq geom app semifinal M2}
&\mathbb{E} \bigg[ \int_0^\infty e^{-\rho t} \big(1+ \big| Z_t \big|^{p-1} + \big| X_t^{x^\lambda; \alpha}\big|^{p-1} \big)\int_0^t \frac{\hat{E}_t}{\hat{E}_s} P_s^2 ds\, dt \bigg]\\ \notag
& \leq C\bigg( \int_0^\infty e^{-\rho t} \mathbb{E}\big[1+\big| X_t^{x; \alpha} \big|^{p'} + \big| X_t^{\bar{x}; \alpha}\big|^{p'} \big] dt \bigg)^\frac{1}{q}  \bigg( \mathbb{E} \bigg[ \int_0^\infty e^{-\rho(1-\frac{1}{q})q^* t} \bigg( \int_0^t \frac{\hat{E}_t}{\hat{E}_s} P_s^2 ds \bigg)^{q*} dt \bigg] \bigg)^{\frac{1}{q^*}}   \\ \notag
& \leq C (  1+|x|^p + |\bar x|^p)^\frac{1}{q}  \bigg( \mathbb{E} \bigg[ \int_0^\infty e^{-\rho(1-\frac{1}{q})q^* t} \bigg( \int_0^t \frac{\hat{E}_t}{\hat{E}_s} P_s^2 ds \bigg)^{q*} dt \bigg] \bigg)^{\frac{1}{q^*}} \leq C_R<\infty,  \notag
\end{align} 
Again, here we have also employed \eqref{eq geom estimate SDE full}, while  the finiteness of the latter integral follows, after some elementary computations, from the requirements on $\rho$ in Condition \ref{ass sigma geometric} in Assumption \ref{assumption}. Finally, combining \eqref{eq geom app estimate z-x} and \eqref{eq geom app semifinal M2}, we obtain $M_2 \leq  C_R \lambda (1-\lambda)|\bar{x} -x|^2$, which, together with \eqref{eq geo M1} and \eqref{estimate A + B}, implies \eqref{eq semiconcavity}.
 
\smallbreak \noindent
\textit{Step 5.}   
From \eqref{eq semiconcavity} we deduce that, for each bounded open set $B\subset D$, there exists a constant $C_B>0$ such that
\begin{equation}
\label{eq appendix Sobolev estimates}
\sup_{\varepsilon \in (0,1)} \| V^\varepsilon \|_{W^{2;\infty}(B)} \leq C_B.  
\end{equation}
This estimate allows, by mean of classical arguments (exploiting  Sobolev compact embedding theorem of $W^{2;q}(B)$ into  $C^1(B)$ for $q>2+d$ and the weak compactness of the closed unit ball in  $W^{2;2}(B)$) to improve the convergence in \eqref{eq app pointwise limit to V}. Indeed (on each subsequence) we now have: 
\begin{align}\label{eq appendix Sobolev limits} 
     & (V^\varepsilon,DV^\varepsilon) \text{ converges to $(V,DV)$ uniformly  in $ B$}; \\ \notag
     & D^2V^\varepsilon \text{ converges to $ D^2V$ weakly  in $L^2( B)$}.
\end{align}
Let us now prove that $V$ solves the HJB equation \eqref{eq appendix HJB equation}. First of all observe that, from \eqref{eq appendix penalized HJB} and \eqref{eq appendix Sobolev estimates}, (unless to take a larger $C_B$) we have 
\begin{equation}
\label{eq app estimates beta}
\frac{1}{\varepsilon} \beta ( (V_{x_1}^\varepsilon)^2 - 1 ) \leq C_B, \quad \text{in $B$}.
\end{equation}
Hence, taking pointwise limits in \eqref{eq appendix penalized HJB} and \eqref{eq app estimates beta}, we obtain 
$$ 
\rho V-\mathcal{L}V - {h} \leq 0,  \quad \text{and} \quad 
|V_{x_1}| - 1 \leq 0 \quad  \text{a.e.\ in } D. 
$$ 
Suppose now that the inequality $|V_{x_1}| - 1\leq 0$ is strict in $\bar{x} \in D$. By continuity of  $V_{x_1}$, there exist $\eta>0$ and a neighborhod $N$ of $\bar{x}$ such that $|V_{x_1}(x)| - 1 \leq - \eta$ for each $x \in N$. Therefore, by uniform convergence in $N$, for each $\varepsilon$ small enough we have $|V_{x_1}^{\varepsilon}(x)| - 1 \leq - \eta/2$, and therefore, by \eqref{eq appendix penalized HJB}, that $\rho V^\varepsilon   -\mathcal{L}V^\varepsilon - {h} = 0$ in $N$. Passing again to the limit, this in turn implies that $\rho V -\mathcal{L}V - {h} = 0$ in $N$,  completing the proof of the theorem.

\end{proof}

\section{Proof of Lemma \ref{lemma waiting region nonempty} and of Proposition \ref{propo charact jumps}}\label{appendix proof propositions kruk}
\subsection{Proof of Lemma \ref{lemma waiting region nonempty}}
We give a proof for $d=2$, the case $d>2$ is analogous.
The set $\mathcal{W}_1(z)$ is an open interval, since,  by convexity of $V$, the function $V_{x_1} (\cdot, z)$ is nondecreasing. We therefore show that the set $\mathcal{W}_1(z)$ is nonempty. 
Suppose that Condition \ref{ass sigma constant} in Assumption \ref{assumption} is in place. Arguing by contradiction,
if $\mathcal{W}_1(z) = \emptyset $, then, by the continuity of $V_{x_1}$, we  have  $V_{x_1}(\cdot,z)= 1$ or $V_{x_1}(\cdot,z)= -1$. If  $V_{x_1}(\cdot,z) = 1$, we have $V(x_1,z) +\kappa_2\geq V(x_1,z)-V(y,z)  = \int_y^{x_1} V_{x_1}(r,z)dr = x_1 - y \to \infty$ as $y \to -\infty$. Therefore $V(x_1,z)=\infty$, contradicting the finiteness of $V$ (see Theorem \ref{theorem V sol HJB} in Appendix \ref{appendix auxiliary results}). In the same way, we can not have that $V_{x_1}(\cdot,z)= -1$, which implies $\mathcal{W}_1(z) \ne \emptyset$.

 On the other hand, suppose that Condition \ref{ass sigma geometric} in Assumption \ref{assumption} holds. 
 Arguing by contradiction, we assume that $\mathcal{W}_1(z)$ is empty. From the continuity of $V_{x_1}$, we have  $V_{x_1}(\cdot,z)= 1$ or $V_{x_1}(\cdot,z)= -1$. 
 If  $V_{x_1}(\cdot,z) = - 1$, then  we have $V(x_1,z)+ \kappa_2 \geq V(x_1,z)-V(y,z) = - \int_{x_1}^{y} V_{x_1}(r,z)dr = y-x_1 \to \infty$ as $y \to \infty$. Therefore $V(x_1,z)=\infty$, contradicting the finiteness of $V$. 
We therefore assume that $V_{x_1}(\cdot,z)= 1$ and we show that this leads anyway to a contradiction. 
 
For a generic $x_1 \in \R$ with $0<x_1 < x_1^*$, let $v \in \mathcal{V}$ be optimal for the initial condition $x:=(x_1,z)$, with $dv=\gamma d|v|$. By repeating the arguments leading to \eqref{eq to compare support of v} in the proof of Proposition \ref{proposition the cont parts acts on the boundary}, an application of It\^o's formula leads to 
$$ 
\mathbb{E} \bigg[  \int_{[0,\infty)} e^{-\rho t} ( 1 + V_{x_1}(X_{t-}^{{x};{v}}) {\gamma}_t ) d|{v}|_t \bigg] \leq 0. $$ 
This in turn implies, using $0\leq 1-|V_{x_1}| \leq 1+V_{x_1} u$ for all $u \in \mathbb{R}$ with $|u|=1$, that  
$$
\mathbb{E}[|v|_0(1+\gamma_0)]=\mathbb{E}[|v|_0(1+\gamma_0V_{x_1}(X_{0-}^{x;v})] \leq \mathbb{E} \bigg[  \int_{[0,\infty)} e^{-\rho t} ( 1 + V_{x_1}(X_{t-}^{{x};{v}}) {\gamma}_t ) d|{v}|_t \bigg] \leq 0,
$$
where the first equality follows from the assumption $V_{x_1}(\cdot,z)=1$. Also, since $|\gamma_0|=1$, $\mathbb{E}[|v|_0(1+\gamma_0)]\geq0$, which combined with the latter inequality gives $\mathbb{E}[|v|_0(1+\gamma_0)]=0$. 
In other words, a possible jump at time zero must be of negative size. Therefore, since $x_1 < x_1^*$, as in the proof of Lemma \ref{lemma geometric X>0}, we deduce that $v$ has no jump at time zero; that is, \begin{equation}\label{eq jump to the left}
\mathbb{P}[|v|_0>0]=0. 
\end{equation}

Next, fix $0< x_1< y_1<x_1^*$ and set $x=(x_1,z)$ and $y=(y_1,z)$. Since we are assuming that $V_{x_1}(\cdot,z)=1$, we have \begin{equation}\label{eq geom value function linear}
V(y)-V(x)=\int_{x_1}^{y_1} V_{x_1}(r,z)dr={y_1}-{x_1}.
\end{equation} 
Next, denote by $v$ and $w$ the optimal control for the initial conditions $x$ and $y$, respectively. By \eqref{eq jump to the left}, neither $v$ or $w$ has a jump a time zero, so that, using \eqref{eq geom value function linear}, we find
$$
J(y;v+x_1-y_1)=J(x;v) + |x_1 - y_1| = V(x) +y_1-x_1 =V(y).
$$
This, by uniqueness of the optimal control implies that $w=v + x_1 -y_1$, so that, since $x_1 <y_1$, the control $w$  has a negative jump at time zero, contradicting \eqref{eq jump to the left}. 

Therefore also the assumption $V_{x_1}(\cdot,z)=1$ leads to a contradiction, completing the proof of  Lemma \ref{lemma waiting region nonempty} under Condition \ref{ass sigma geometric} in Assumption \ref{assumption}.

\subsection{Proof of Proposition \ref{propo charact jumps}}
 We split the proof in three steps. 
\smallbreak \noindent
\textit{Step 1.} Let $x\in \partial  \mathcal{W} $ be such that $x\in I$ for some interval $I \subset \R^{2}$, with $ I \subset \partial \mathcal{W}$ and of the form
$$
I=I_{a,c}:=\{ a + r \eta \,|\, r \in [0,c]\},
$$ 
for some $ a \in \mathbb{R}^{2},$ with $ \eta =V_{x_1}(y)e_1,$ for each $y\in I\setminus \{a\}$. 
Denote by $\mathcal{H}$ the set of all such $x$.
Furthermore, assume that $I$ in the above definition is maximal, in the sense that $a-r\eta \notin \partial \mathcal{W}$, for every $r>0$.

Observe that, since $\partial_\eta V(\cdot) = \eta DV = |V_{x_1}(\cdot)|^2 = 1$, then
\begin{equation}
\label{eq V linear in I}
V(a + r\eta ) = V(a) + r, \quad \text{for each } r\in [0,c].  \end{equation}
We have that 
\begin{equation*}
    \mathcal{H}= \bigcup_{i=1}^\infty \big\{ y \in \partial \mathcal{W} \, | \,  V(y)-V(y -e_1 V_{x_1}(y)/i ) =1/ i  \big\}.
\end{equation*}

Suppose now that $\bar{x} \in \mathcal{H}$. Then there exists $a\in \mathbb{R}^{2}$ and $c>0 $ such that $x\in I_{a,c}$. Let $v^a \in \mathcal{V}$ be an optimal control for $a$.
By \eqref{eq V linear in I}, we find $$
J(\bar{x}; a-\bar{x} +v^a) = J(a;v^a) + |a-\bar x|  =  V(a) +|a- \bar x| = V(\bar x), 
$$
which, by the uniqueness of the optimal control, implies that $\bar{v}_t=a-\bar{x} + v_t^a$, for any $t\geq 0$. This means exactly that the optimally controlled state starting from $\bar{x}$ jumps immediately to $a$.
\smallbreak \noindent
\textit{Step 2.} Let now $\bar{x} \in \overline{\mathcal{W}}$ be generic. We  want to prove that $X^{\bar{x};\bar{v}}$ jumps only at those times $t$ for which $X_{t-}^{\bar{x};\bar{v}} \in \mathcal{H}$. We argue by contradiction, and suppose that 
$$
\mathbb{P}[\, \omega \in \Omega \text{ s.t.\ there exists } t\geq 0 \text{ s.t. }   X_{t-}^{\bar{x};\bar{v}}(\omega) \notin \mathcal{H}\text{ and } |X_t^{\bar{x};\bar{v}}(\omega)-X_{t-}^{\bar{x};\bar{v}}(\omega)| >0 ] > 0.
$$
For each $\varepsilon > 0$, let 
\begin{equation}
\label{eq definition tau epsilon}    
\tau_\varepsilon := \inf \{ t \geq 0 \, | \, X_{t-}^{\bar{x};\bar{v}} \notin \mathcal{H}, \ |X_t^{\bar{x};\bar{v}}-X_{t-}^{\bar{x};\bar{v}}| \geq \varepsilon \}.
\end{equation}
Take $\varepsilon  >0$ small enough such that $\mathbb{P} [\tau_\varepsilon < \infty ] >0$.
Consider a sequence $(\bar{\tau}_k)_{k \in \mathbb{N}}$ of stopping times exhausting the jumps of $X^{\bar{x};\bar{v}}$ (see, e.g., Proposition 2.26 at p.\ 10 in  \cite{karatzas&shreve1998}, for a construction of such a sequence), so that 
\begin{equation}\label{eq sequence stopp times}
\tau_\varepsilon := \inf \{ \bar{\tau}_k \, | \,k \in \mathbb{N},\ X_{\bar{\tau}_k -}^{\bar{x};\bar{v}} \notin \mathcal{H}, \ |X_{\bar{\tau}_k}^{\bar{x};\bar{v}}-X_{\bar{\tau}_k -}^{\bar{x};\bar{v}}| \geq \varepsilon \}.
\end{equation}
Since the jumps of $\bar{v}$ coincides with the jumps of $X^{\bar{x};\bar{v}}$, if $X^{\bar{x};\bar{v}}$ would have an infinite number of jumps of size greater than $\varepsilon$ on some interval $[0,T]$ with $T\in (0,\infty)$, then $\bar{v}$ would not be of bounded variation on the interval $[0,T]$.
Thus $X^{\bar{x};\bar{v}}$ has only a finite number of jumps of size greater than $\varepsilon$ on each interval $[0,T]$.  This reveals that  
$\tau_\varepsilon$ in \eqref{eq sequence stopp times} is actually the minimum of a finite number of stopping times, which implies that $\tau_\varepsilon$ is itself a stopping time. 

Next, {with the notation $d \bar v = \bar \gamma d |\bar v|$}, on $\{\tau_\varepsilon < \infty \}$, we find 
\begin{align}\label{strict inequality}
V( X_{\tau_\varepsilon}^{\bar{x};\bar{v}}) - V( X_{\tau_\varepsilon - }^{\bar{x};\bar{v}}) & = \int_0^1 DV(\tau_\varepsilon, X_{\tau_\varepsilon - }^{\bar{x};\bar{v}} + \lambda ( X_{\tau_\varepsilon}^{\bar{x};\bar{v}} -X_{\tau_\varepsilon - }^{\bar{x};\bar{v}}) )( X_{\tau_\varepsilon}^{\bar{x};\bar{v}} -X_{\tau_\varepsilon - }^{\bar{x};\bar{v}}) d\lambda \\ \notag
& = \int_0^1 V_{x_1}(\tau_\varepsilon, X_{\tau_\varepsilon - }^{\bar{x};\bar{v}} + \lambda ( X_{\tau_\varepsilon}^{\bar{x};\bar{v}} -X_{\tau_\varepsilon - }^{\bar{x};\bar{v}}) ) \bar{\gamma}_{\tau_\varepsilon} ( |\bar{v}|_{\tau_\varepsilon} -|\bar{v}|_{\tau_\varepsilon - }) d\lambda \\ \notag
& > -| X_{\tau_\varepsilon}^{\bar{x};\bar{v}} -X_{\tau_\varepsilon - }^{\bar{x};\bar{v}}|,
\end{align}
where the strict inequality follows from the fact that, by Proposition \ref{prop X lives in W},  $X_{\tau_\varepsilon}^{\bar{x};\bar{v}} \in \overline{\mathcal{W}}$ but $\tau_\varepsilon$ is such that $X_{\tau_\varepsilon -}^{\bar{x};\bar{v}} \notin \mathcal{H}$.  
Recalling that 
 $\tau_\varepsilon$ is a stopping time, define the sequence of stopping times $\tau_k := (\tau_\varepsilon + \frac{1}{k}) \land T$. By the dynamic programming principle (see, e.g., \cite{haussmann.suo1995b}) we have, for each $k$  
\begin{equation}\label{eq dpp}
V(\bar{x}) = \mathbb{E} \bigg[ \int_0^{\tau_k} e^{-\rho t} h( X_t^{\bar{x};\bar{v}} ) dt + \int_{[0,\tau_k)} e^{-\rho t} d|\bar{v}|_t + e^{-\rho\tau_k} V( X_{\tau_k-}^{\bar{x};\bar{v}}) \bigg]. 
\end{equation}
Therefore, taking limits  as $k\to \infty$ in \eqref{eq dpp}, using \eqref{strict inequality}  we find 
\begin{align*}
V(\bar{x}) & = \mathbb{E} \bigg[ \int_0^{\tau_\varepsilon} e^{-\rho t} h( X_t^{\bar{x};\bar{v}} ) dt + \int_{[0,\tau_\varepsilon]} e^{-\rho t} d|\bar{v}|_t + e^{-\rho\tau_\varepsilon} V(X_{\tau_\varepsilon }^{\bar{x};\bar{v}}) \bigg] \\ 
& = \mathbb{E} \bigg[ \int_0^{\tau_\varepsilon} e^{-\rho t} h( X_t^{\bar{x};\bar{v}} ) dt 
+ \int_{[0,\tau_\varepsilon)} e^{-\rho t} d|\bar{v}|_t + e^{-\rho\tau_\varepsilon} | X_{\tau_\varepsilon}^{\bar{x};\bar{v}} -X_{\tau_\varepsilon - }^{\bar{x};\bar{v}}| + e^{-\rho \tau_\varepsilon} V( X_{\tau_\varepsilon }^{\bar{x};\bar{v}}) \bigg] \\
& > \mathbb{E}\bigg[ \int_0^{\tau_\varepsilon}e^{-\rho t} h( X_t^{\bar{x};\bar{v}} ) dt 
+ \int_{[0,\tau_\varepsilon)} e^{-\rho t} d|\bar{v}|_t
+  e^{-\rho \tau_\varepsilon} V( X_{\tau_\varepsilon -}^{\bar{x};\bar{v}}) \bigg] = V(\bar{x}),
\end{align*}
which is a contradiction, hence $X^{\bar{x};\bar{v}}$ jumps only at times $t$ such that $X_{t-}^{\bar{x};\bar{v}} \in \mathcal{H}$.
\smallbreak \noindent
\textit{Step 3.} Suppose now that $X_{t-}^{\bar{x};\bar{v}} \in \mathcal{H}$ for some $t>0$. It remains to prove that, also in this case,  $\mathbb{P}$-a.s.\ the process $X^{\bar{x};\bar{v}}$ jumps at time $t$ to the endpoint of the interval $I$. 
Now, for any $\mathbb{F}$-stopping time $\tau$, $\mathbb{P} \circ ( X_{\tau -}^{\bar{x};\bar{v}})^{-1}$-a.s.\ in $\mathbb{R}^{2}$,  
we have that the control 
\begin{equation}\label{control after time}
    \bar{v}_t^{\tau} := \bar{v}_{\tau + t} -\bar{v}_{\tau-}, \quad t\geq 0,
\end{equation} 
is optimal for the initial condition $X_{\tau-}^{\bar{x};\bar{v}}$ (see Lemma 2.11 and the discussion at p.\ 1616 in \cite{kruk2000}).  
Let now $\tau^1$ be the first time at which the optimally controlled process $X^{\bar{x};\bar{v}}$ enters the set $\mathcal{H}$. Combining \eqref{control after time} together with Step 1, we obtain that $X^{\bar{x};\bar{v}}$ jumps to the endpoint of $I$. By constructing an increasing sequence $\tau_k$ of hitting times of the set $\mathcal{H}$, which exhausts the set in which $X^{\bar{x};\bar{v}} \in \mathcal{H}$, we conclude that $\mathbb{P}$-a.s.\ the process $X^{\bar{x};\bar{v}}$ jumps at time $t$ to the endpoint of the interval $I$.

\smallskip
\textbf{Acknowledgements.} 
Funded by the Deutsche Forschungsgemeinschaft (DFG, German Research Foundation) - Project-ID 317210226 - SFB 1283. 
We are also grateful to two anonymous reviewers for their careful reading and for their precise and useful comments.

\bibliographystyle{siam} 
\bibliography{main.bib}

\begin{thebibliography}{10}

\bibitem{alvarez2001}
{\sc L.~H.~R. Alvarez}, {\em Reward functionals, salvage values, and optimal
  stopping}, Mathematical Methods of Operations Research, 54 (2001),
  pp.~315--337.

\bibitem{Bank05}
{\sc P.~Bank}, {\em Optimal control under a dynamic fuel constraint}, SIAM J.
  Control Optim., 44 (2005), pp.~1529--1541.

\bibitem{bank&elkaroui2004}
{\sc P.~Bank and N.~El~Karoui}, {\em A stochastic representation theorem with
  applications to optimization and obstacle problems}, Ann. Probab., 32 (2004),
  pp.~1030--1067.

\bibitem{Bank&Riedel01}
{\sc P.~Bank and F.~Riedel}, {\em Optimal consumption choice with intertemporal
  substitution}, Ann. Appl. Probab., 11 (2001), pp.~750--788.

\bibitem{Bather&Chernoff67}
{\sc J.~Bather and H.~Chernoff}, {\em Sequential decisions in the control of a
  spaceship}, in Fifth Berkeley Symposium on Mathematical Statistics and
  Probability, vol.~3, 1967, pp.~181--207.

\bibitem{benth&reikvam2004}
{\sc F.~E. Benth and K.~Reikvam}, {\em A connection between singular stochastic
  control and optimal stopping}, Appl. Math. Optim., 49 (2004), pp.~27--41.

\bibitem{boetius2005}
{\sc F.~Boetius}, {\em Bounded variation singular stochastic control and
  {D}ynkin game}, SIAM J. Control Optim., 44 (2005), pp.~1289--1321.

\bibitem{Boetius&Kohlmann98}
{\sc F.~Boetius and M.~Kohlmann}, {\em Connections between optimal stopping and
  singular stochastic control}, Stochastic Process. Appl., 77 (1998),
  pp.~253--281.

\bibitem{boryc&kruk2016}
{\sc M.~Boryc and {\L}.~Kruk}, {\em Characterization of the optimal policy for
  a multidimensional parabolic singular stochastic control problem}, SIAM J.
  Control Optim., 54 (2016), pp.~1657--1677.

\bibitem{buckdahn&cannarsa2010}
{\sc R.~Buckdahn, P.~Cannarsa, and M.~Quincampoix}, {\em Lipschitz continuity
  and semiconcavity properties of the value function of a stochastic control
  problem}, Nonlinear Differential Equations and Applications, 17 (2010),
  pp.~715--728.

\bibitem{budhiraja.ross2006}
{\sc A.~Budhiraja and K.~Ross}, {\em Existence of optimal controls for singular
  control problems with state constraints}, Ann. Appl. Probab., 16 (2006),
  pp.~2235--2255.

\bibitem{burdzy2004}
{\sc K.~Burdzy, Z.-Q. Chen, and J.~Sylvester}, {\em The heat equation and
  reflected {B}rownian motion in time-dependent domains}, Ann. Probab., 32
  (2004), pp.~775--804.

\bibitem{burdzy&ramanan2009}
{\sc K.~Burdzy, W.~Kang, and K.~Ramanan}, {\em The {S}korokhod problem in a
  time-dependent interval}, Stochastic Process. Appl., 119 (2009),
  pp.~428--452.

\bibitem{cannarsa&sinestrani2004}
{\sc P.~Cannarsa and C.~Sinestrari}, {\em Semiconcave Functions,
  Hamilton-Jacobi Equations, and Optimal Control}, vol.~58, Springer Science \&
  Business Media, 2004.

\bibitem{chiarolla&haussmann1998}
{\sc M.~B. Chiarolla and U.~G. Haussmann}, {\em Optimal control of inflation: a
  central bank problem}, SIAM J. Control Optim., 36 (1998), pp.~1099--1132.

\bibitem{chiarolla&haussmann2000}
{\sc M.~B. Chiarolla and U.~G. Haussmann}, {\em Controlling inflation: the
  infinite horizon case}, Appl. Math. Optim., 41 (2000), pp.~25--50.

\bibitem{chow&menaldi&robin1985}
{\sc P.~L. Chow, J.~L. Menaldi, and M.~Robin}, {\em Additive control of
  stochastic linear systems with finite horizon}, SIAM J. Control Optim., 23
  (1985), pp.~858--899.

\bibitem{cohen2021}
{\sc A.~Cohen}, {\em On singular control problems, the time-stretching method,
  and the weak-{M1} topology}, SIAM J. Control Optim., 59 (2021), pp.~50--77.

\bibitem{Cont.Guo.Xu2020}
{\sc R.~Cont, X.~Guo, and R.~Xu}, {\em Interbank lending with benchmark rates:
  {P}areto optima for a class of singular control games}, Mathematical Finance,
   (2021).

\bibitem{costantini1992}
{\sc C.~Costantini}, {\em The {S}korohod oblique reflection problem in domains
  with corners and application to stochastic differential equations}, Probab.
  Theory Related Fields, 91 (1992), pp.~43--70.

\bibitem{costantini2006}
{\sc C.~Costantini, E.~Gobet, and N.~El~Karoui}, {\em Boundary sensitivities
  for diffusion processes in time dependent domains}, Appl. Math. Optim., 54
  (2006), pp.~159--187.

\bibitem{davis&zervos1998}
{\sc M.~H. Davis and M.~Zervos}, {\em A pair of explicitly solvable singular
  stochastic control problems}, Appl. Math. Optim., 38 (1998), pp.~327--352.

\bibitem{angelis&ferrari&moriarty2019}
{\sc T.~De~Angelis, G.~Ferrari, and J.~Moriarty}, {\em A solvable
  two-dimensional degenerate singular stochastic control problem with nonconvex
  costs}, Math. Oper. Res., 44 (2019), pp.~512--531.

\bibitem{dupuis&ishii1993}
{\sc P.~Dupuis and H.~Ishii}, {\em S{D}{E}s with oblique reflection on
  nonsmooth domains}, Ann. Probab., 21 (1993), pp.~554--580.

\bibitem{evans1979}
{\sc L.~C. Evans}, {\em A second order elliptic equation with gradient
  constraint}, Communications in Partial Differential Equations, 4 (1979),
  pp.~555--572.

\bibitem{federico&ferrari&schuhmann2019}
{\sc S.~Federico, G.~Ferrari, and P.~Schuhmann}, {\em A singular stochastic
  control problem with interconnected dynamics}, SIAM J. Control Optim., 58
  (2020), pp.~2821--2853.

\bibitem{federico&ferrari&schuhumann2020b}
\leavevmode\vrule height 2pt depth -1.6pt width 23pt, {\em Singular control of
  the drift of a {B}rownian system}, Appl. Math. Optim.,  (2021), pp.~1--30.

\bibitem{federico.pham2014}
{\sc S.~Federico and H.~Pham}, {\em Characterization of the optimal boundaries
  in reversible investment problems}, SIAM J. Control Optim., 52 (2014),
  pp.~2180--2223.

\bibitem{friedman2010}
{\sc A.~Friedman}, {\em Variational Principles and Free-Boundary Problems},
  Dover Books on Mathematics, Dover Publications, Mineola, NY, 2010.

\bibitem{gilbarg2001}
{\sc D.~Gilbarg and N.~S. Trudinger}, {\em Elliptic Partial Differential
  Equations of Second Order}, vol.~224, Springer Science \& Business Media,
  2001.

\bibitem{guo.pham2005}
{\sc X.~Guo and H.~Pham}, {\em Optimal partially reversible investment with
  entry decision and general production function}, Stochastic Process. Appl.,
  115 (2005), pp.~705--736.

\bibitem{Guo&Tang&Xu18}
{\sc X.~Guo, W.~Tang, and R.~Xu}, {\em A class of stochastic games and moving
  free boundary problems}, arXiv preprint arXiv:1809.03459,  (2018).

\bibitem{guo&tomecek2009}
{\sc X.~Guo and P.~Tomecek}, {\em A class of singular control problems and the
  smooth fit principle}, SIAM J. Control Optim., 47 (2009), pp.~3076--3099.

\bibitem{Guo&Xu18}
{\sc X.~Guo and R.~Xu}, {\em Stochastic games for fuel follower problem: {N}
  versus mean field game}, SIAM J. Control Optim., 57 (2019), pp.~659--692.

\bibitem{Haussmann&Suo95}
{\sc U.~G. Haussmann and W.~Suo}, {\em Singular optimal stochastic controls
  {I}: {E}xistence}, SIAM J. Control Optim., 33 (1995), pp.~916--936.

\bibitem{haussmann.suo1995b}
\leavevmode\vrule height 2pt depth -1.6pt width 23pt, {\em Singular optimal
  stochastic controls {II}: {D}ynamic programming}, SIAM J. Control Optim., 33
  (1995), pp.~937--959.

\bibitem{hynd2013}
{\sc R.~Hynd}, {\em Analysis of {H}amilton-{J}acobi-{B}ellman equations arising
  in stochastic singular control}, ESAIM Control Optim. Calc. Var., 19 (2013),
  pp.~112--128.

\bibitem{ikeda1977}
{\sc N.~Ikeda and S.~Watanabe}, {\em A comparison theorem for solutions of
  stochastic differential equations and its applications}, Osaka Journal of
  Mathematics, 14 (1977), pp.~619--633.

\bibitem{jack&johnson&zervos2008}
{\sc A.~Jack, T.~C. Johnson, and M.~Zervos}, {\em A singular control model with
  application to the goodwill problem}, Stochastic Process. Appl., 118 (2008),
  pp.~2098--2124.

\bibitem{jack&zervos2006}
{\sc A.~Jack and M.~Zervos}, {\em A singular control problem with an expected
  and a pathwise ergodic performance criterion}, International Journal of
  Stochastic Analysis, 2006 (2006), pp.~1--19.

\bibitem{K}
{\sc Y.~M. Kabanov}, {\em Hedging and liquidation under transaction costs in
  currency markets}, Finance Stoch., 3 (1999), pp.~237--248.

\bibitem{karatzas&shreve1998}
{\sc I.~Karatzas and S.~Shreve}, {\em Brownian Motion and Stochastic Calculus},
  vol.~113, Springer, 1998.

\bibitem{Karatzas&Shreve84}
{\sc I.~Karatzas and S.~E. Shreve}, {\em Connections between optimal stopping
  and singular stochastic control {I}. {M}onotone follower problems}, SIAM J.
  Control Optim., 22 (1984), pp.~856--877.

\bibitem{karatzas&wang2001}
{\sc I.~Karatzas and H.~Wang}, {\em Connections between bounded variation
  control and dynkin games}, Optimal Control and Partial Differential Equations
  (Volume in honor of A. Bensoussan),  (2001), pp.~363--373.

\bibitem{koch&vargiolu2019}
{\sc T.~Koch and T.~Vargiolu}, {\em Optimal installation of solar panels with
  price impact: a solvable singular stochastic control problem}, SIAM J.
  Control Optim., 59 (2021), pp.~3068--3095.

\bibitem{kruk2000}
{\sc L.~Kruk}, {\em Optimal policies for n-dimensional singular stochastic
  control problems part {I}: The {S}korokhod problem}, SIAM J. Control Optim.,
  38 (2000), pp.~1603--1622.

\bibitem{krylov1980}
{\sc N.~V. Krylov}, {\em Controlled Diffusion Processes}, vol.~14, Springer
  Science \& Business Media, 2008.

\bibitem{li&zitkovic2017}
{\sc J.~Li and G.~\v{Z}itkovi\'c}, {\em Existence, characterization, and
  approximation in the generalized monotone-follower problem}, SIAM J. Control
  Optim., 55 (2017), pp.~94--118.

\bibitem{lions&sznitman1984}
{\sc P.~L. Lions and A.~S. Sznitman}, {\em Stochastic differential equations
  with reflecting boundary conditions}, Communications on Pure and Applied
  Mathematics, 37 (1984), pp.~511--537.

\bibitem{lundstrom&onskog2019}
{\sc N.~L. Lundstr{\"o}m and T.~{\"O}nskog}, {\em Stochastic and partial
  differential equations on non-smooth time-dependent domains}, Stochastic
  Process. Appl., 129 (2019), pp.~1097--1131.

\bibitem{ma1992}
{\sc J.~Ma}, {\em On the principle of smooth fit for a class of singular
  stochastic control problems for diffusions}, SIAM J. Control Optim., 30
  (1992), pp.~975--999.

\bibitem{mao2008}
{\sc X.~Mao}, {\em Stochastic Differential Equations and Applications},
  Elsevier, 2007.

\bibitem{Menaldi&Taksar89}
{\sc J.~L. Menaldi and M.~I. Taksar}, {\em Optimal correction problem of a
  multidimensional stochastic system}, Automatica J. IFAC, 25 (1989),
  pp.~223--232.

\bibitem{nystrom&onskog2010}
{\sc K.~Nystr{\"o}m and T.~{\"O}nskog}, {\em The {S}korohod oblique reflection
  problem in time-dependent domains}, Ann. Probab., 38 (2010), pp.~2170--2223.

\bibitem{Protter05}
{\sc P.~E. Protter}, {\em Stochastic Integration and Differential Equations},
  Springer, 2nd~ed., 2005.

\bibitem{saisho1987}
{\sc Y.~Saisho}, {\em Stochastic differential equations for multi-dimensional
  domain with reflecting boundary}, Probab. Theory Related Fields, 74 (1987),
  pp.~455--477.

\bibitem{Schuhmann2021}
{\sc P.~Schuhmann}, {\em On some {T}wo-{D}imensional {S}ingular {S}tochastic
  {C}ontrol {P}roblems and their {F}ree-{B}oundary {A}nalysis}, PhD thesis,
  Bielefeld University, 2021.

\bibitem{shreve.soner.1990elliptic}
{\sc S.~E. Shreve and H.~M. Soner}, {\em A free boundary problem related to
  singular stochastic control}, in Applied Stochastic Analysis, Davis MHA,
  Elliot RJ, eds., Stochastic Monographs, Vol. 5, Gordon and Breach Science
  Publishers, New York, 1990, p.~876–907.

\bibitem{soner&shreve1989}
{\sc H.~M. Soner and S.~E. Shreve}, {\em Regularity of the value function for a
  two-dimensional singular stochastic control problem}, SIAM J. Control Optim.,
  27 (1989), pp.~876--907.

\bibitem{soner.shreve.elkaroui.1991parabolic}
{\sc H.~M. Soner and S.~E. Shreve}, {\em A free boundary problem related to
  singular stochastic control: the parabolic case}, Comm. Partial Differential
  Equations, 16 (1991), pp.~373--424.

\bibitem{suo1994}
{\sc W.~Suo}, {\em The {E}xistence of {O}ptimal {S}ingular {C}ontrols for
  {S}tochastic {D}ifferential {E}quations}, PhD thesis, University of British
  Columbia, 1994.

\bibitem{taksar1992}
{\sc M.~Taksar}, {\em {S}korohod problems with nonsmooth boundary conditions},
  Journal of Computational and Applied Mathematics, 40 (1992), pp.~233--251.

\bibitem{tanaka1979}
{\sc H.~Tanaka}, {\em Stochastic differential equations with reflecting
  boundary condition in convex regions}, Hiroshima Mathematical Journal, 9
  (1979), pp.~163--177.

\bibitem{weerasinghe2005}
{\sc A.~Weerasinghe}, {\em A bounded variation control problem for diffusion
  processes}, SIAM J. Control Optim., 44 (2005), pp.~389--417.

\bibitem{williams.chow.menaldi.1994}
{\sc S.~A. Williams, P.-L. Chow, and J.-L. Menaldi}, {\em Regularity of the
  free-boundary in singular stochastic control}, J. Differential Equations, 111
  (1994), pp.~175--201.

\bibitem{yang2014}
{\sc Y.~Yang}, {\em A multidimensional stochastic singular control problem via
  {D}ynkin game and {D}irichlet form}, SIAM J. Control Optim., 52 (2014),
  pp.~3807--3832.

\end{thebibliography}
\end{document}